\documentclass[11pt]{amsart}
\usepackage{amssymb,mathrsfs,graphicx,enumerate,color}
\usepackage{hyperref}
\usepackage{comment}
\usepackage{colortbl}
\definecolor{black}{rgb}{0.0, 0.0, 0.0}
\definecolor{red}{rgb}{1.0, 0.5, 0.5}

\title[Stability of inflow problem of the Navier--Stokes--Fourier system]{Long-time dynamics toward a generic composite wave for the inflow problem of the Navier--Stokes--Fourier system}

\author[Huang]{Xushan Huang}
\address[Xushan Huang]{\newline Department of Mathematics
\newline Yonsei University, Seoul 03722, Republic of Korea}
\email{xushanhuang@yonsei.ac.kr}

\author[Kang]{Moon-Jin Kang}
\address[Moon-Jin Kang]{\newline Department of Mathematical Sciences \newline Korea Advanced Institute of Science and Technology, Daejeon  34141, Republic of Korea}
\email{moonjinkang@kaist.ac.kr}

\author[Lee]{Hobin Lee}
\address[Hobin Lee]{\newline Department of Mathematical Sciences \newline Korea Advanced Institute of Science and Technology, Daejeon  34141, Republic of Korea}
\email{lcuh11@kaist.ac.kr}

\author[Oh]{HyeonSeop Oh}
\address[HyeonSeop Oh]{\newline Department of Mathematical Sciences \newline Korea Advanced Institute of Science and Technology, Daejeon  34141, Republic of Korea}
\email{ohs2509@kaist.ac.kr}

\newtheorem{proposition}{Proposition}[section]
\newtheorem{remark}{Remark}[section]

\newcommand{\bbr}{\mathbb R}

\newcommand{\e}{\varepsilon}

\numberwithin{figure}{section}
\newcommand{\beq}{\begin{equation}}
\newcommand{\eeq}{\end{equation}}
\newcommand{\bsp}{\begin{split}}
\newcommand{\esp}{\end{split}}

\newcommand{\R}{\mathbb{R}}

\newcommand{\s}{\sigma}

\def\eps{\varepsilon }

\newcommand\adots{\mathinner{\mkern2mu\raise1pt\hbox{.}
\mkern3mu\raise4pt\hbox{.}\mkern1mu\raise7pt\hbox{.}}}

\newtheorem{theo}{Theorem}[section]

\newtheorem{lem}[theo]{Lemma}

\newtheorem{rem}[theo]{Remark}

\numberwithin{equation}{section}

\def\charf {\mbox{{\text 1}\kern-.30em {\text l}}}

\def \l {\lambda}

\def \m {\mu}
\def \rd {\partial}

\def \d {\delta}

\def \r {\rho}

\def \b {\beta}
\def \g {\gamma}
\def \Rp {\mathbb{R}_+}
\def \intRp {\int_{\mathbb{R}_+}}
\def \x {\xi}
\def \th {\theta}
\newcommand{\norm}[1]{\left\|#1\right\|}

\def \Util {\widetilde{U}}

\def \Ubar {\overline{U}}
\def \vbar {\bar{v}}
\def \ubar {\bar{u}}
\def \thbar {\bar{\theta}}
\def \thtil {\widetilde{\theta}}
\def \vtil {\widetilde{v}}
\def \pbar {\bar{p}}
\def \Th {\Theta}
\def \k {\kappa}
\def \t {\tau}

\begin{document}
\bibliographystyle{acm}

\date{\today}


\thanks{\textbf{Acknowledgment.} X. Huang was supported by the National Research Foundation of Korea (NRF) grant funded by the Korea government (MSIT) (No. RS-2019-NR040050 and RS-2024-00406821). M.-J. Kang, H. Lee and H. Oh were partially supported by the National Research Foundation of Korea (RS-2024-00361663 and NRF-2019R1A5A1028324).
}

\begin{abstract}
We study the time-asymptotic stability of solutions to the inflow problem for the one-dimensional Navier--Stokes--Fourier system on the half-line. We consider the most generic wave pattern: the superposition of a degenerate boundary layer, a rarefaction, a viscous contact wave, and a viscous shock. More precisely, if the boundary data belongs to the subsonic region, and the initial perturbation and strengths of the boundary layer, viscous contact wave, and viscous shock are sufficiently small, then the solution to the inflow problem converges to the corresponding superposition, up to a time-dependent shift for a shock. The rarefaction wave, however, is allowed to have arbitrarily large strength. To control the viscous shock, we employ the method of $a$-contraction with shifts. A notable feature of our analysis is that this method can be applied even when the rarefaction wave has large amplitude. In particular, this resolves, in a generic setting, the open problem of the stability of inflow wave patterns containing a viscous shock for Navier--Stokes--Fourier system.
\end{abstract}

\maketitle

\centerline{\date}

\tableofcontents

\section{Introduction}
\setcounter{equation}{0}
We consider the inflow problem of the one-dimensional Navier--Stokes--Fourier (NSF) system on the half-line $\mathbb{R}_+ := (0,\infty)$: 
\begin{equation} \label{eq:NSF1}
	\begin{cases}
		& \rho_t + (\rho u)_x = 0, \quad x\geq 0,\quad t\geq 0,\\
		&(\rho u)_t + (\rho u^2 + p)_x = \mu u_{xx}, \\
		&\mathcal{E}_t + \left(u\mathcal{E}+ pu\right)_x = \kappa \theta_{xx} + \mu(uu_x)_x\\
		&(\rho,u,\theta)(t,0) = (\rho_-, u_-, \theta_-) \quad \text{with} \quad \rho_-, u_-, \theta_- >0,
	\end{cases}
\end{equation}
subject to the initial data

\[(\rho_0, u_0,\theta_0)(x)=(\rho,u,\theta)(0,x), \quad \text{with} \quad 
		\inf_{x\in \mathbb{R}_+} \rho_0(x) > 0, \quad \inf_{x\in \mathbb{R}_+} \theta_0(x) > 0, 
\]
and the far-field condition
\begin{equation*}
	\left(\rho, u, \theta\right)(0,x) \to(\rho_+,u_+, \theta_+)\quad\mbox{as}\quad {x}\to\infty.
\end{equation*}
Here, $\rho =\rho(t,x), u= u(t,x)$ and $ \theta = \theta(t,x) $ represent the fluid density, velocity and absolute temperature respectively, $\mathcal{E}=\rho (e+ \frac{u^2}{2})$ is the total energy function. For the ideal polytropic gas, the pressure function $p$ and the internal energy function $e$ are given by 
\begin{equation*}
p(\rho, \theta) = R\rho \theta = A\r^{\gamma} \exp\left(\frac{\g - 1}{R}s \right),\qquad e(\rho, \theta) = \frac{R}{\gamma -1} \theta + const,
\end{equation*}
where $s$ is the (physical) entropy, $A>0,R>0$, $\gamma>1$ being both constant related to the fluid, while $\mu$ and $\kappa$ denote the viscosity and the heat-conductivity.

In Lagrangian mass coordinates, the inflow problem \eqref{eq:NSF1} can be rewritten as
\begin{equation} \label{eq:NSFL}
	\begin{cases}
		& v_t - u_x = 0, \quad t>0, \quad x > \sigma_- t, \\
		& u_t + p(v,\theta)_x = \left(\mu \frac{u_x}{v}\right)_x,\\
		& E_t + (p(v,\theta)u)_x = \left(\kappa \frac{\theta_x}{v}\right)_x + \left(\mu \frac{u u_x}{v}\right)_x,\\
		& (v,u,\theta)(t,\sigma_- t) = (v_-, u_-,\th_-), \quad v_-, u_-,\th_- > 0,
	\end{cases}
\end{equation}
where $v(t,x) = 1/\rho$ represents the specific volume, $v_- = \frac{1}{\rho_-}, \sigma_- = -\frac{u_-}{v_-}<0, p(v,\theta) = \frac{R\theta}{v}$, and $E = e + \frac{u^2}{2}$ denotes the specific total energy. Moreover, the far-field condition is reformulated by 
\[
(v,u,\theta)(0,x) \to (v_+, u_+, \theta_+) \quad\mbox{as}\quad {x}\to\infty.
\]
We are interested in the long-time dynamics of solutions $U = (v,u,\th)$ to the inflow problem \eqref{eq:NSFL}. It is well known that these asymptotic patterns for the NSF system are closely related to the Riemann problem for the full Euler system (the NSF system with $\mu=\kappa=0$). For the Cauchy problem on the whole line $\mathbb{R}$, one of the generic asymptotic wave patterns consists of a rarefaction wave, a viscous contact wave, a viscous shock, or their superpositions. However, for the inflow problem, an additional asymptotic pattern, namely, a boundary layer solution (stationary solution), appears due to the presence of the boundary. Indeed, the asymptotic wave patterns for the inflow problem are determined by the locations of the boundary state $U_-$ and the far-field state $U_+$. To describe these wave patterns, we define the sound speed $c(\th)=\sqrt{R \gamma \th}$ and decompose the state space into the subsonic, transonic, and supersonic regions:
\begin{align*}
    \Omega_{sub} &:= \{(v,u,\th)\,:\,|u| < c(\th) \},\\
    \Gamma_{trans} &:=\{(v,u,\th)\,:\, |u| = c(\th) \},\\
    \Omega_{super} &:=\{(v,u,\th)\,:\, |u| > c(\th) \}.
\end{align*}
Since we consider the inflow problem, we restrict our attention to the regions with positive velocity and define $\Omega_{sub}^+, \Gamma_{trans}^+$, and $\Omega_{super}^+$ by 
\[
\Omega_{sub}^+:= \Omega_{sub} \cap \{u>0\}, \quad \Gamma_{trans}^+:= \Gamma_{trans} \cap \{u>0\}, \quad \Omega_{super}^+ = \Omega_{super} \cap \{u>0\}.
\]
We consider the case where $U_- \in \Omega_{sub}^+$. In particular, we consider the case where the asymptotic behavior of solutions to the inflow problem \eqref{eq:NSFL} is given as a superposition of a degenerate boundary layer solution, a rarefaction wave, a viscous contact wave, and a viscous shock wave, which are described below.

Before describing these waves, we first introduce the change of variables $(t,\xi):=(t,x - \sigma_- t)$ to fix the boundary. Then, the system \eqref{eq:NSFL} can be written in terms of $(t,\xi)$ as
\begin{equation} \label{eq:NSF}
	\begin{cases}
		& v_t - \s_- v_\x - u_\x = 0, \quad t>0, \quad \x >0, \\
		& u_t - \s_- u_\x + p(v,\theta)_\x = \left(\mu \frac{u_\x}{v}\right)_\x,\\
		& E_t - \s_- E_\x + (p(v,\theta)u)_\x = \left(\kappa \frac{\theta_\x}{v}\right)_\x + \left(\mu \frac{u u_\x}{v}\right)_\x,\\
		& (v,u,\theta)(t,0) = (v_-, u_-,\theta_-), \quad v_-, u_-, \theta_- > 0,
	\end{cases}
\end{equation}
subject to the initial data and the far-field condition
\begin{equation}\label{IC}
	(v,u,\theta)(0,\x) = (v_0, u_0, \th_0)(\x), \quad \lim_{\x \to \infty}(v,u,\th)(0,\x) = (v_+, u_+, \th_+).
\end{equation}\vspace{3mm}\\
$\bullet$ {\bf Boundary layer solution:} For a given $U_* \in \Gamma_{trans}^+$, \cite{WYY2025} proved that there exists a curve $\Sigma \subset \{(u,\theta)\, :\, u>0, \, \, \theta>0\}$, such that, for any $U_- \in BL(U_*)$, there exists a unique (degenerate) boundary layer solution (stationary solution) $U^{BL}=U^{BL}(\xi)$ to \eqref{eq:NSF} satisfying
\begin{equation}\label{eq:BL}
	\begin{cases}
		&- \s_- v^{BL}_\x - u^{BL}_\x = 0, \quad \x >0, \\
		&- \s_- u^{BL}_\x + p(v^{BL},\theta^{BL})_\x = \left(\mu \frac{u^{BL}_\x}{v^{BL}}\right)_\x,\\
		&- \s_- \left(\frac{R}{\gamma-1} \th^{BL} + \frac{(u^{BL})^2}{2}\right)_\x +  \big(p(v^{BL},\theta^{BL}) u^{BL}\big)_\x = \left(\kappa \frac{\theta^{BL}_\x}{v^{BL}}\right)_\x + \left(\mu \frac{u^{BL} u^{BL}_\x}{v^{BL}}\right)_\x,\\
		&(v^{BL},u^{BL},\th^{BL})(0) = (v_-, u_-, \th_-), \quad (v^{BL},u^{BL},\th^{BL})(+\infty) = (v_*,u_*,\th_*).
	\end{cases}
\end{equation}
Here, the boundary layer curve $BL(U_*)$ is defined by 
\[
BL(U_*):= \left\{(v,u,\th) \,:\, \frac{u}{v} = \frac{u_*}{v_*}=-\sigma_-, \quad (u,\th) \in \Sigma \right\} \subset \Omega_{sub}^+.
\]
$\bullet$ {\bf 1-rarefaction wave and approximate 1-rarefaction wave:} We consider the situation in which the state $U_*$ and the far-field state $U_+$ are connected by a 1-rarefaction wave, contact discontinuity, and 3-shock wave, with intermediate states $U_m$ and $U^*$. First, we consider $U_m \in R_1(U_*)$, where $R_1(U_*)$ is a 1-rarefaction curve. Namely, $R_1(U_*)$ the integral curve of the first characteristic field:
\begin{align*}
&R_1(U_*):=\bigg\{(v,u,\th)\,:\,v>v_*, \, \,  s(v,\th)=s(v_*,\th_*):=s(v_m,\th_m), \\
&\phantom{R_1(U_*):=\bigg\{(v,u,\th)\,:\,\,}u=u_*-\int_{v_*}^{v} \lambda_1(s_*,\eta)\,d\eta \bigg\},
\end{align*}
where 
\begin{equation}\label{def:s}
    s(v,\theta)=\frac{R}{\gamma-1}\ln\left(\frac{R}{A}\th v^{\gamma-1}\right)
\end{equation}
 denotes the entropy, and $\lambda_1(s,v)=-\sqrt{\frac{\gamma p(v,s)}{v}}$ is the first eigenvalue of the Jacobian of the hyperbolic part of \eqref{eq:NSFL}.

Then, for any $U_m \in R_1(U_*)$, the solution to the Riemann problem of the Euler system 
\begin{equation}\label{eq:E}
\begin{cases}
		& v_t - u_x = 0, \quad (t,x) \in \Rp \times \mathbb{R}, \\
		& u_t + p(v,\theta)_x = 0,\\
		& E_t + (p(v,\theta)u)_x = 0,\\
\end{cases} 
\end{equation}
with the initial data 
\[
(v,u,\th)(0,x) = \begin{cases}
		    (v_*,u_*,\th_*), \, \, & \text{if } x<0,\\
             (v_m,u_m,\th_m), \, \, & \text{if } x>0,\\
		\end{cases}
\]
is given by a 1-rarefaction wave $U^r(t,x)$, defined by
\[
\lambda_1(v^r(t,x),\th^r(t,x)) = \begin{cases}
    \lambda_1(v_*,\th_*), \, \, &\text{if } x<\lambda_1(v_*,\th_*)t,\\
    \frac{x}{t}, \, \, &\text{if } \lambda_1(v_*,\th_*)t<x<\lambda_1(v_m,\th_m)t,\\
    \lambda_1(v_m,\th_m), \, \, &\text{if } x>\lambda_1(v_m,\th_m)t,
\end{cases}
\]
together with 
\[
\begin{aligned}
    &Z_1(v^r(t,x),u^r(t,x);s^r) = Z_1(v_*,u_*;s^r) = Z_1(v_m,u_m;s^r),\\
    &s(v^r(t,x),\th^r(t,x)) = s(v_*,\th_*) = s(v_m,\th_m)=:s^r,
\end{aligned}
\]
where $Z_1(v,u;s):= u + \int^v \lambda_1(s,v')dv'$, and $s$, defined in \eqref{def:s}, are Riemann invariants associated with the first characteristic field.

Due to the lack of regularity of the rarefaction wave $U^r$ at the rarefaction fronts, we introduce the following smooth approximation of $U^r$ as in \cite{HMS2003}.

First, we define the smooth solution $w(t,x) = w(t,x;w_-,w_+)$ of the Burgers equation
\begin{eqnarray}\label{eq:SBurgers}
	\left\{
	\begin{array}{l}
		 w_{t}+ww_{x}=0,\\
	     w( 0,x )=w_0(x)=\left\{
		\begin{array}{l}
			 w_-,\qquad \qquad\qquad\qquad \qquad \,  x<
			0 ,\\
			  w_-+C_q \d_r \int^{ \e_r x }_0y^qe^{-y}\,dy, \quad x\geq 0 ,
		\end{array}
		\right.
	\end{array}
	\right.
\end{eqnarray}
where $\d_r=w_+-w_-$, $q$ is a constant satisfying $q\geq 8$ to be specified later, and $C_q$ is defined by $$C_q\int^{\infty}_0y^qe^{-y} dy=1 .$$
Here, $\e_r<1$ is a sufficiently small constant to be chosen later.

We now construct the smooth approximate 1-rarefaction wave $U^{AR}(t,x)$ corresponding to the 1-rarefaction wave $U^r(x/t)$ by 
\begin{equation}\label{def:APR}
    \begin{cases}
    &\l_1(v^{AR},\theta^{AR})(t,x)= w(1+t,x;\l_1((v_*,\th_*)),\l_1((v_m,\th_m))), \\
    &Z_1(v^{AR}(t,x),u^{AR}(t,x);s^r) = Z_1(v_*,u_*;s^r) = Z_1(v_m,u_m;s^r),\\
    &s(v^{AR}(t,x),\th^{AR}(t,x)) = s(v_*,\th_*) = s(v_m,\th_m) = s^r.
    \end{cases}
\end{equation}
Then, one can verify that the approximate 1-rarefaction wave $U^{AR}$ satisfies the Euler system \eqref{eq:E}. Moreover, $U^{AR}$ indeed approximates the self-similar rarefaction wave $U^r$ in the sense that (see Lemma~\ref{lem:rarefaction}):
\[
\sup_{x \in \mathbb{R}} \left|(v^{AR},u^{AR},\th^{AR})(t,x) - (v^r,u^r,\th^r)\left(\frac{x}{1+t} \right) \right| \to 0 \quad \text{as} \quad t \to \infty.
\]
We then define a approximate 1-rarefaction wave $U^R$ on the half-line $\Rp$ as the restriction of $U^{AR}$:
\begin{equation}\label{def:SR}
(v^R,u^R,\th^R)(t,\x) := (v^{AR},u^{AR},\th^{AR})(t,\x + \s_- t).
\end{equation}
$\bullet$ {\bf Viscous contact wave:} Next, we consider $U^* \in CD(U_m)$, where $CD(U_m)$ denotes the 2-contact discontinuity curve defined by 
\begin{equation}\label{def:CDc}
CD(U_m):= \big\{(v,u,\th)\,:\, u = u_m, \, \, p(v,\th) = p(v_m,\th_m) \big\}.
\end{equation}
Since the second characteristic field of the Euler system \eqref{eq:E} is linearly degenerate, the corresponding Riemann solution for the Euler system \eqref{eq:E} with the initial data 
\[
(v,u,\th)(0,x) = \begin{cases}
    (v_m,u_m,\th_m), \quad &\text{if } x<0,\\
    (v^*,u^*,\th^*), \quad &\text{if } x>0,
\end{cases}
\]
is given by a contact discontinuity $U^c$ defined by 
\[
(v^c,u^c,\th^c)(t,x) =  \begin{cases}
    (v_m,u_m,\th_m), \quad &\text{if } x<0,\\
    (v^*,u^*,\th^*), \quad &\text{if } x>0.
\end{cases} 
\]
In the viscous setting, for any $U^* \in CD(U_m)$, we define a viscous contact wave $U^{VC}(t,x)$ on $\mathbb{R}$ associated with the contact discontinuity $U^c$ by
\begin{equation}\label{eq:VCD}
	\begin{cases}
		v^{VC}\left(t,x\right) = \frac{R\Th^{sim}}{p^*},\\
		u^{VC} \left(t,x\right) = u^* + \frac{(\gamma - 1) \kappa \Th_x^{sim}}{R\gamma \Th^{sim}},\\
		\th^{VC}\left(t,x\right) = \Th^{sim},
	\end{cases}
\end{equation}
where $p^* := p(v^*, \th^*)$ and $\Th^{sim} = \Th^{sim}\left(\frac{x}{\sqrt{1+t}}\right)$ is the unique self-similar solution to the nonlinear diffusion equation:
\begin{equation*}
	\begin{cases}
		\Th_t = \frac{(\gamma-1)\kappa p^*}{R^2 \gamma}\left(\frac{\Th_x}{\Th}\right)_x, \quad t>0, x \in \mathbb{R},\\
		\Th(t,-\infty) = \th_m, \quad \Th(t,\infty) = \th^*.
	\end{cases}
\end{equation*}
We refer to \cite{HLM10} for the derivation of the viscous contact wave $U^{VC}$. We also remark that, for any $p \ge 1$ (see \eqref{est:contact}$_1$ and \eqref{est:contact}$_2$),
\[ 
\norm{(v^C,u^C, \th^C)(t,\cdot) - (v^c,u^c,\th^c)(t,\cdot)}_{L^p(\mathbb{R})} = O(1)\kappa^{\frac{1}{2p}}(1+t)^{\frac{1}{2p}}, \quad \forall t>0.
\]
Consequently, the viscous contact wave converges to the inviscid contact discontinuity in $L^p$ as $\kappa \to 0+$ on any finite time interval. However, the two profiles may not remain close for large time.

Then, we define a viscous contact wave $U^C$ on the half-line $\Rp$ as the restriction of $U^{VC}$:
\begin{equation}\label{def:VC}
(v^C, u^C, \th^C)(t,\xi) := (v^{VC},u^{VC},\th^{VC})(t,\x + \s_- t).    
\end{equation}

$\bullet$ {\bf Viscous 3-shock wave:} Finally, we consider $U_+ \in S_3(U^*)$, where $S_3(U^*)$ denotes the Hugoniot locus associated with the third characteristic field. More precisely, $S_3(U^*)$ consists of all states $U = (v,u,\th)$ for which there exists a shock speed $\sigma$ satisfying Rankine Hugoniot condition and Lax entropy condition
\[
\begin{aligned}
     &-\sigma (v - v^*)-(u - u^*)=0,\\
     &-\sigma(u - u^*)+(p(v,\th) - p(v^*,\th^*)) = 0,\\
     &-\sigma(E- E^*)+ (p(v,\th)u - p(v^*,\th^*)u^*)=0,
      \end{aligned}
       \qquad v^* > v, \quad \sigma = \sqrt{-\frac{p(v,\th) - p(v^*,\th^*)}{v-v^*}}.
\]
For any $U_+ \in S_3(U^*)$, a viscous 3-shock wave $U^{S}$, which is the viscous counterpart of the inviscid shock wave, is a traveling wave $U^{S}(\zeta) = U^{S}(x - \sigma t)$ of \eqref{eq:NSFL} satisfying 
\begin{equation}\label{eq:VS}
	\begin{cases}
		&-\s (v^S)' - (u^S)' = 0, \quad ' = \frac{d}{d\zeta},\\
		&- \s (u^S)' + p(v^S,\theta^S)' = \left(\mu \frac{(u^S)'}{v^S}\right)',\\
		&- \s (E^S)' + (p(v^S,\theta^S)u^S)' = \left(\kappa \frac{(\theta^S)'}{v^S}\right)' + \left(\mu \frac{u^S (u^S)'}{v^S}\right)', \\
        & (v^S, u^S, \th^S)(-\infty)=(v^*,u^*, \theta^*), \quad (v^S, u^S, \th^S)(+\infty)=(v_+,u_+, \theta_+),
	\end{cases}
\end{equation}
where $E^S:= \frac{R}{\gamma -1}\th^S + \frac{(u^S)^2}{2}$, and the shock speed $\sigma$ is given by 
\[
\sigma =  \sqrt{-\frac{p(v_+,\th_+) - p(v^*,\th^*)}{v_+-v^*}}.
\]

Suppose that the boundary state $U_- \in \Omega^+_{sub}$ and the far-field state $U_+$ are connected by a boundary layer curve, a 1-rarefaction curve, a contact discontinuity curve, and a 3-Hugoniot locus, i.e., there exist intermediate states $U_* \in \Gamma_{trans}^+$, $U_m$, and $U^*$ such that 
\[
U_- \in BL(U_*), \quad U_m \in R_1(U_*), \quad U^* \in CD(U_m), \quad U^* \in S_3(U_+).
\]
Then, the asymptotic state is expected to be a superposition of a boundary layer, a 1-rarefaction, a viscous contact wave and a viscous 3-shock wave. The goal of this paper is to establish the stability of this superposition.
\subsection{Main Result} Now, we state the main result for the inflow problem \eqref{eq:NSFL}.
\begin{theo}\label{thm:main}
   Suppose that $1<\gamma \leq 2$. For any $U_* \in \Gamma_{trans}^+$, fix a state $U_m \in \Omega_{super}^+$ with $U_m \in R_1(U_*)$. Then there exist positive constants $\d_0, \e_0$ such that the following holds.

   For any three states $U_-, U^*, U_+$ with $U_- \in BL(U_*), U^* \in CD(U_m)$, and $U_+ \in S_3(U^*)$ satisfying 
   \[
    |U_- - U_*| + |U_m - U^*| + |U^* - U_+| < \d_0,
   \]
   there exist a sufficiently large constant $\beta>0$ $($depending only on the shock strength $|U^* - U_+|)$ and a sufficiently small constant $0<\e_r<1$ $($depending on $\d_0$ and the rarefaction strength $|U_m - U_*|)$ such that the following holds.

   Let $U^R$ be the smooth rarefaction wave with the parameter $\e_r$ defined in \eqref{eq:SBurgers}. If any initial data $U_0$ satisfies
   \begin{equation}\label{inismall}
   \begin{aligned}
       &\|U_0-U^* - (U^R(0,\cdot)-U_m)\|_{L^2(0,\beta)} +  \|U_0 - U_+ - (U^R(0,\cdot) - U_m)\|_{L^2(\beta,\infty)}\\
   &\qquad + \|(U_0 - U^R(0,\cdot))_x\|_{L^2(\Rp)} < \e_0,
   \end{aligned}
   \end{equation}
   then the inflow problem \eqref{eq:NSFL} subject to the initial data $U_0$ admits a global-in-time solution $U(t,x)$. Moreover, there exists a Lipschitz continuous shift function $X(t)$ such that 
   \begin{equation}\label{GE}
   \begin{aligned}
    &(v-\vbar, u-\ubar, \th -\thbar) \in C(0,\infty;H^1(\s_- t, \infty)),\\
	&(u-\ubar)_{xx} \in L^2(0,\infty;L^2(\s_- t, \infty)), \quad (\th - \thbar)_{xx} \in L^2(0,\infty;L^2(\s_- t, \infty)),
   \end{aligned}
   \end{equation}
where $\Ubar(t,x)$ denotes the superposition of a boundary layer, a smooth 1-rarefaction, a contact wave and a viscous 3-shock wave:
\begin{equation}\label{def:Ubar}
    \Ubar(t,x) = U^{BL}(x-\s_- t) + U^R(t,x) + U^C(t,x) + U^S(x - \s t - X(t)-\beta) - U_* - U_m - U^*.
\end{equation}
Furthermore, the solution asymptotically converges to the superposition of a BL solution, an inviscid 1-rarefaction, a viscous contact wave, and a viscous 3-shock wave
\begin{equation}\label{Asym}
\begin{aligned}
    &\sup_{x \in (\s_-t,\infty)}\Big|U(t,x) - \big(U^{BL}(x-\s_- t) + U^r(t,x) + U^C(t,x) \\
    &\phantom{\sup_{x \in (\s_-t,\infty)}\Big|U(t,x) - \big(U^{BL}(x-\s_- t) }+ U^S(x- \s t - X(t) - \b) - U_* - U_m - U^*\big) \Big| \to 0,
\end{aligned}
\end{equation}
as $t \to \infty$, and the speed of the shift $X(t)$ converges to zero:
\begin{equation}\label{Xasym}
    |\dot{X}(t)| \to 0 \, \, \text{as } t \to \infty.
\end{equation}
\end{theo}

\begin{rem}
\begin{enumerate}
\item Theorem~\ref{thm:main} requires the strengths of BL solution, viscous contact wave, and viscous shock wave to be small and bounded by the constant $\delta_0$, which does not depend on $\eps_0$ the smallness of the initial perturbation. However, the rarefaction strength need not be small. Unlike the barotropic Navier--Stokes system as in \cite{HKKKO25,MatBVP}, little is known about explicit formula of the boundary layer curve. So, it is natural to define the boundary layer curve starting from the given right state $U_*$ via the center-stable manifold.  

\item The assumption  $1<\gamma\leq 2$ is physically natural. For an ideal gas, the adiabatic exponent is determined by the number of effective molecular degrees of freedom; in particular, ordinary gases have at least three translational degrees of freedom, leading to  $\gamma\le 5/3< 2$. More generally, additional rotational or vibrational modes only decrease $\gamma$, so the range $1<\gamma\leq 2$ covers the physically relevant regimes.
Technically, the restriction $1<\gamma\leq 2$ is only required in the proof for the monotonicity \eqref{est:Dth1} of $v^R\theta^R$ that is used to get a sharp estimate for the diffusion term  despite the large strength of a rarefaction wave.

\item We first shift the viscous shock profile by a sufficiently large constant $\beta>0$, which is sufficiently larger than the reciprocal of shock strength. That is, we consider the initial location of shock front sufficiently far from the boundary by $O(\beta)$. This allows us to control the discrepancy between the prescribed boundary data and the boundary values of the viscous shock profile. This, in turn, guarantees the stability of the viscous shock  under sufficiently small perturbations.
  
    \item 
   The $L^2$-norm of the initial perturbation is measured relative to the shift $\beta$, since the viscous shock  is located far from the origin by $\beta$-scale. Moreover, since the strength of a rarefaction wave is  large, while the initial perturbation is small,  the approximate rarefaction $U^R$ is included in all the norms appearing in the initial perturbation.

   \item  Since the speed of shift $X(t)$ vanishes in the limit as $t \to \infty$, the shifted shock wave would time-asymptotically belong to one-parameter family of shock waves. However, we do not ensure the final destination of shift, since we do not guarantee the existence of $\displaystyle \lim_{t\rightarrow+\infty} X(t)$.

\end{enumerate}

\end{rem}

\subsection{Previous Results}
There have been extensive studies on the well-posedness of hyperbolic systems of conservation laws and viscous conservation laws. For scalar conservation laws, Bardos, le Roux, and N\'ed\'elec \cite{BRN79} extended the Kru{\v{z}}kov theory \cite{K1} to initial-boundary value problems by introducing an appropriate entropy boundary condition. We also refer to \cite{Otto96, Vasseur01} for further developments on boundary conditions and traces. For systems of conservation laws, Amadori \cite{AMA97} proved the existence of small BV entropy solutions, with the boundary condition determined by the whole-space Riemann problem, allowing for characteristic boundaries. Later, Amadori and Colombo \cite{AC97} constructed the Standard Riemann Semigroup for $2 \times 2$ systems, and Donadello and Marson \cite{DM07} subsequently extended this construction to $N \times N$ systems with non-characteristic boundaries. More recently, Ancona, Marson, and Spinolo \cite{AMS24} showed the existence of small BV solutions satisfying a boundary condition consistent with the vanishing-viscosity approximation, building on the boundary-layer analyses in \cite{BS09, BS20}.

We next turn to viscous conservation laws. In initial-boundary value problems, boundary effects may give rise to nontrivial stationary solutions as possible asymptotic states. We refer to \cite{HM21, KK06, KZ08,  Kim26a, Kim26, SZK21, UNK10} for studies of stationary solutions for the compressible Navier--Stokes system and related systems in one and several space dimensions. In the one-dimensional inflow setting, the associated boundary Riemann problem for the corresponding Euler system, as characterized in \cite{AMS24, BS09}, determines the asymptotic wave pattern of the Navier--Stokes system (see also \cite{MatBVP}). 
For the isentropic Navier--Stokes (NS) system, Matsumura and Nishihara \cite{MN01} established the stability of BL solutions, rarefaction waves, and their superpositions. Later, Fan, Liu, Wang, and Zhao \cite{FLWZ14} extended these stability results for boundary layer solutions (stationary solutions) and rarefaction waves to allow initial perturbations with large oscillations. 

For the NSF system, Huang, Li, and Shi \cite{HLS10} proved the stability of boundary layers, rarefaction waves, and their superpositions. Qin and Wang \cite{QW09} established the stability of superpositons consisting of a nondegnerate BL solution, rarefaction waves, and a viscous contact wave. Subsequently, Qin and Wang \cite{QW11} extended this result to the case of a degenerate BL solution.

Regarding viscous shock waves in the NS system, Huang, Matsumura, and Shi \cite{HMS03} proved the stability of a superposition of a boundary layer solution and a viscous shock by employing the antiderivative method and by introducing the modified variable. However, the stability of wave patterns including a viscous shock in the inflow problem for the NSF system has remained open. 
In contrast to the classical antiderivative method used in earlier works, the method of $a$-contraction with shifts provides an alternative framework for studying initial-boundary value problems for the compressible Navier--Stokes system (see, for example, \cite{HKKKO25, HKKL25, KOW}). In this paper, we employ the method of $a$-contraction with shifts to resolve the open problem on the long-time behavior toward an asymptotic pattern containing a shock wave. More importantly, our asymptotic pattern consists of a degenerate boundary layer solution, a rarefaction wave, a viscous contact wave, and a viscous shock. We emphasize that this is the most generic asymptotic pattern containing a shock.

\subsection{Idea of Proof}
The main part of the proof is devoted to deriving the $L^2$-estimates for perturbations. In the whole-space setting, the $L^2$-distance provides, as in \cite{KVW-NSF}, a unified framework for establishing the stability of small perturbations around a composite wave consisting of a rarefaction wave, a viscous contact wave, and a viscous shock wave. So, we employ the same $L^2$-framework for the initial-boundary value problem considered in the present paper. The key ingredient  of the $L^2$-framework is the method of $a$-contraction with shifts (in short, the $a$-contraction method), which provides a robust method for controlling perturbations around viscous shock waves. 
This method has been successfully applied to establish contraction properties for viscous shocks on the half-line in \cite{HKKKO25,HKKL25,KOW} (see also  \cite{EEK25,HKK23,HKK25,HL25,Kang19,Kang-V-1,KV21,KV-Inven,KV-2shock,KVW23,KVW-NSF} for the whole-space setting). \\

To derive the zeroth-order estimates, we employ a weighted relative entropy functional with a temporal shift to control the perturbation around the composite wave under the a priori assumption of smallness. The weight is chosen according to the viscous shock profile as defined in \eqref{def:a}, while the time evolution of the weighted relative entropy provides the fundamental dissipative structure established in Lemma \ref{lem:rel}. In addition, the temporal shift is defined as the average of a localized perturbation, which enables us to control the bad terms localized near the shock through the $a$-contraction method.

More precisely, the main difficulties in establishing the zeroth-order energy estimates are as follows.

As the core of the $a$-contraction method, we first maximize the relative entropy flux and then apply the change of variables $\xi\mapsto y$ defined by \eqref{def:y} to the leading-order terms associated with the modulation term (induced by the temporal shift), the diffusion terms, and the hyperbolic terms localized by the derivatives of the shock profile (or equivalently, the weight function). Under this transformation, these leading-order contributions are reduced to quantities that can be controlled by the Poincar\'e-type inequality stated in Lemma \ref{lem:Poincare}.\\
Unlike the previous works based on the $a$-contraction method, however, the main difficulty in the present problem arises from the fact that the rarefaction wave is not assumed to have small strength. In particular, this creates a serious obstacle when reducing the diffusion term
\[
D_{\theta_1}:= \k \int \frac{a}{v\th}|(\theta-\bar\theta)_\x|^2\,d\xi
\]
to the favorable form
\[
\int_{y_0}^1 (y-y_0)(1-y)|(\theta-\bar\theta)_y|^2\,dy
\]
which appears in Lemma \ref{lem:Poincare}, through the Jacobian estimates for $dy/d\xi$ established in Lemma \ref{lem:Jac}. Indeed, because the rarefaction wave may have large amplitude, the coefficient $1/(v\theta)$ is no longer uniformly close to $1/(v^S\theta^S)$. On the other hand, the Jacobian estimates depend only on the viscous shock profile and therefore do not directly capture the effect of the rarefaction wave.

To overcome this difficulty, we observe a delicate monotonicity property of the quantity $v^S\theta^S$, established in \eqref{est:Dth1}, which holds for $1<\gamma\le 2$. This property follows from the explicit representation of the entropy in \eqref{def:s}, together with the fact that the rarefaction wave preserves the entropy. As a consequence, the mismatch between the diffusion coefficient and the Jacobian can still be controlled, allowing us to recover the favorable diffusion structure required for the Poincar\'e-type inequality.

Second, another major difficulty arises in estimating the interactions among the wave components, since the boundary layer is degenerate and the rarefaction wave is allowed to have large strength.
In the present problem, we consider a degenerate boundary layer whose right state lies on the transonic curve. Unlike the non-degenerate case, such a boundary layer exhibits only algebraic decay toward the far field, as shown in Lemma \ref{lem:BL}. Consequently, its interactions with the other wave components decay much more slowly and therefore require a more delicate analysis.

A further difficulty stems from the large-amplitude rarefaction wave. To control its interactions with the other wave components, we introduce a smooth approximate rarefaction profile parameterized by a small parameter $\varepsilon_r$,  whose transition layer has thickness of order $\eps_r^{-1}$.  This construction ensures that the derivatives of the approximate rarefaction are sufficiently small. The parameter $\eps_r$ is then chosen appropriately so that all interaction terms involving the rarefaction wave remain sufficiently small. In this way, the approximation error of the smooth rarefaction profile and the nonlinear interaction terms can be balanced simultaneously within the stability analysis.

Third, a fundamental difficulty in the analysis of the initial-boundary value problem arises from the boundary terms generated by integration by parts. For the rarefaction wave constructed in \eqref{eq:R}, its boundary value coincides with its left state $U_*$, so this component produces no boundary mismatch. Consequently, the boundary discrepancies arise only from the viscous contact wave and the viscous shock.
For the viscous contact wave, the inflow structure implies that its boundary discrepancy decays exponentially in time, as shown in \eqref{est:CDbd}. In particular, it is bounded by the contact wave strength $\delta_C$. For the viscous shock, the a priori estimate for the temporal shift implies that $|\dot{X}(t)|$ remains sufficiently small, so that the propagation speed of the shifted shock stays close to the Rankine--Hugoniot speed. Combined with the exponential decay of the viscous shock profile, this yields a small boundary discrepancy. More precisely, the discrepancy can be bounded by $e^{-c\delta_S \beta}$, provided that the initial shock location $\beta$ is chosen sufficiently large compared to $\delta_S^{-1}$.

Finally, by applying an interpolation inequality, we show that all boundary terms can be controlled by these small boundary discrepancies together with the diffusion terms. Consequently, the boundary contributions can be absorbed into the energy estimates.

\subsection{Organization of the Paper}
The paper is organized as follows. In Section~\ref{Sec:Pre}, we provide the useful quantitative properties of the boundary layer solution, rarefaction wave, viscous contact wave, and viscous shock. In Section~\ref{Sec:Apr}, we state the local existence result and formulate the \textit{a priori} estimates. Section~\ref{Sec:Wint} is devoted to estimates of the interactions between the boundary layer and the elementary waves. In Section~\ref{Sec:L2} and Section~\ref{Sec:H1}, we establish the $L^2$- and $H^1$- energy estimates for the perturbation between the solutions to the inflow problem \eqref{eq:NSF} and the superposed wave profile $\Ubar$, thereby completing the proof of Proposition~\ref{prop:ap}.

\section{Preliminaries}\label{Sec:Pre}
\setcounter{equation}{0}
In this section, we provide several fundamental properties of the basic wave patterns, namely BL solutions, (approximate) rarefaction waves, viscous contact waves, and viscous shock waves. We also present a one-dimensional Poincar\'e-type inequality on any compact interval, which is a key ingredient in deriving the $L^2$ energy estimates localized by a viscous shock wave.

\subsection{Boundary Layer Solution}
For given $U_* \in \Gamma_{trans}^+$ and $U_- \in BL(U_*)$, \cite{WYY2025} showed the existence of a degenerate BL solution $U^{BL}$ connecting $U_-$ to $U_*$. The following lemma presents the quantitative properties of the degenerate BL solution.
\begin{lem}\label{lem:BL} \cite{QW09,WYY2025}
For given $U_* \in \Gamma_{trans}$ and $U_- \in BL(U_*)$, let $U^{BL}(\xi)$ be a BL solution connecting $U_-$ to $U_*$ with strength $\d_{BL} := |u_- - u_*|$. Then, the following holds:
\begin{enumerate}
	\item $v^{BL}_\x >0, \, \, u^{BL}_\x >0$, and $\th^{BL}_\x<0$.
	\item There exist constants $\delta_0>0$ and $C>0$ such that for any $0<\delta_{BL} \leq \delta_0$,
	\begin{equation}\label{est:BL}
		\left|\frac{d^n}{d\x^n}(v^{BL} - v_*, u^{BL} - u_*, \th^{BL} - \th_*)\right| \leq \frac{(\delta_{BL})^{n+1}}{(1+\delta_{BL}\x)^{n+1}}, \quad \x >0, \quad n=0,1,2,\cdots.
	\end{equation}
\end{enumerate}
\end{lem}
\begin{proof}
 For the proofs of (1) and (2), we refer to \cite{WYY2025} and \cite{QW09}, respectively.
\end{proof}
\subsection{Approximate Rarefaction Wave}
In \eqref{def:SR}, we constructed the smooth rarefaction wave $(v^R, u^R, \theta^R)(t,\x)$, which approximates the rarefaction wave $(v^r,u^r,\th^r)$ connecting $(v_*, u_*, \th_*)$ to $(v_m, u_m, \th_m)$. Since the basic properties of $(v^R,u^R,\th^R)$ follow from those of the smooth solution $w$ of the Burgers equation in \eqref{eq:SBurgers}, we first provide the following estimates on $w$.

Throughout the paper, any rarefaction strength appearing in the estimates is replaced by its maximum with $1$ if necessary, and hence is understood to be bounded below by $1$.
\begin{lem} \cite{HMS2003}\label{lem:rarefaction1}
Suppose that $w_-<w_+$. Then, \eqref{eq:SBurgers} admits a unique smooth solution $w(t,x)$ satisfying the following properties:
\begin{enumerate}
    \item ~ $w_-\leq w(t,x)<w_+,~w_x(t,x) > 0 $.
    \item ~ For any $p$ $(1\leq p\leq \infty)$, there exists a constant $C_{pq}$ such that
\[
    \begin{aligned}
  	& \| w_x(t,\cdot)\|_{L^p(\mathbb{R})}\leq
	C_{pq}\min\big{\{}\delta_r\e_r^{1- 1/p},~
	\delta_r^{1/p}t^{-1+1/p}\big{\}}, \\
    & \|w_{xx}(t,\cdot)\|_{L^p(\mathbb{R})}\leq
	C_{pq}\min\big{\{}\delta_r\e_r^{2-1/p},~ \delta_r^{ 1/q}\varepsilon_r^{1-1/p + 1/q}t^{-1+1/q}\big{\}}.
    \end{aligned}
\]
\item ~ If $ x>w_+t $, then
\begin{align*}
    &|w(t,x)-w_+|\leq C\delta_r e^{-c\e_r\left|x-w_+t\right|},\\ 
   &|w_x(t,x)|\leq C \e_r \delta_r e^{-c\e_r\left|x-w_+t\right|}.
\end{align*}
\item  ~ As $t \to \infty$, we have 
\[
\sup\limits_{x\in\mathbb{R}} \left|w(t,x)-w^r\left(\frac{x}{t}\right)\right|\rightarrow 0,
\]
where $w^r(x/t)$ denotes the inviscid rarefaction wave of the Burgers equation given by
\[
w^r\left(\frac{x}{t}\right) = \begin{cases}
    w_-  \, \, &\text{if } x < w_- t,\\
    \frac{x}{t} \, \, &\text{if } w_- t < x < w_+ t,\\
    w_+ \, \, &\text{if } x > w_+ t.
\end{cases}
\]
\end{enumerate}
\end{lem}
\begin{proof}
Since (1), (2), and (4) were proved in \cite{HMS2003}, For completeness, we provide a proof only for property (3). By the method of characteristics, the solution to the Burgers equation \eqref{eq:SBurgers} can be represented as
\begin{align*}
    w(t,x) = w\left(x_0(t,x)\right), 
\end{align*}
where $x_0(t,x)$ is determined by 
\[x=x_0(t,x) + w\left(x_0(t,x)\right) t. \]
It is easy to see that 
\begin{align}\label{eq:wx}
    \frac{\partial x_{0}(t,x)}{\partial x}= \frac{1}{1+w^{\prime}_0(x_0)t}, \quad \text{and} \quad     w_x=\frac{w_0^{\prime}(x_0)}{1+w_0^{\prime}(x_0)t}.
\end{align}
It follows from $\eqref{eq:SBurgers}_2$ that
\begin{align}\label{eq:w0x}
	w_0^{\prime}(x)=
	\begin{cases}
		0, & x<0, \\
		C_q\delta_r\e_r(\e_rx)^qe^{-\e_rx}, & x>0.
	\end{cases}
\end{align}
Combining \eqref{eq:wx} and \eqref{eq:w0x}, we obtain the second estimate in property (3).

For $x\geq 0$, it follows from $\eqref{eq:SBurgers}_2$ that 
\begin{align*}
    \begin{aligned}
        &|w(t,x) -w_+| = |w\left(x_0(t,x)\right)-w_+| =C_q(w_+-w_-)\int^{\infty}_{ \e_r x_0(t,x)}y^qe^{-y}\,dy \\
        &\quad = C_q\d_r\int^{\infty}_{ \e_r \big(x-w\left(x_0(t,x)\right)t\big)}y^qe^{-y}\,dy  \leq  C_q\d_r\int^{\infty}_{ \e_r \left(x-w_+t\right)}y^qe^{-y}\,dy  \leq  C_q\d_re^{-c\e_r \left(x-w_+t\right)},
    \end{aligned}
\end{align*}
which completes the proof.
\end{proof}
We now present the properties of approximate rarefaction wave $U^R$ defined in \eqref{def:APR}. It follows from \eqref{def:APR}$_1$ that 
\[
|U_{\x\x}^R| \leq C(|w_\x|^2 + |w_{\x\x}|)
\]
for some positive constant $C>0$. This estimate, together with Lemma~\ref{lem:rarefaction1}, yields the following Lemma. 
\begin{lem}\label{lem:rarefaction}\cite{QW09}
For any $U_m \in R_1(U_*)$, let $U^R(t,\x)$ be the approximate 1-rarefaction connecting $U_*$ and $U_m$, defined in \eqref{def:SR} with a parameter $0<\e_r \leq 1$. We denote its strength by $\delta_R := |u_* - u_m|\sim|v_* - v_m| \sim |\th_* - \th_m|$. 
Then, $U^R(t,\x)$ satisfies
\begin{equation}\label{eq:R}
	\begin{cases}
	&v^R_{t}-\s_-v^R_{\xi}-u^R _{\xi}=0, \quad t>0, \, \, \x>0,\\
	&u^R_{t}-\s_-u^R_{\xi}+p^R_{\xi}= 0, \\
	&\frac{R}{\g-1}(\theta^R _{t}-\s_-\theta^R_{\xi})+p^Ru^R_{\xi}=0,
    \end{cases}
\end{equation}
where $p^R:=p(v^R, \theta^R)$, and
\[
		 (v^R, u^R, \theta^R)(t,0)=(v_*, u_*, \th_*),\quad (v^R, u^R,
		\theta^R)(t,+\infty)=(v_m, u_m, \th_m).
\]
Moreover, the following estimates hold:
\begin{enumerate}
    \item $v^R_{\x}> 0, \, \,  u^R_{\x}> 0,\, \, \theta^R_{\x}< 0$, for any $t>0$ and $\x > 0$.
     \item  For any $p$  $(1\leq p\leq \infty)$,  there exists a constant
$C_{pq}$ such that for any $t\geq 0$,
\begin{align}\label{est:rare1}
	\begin{aligned}
	& \|(v^R_{\x}, u^R_{\x}, \theta^R_{\x})(t)\|_{L^p}\leq
	C_{pq}\min\big{\{}\delta_R \e_r^{1- 1/p },~ (\delta_R)^{1/p
	}(1+t)^{-1+1/p}\big{\}},  \\
	&\|(v^R_{\x\x},u^R_{\x\x},
	\theta^R_{\x\x})(t)\|_{L^p}\leq C_{pq}\min\big{\{}  \delta_R
	\e_r^{2-1/p},~ \big{(} (\delta_R)^{1/p }+(\delta_R)^{
		1/q}\big{)}(1+t)^{-1+ 1 /q }\big{\}},
	\end{aligned}
\end{align}

\item  If $ \x + \s_- t\geq  \l_1(v_m,
\th_m)(1+t) $, then 
\begin{align*}
\begin{aligned}
& \left|(v^R, u^R, \theta^R)(t,\x)-(v_m, u_m,
\th_m)\right| \leq C\delta_R e^{-c\e_r\left|\left(\xi+\sigma_-t - \l_1(v_m,
	\th_m)(1+t) \right)\right|},  \\
&	\left|(v^R_\xi,u^R_\xi, \theta^R_\xi)\right| \leq C\e_r\delta_R e^{-c\e_r\left|\left(\xi+\sigma_-t - \l_1(v_m,
	\th_m)(1+t) \right)\right|}.
\end{aligned}
\end{align*}
\item $ \sup\limits_{\x\in\mathbb{R}_+}\big{|}(v^R, u^R, \theta^R)( t, \x)-(v^r,u^r,\th^r)
\big{(}\frac{\x}{1+t}\big{)}\big{|}\rightarrow 0$, as $t\rightarrow
\infty$.
\end{enumerate}
\end{lem}

\subsection{Viscous Contact Wave}
Next, we consider $U^* \in CD(U_m)$, where $CD(U_m)$ is the 2-contact discontinuity curve defined in \eqref{def:CDc}. Let $U^{C}$ be the viscous contact wave connecting $U^*$ and $U_m$, defined in \eqref{def:VC}. 
Then, $(v^C, u^C, \th^C)(t,\xi)$ satisfies
\begin{equation}\label{eq:contact}
	\begin{cases}
		&v^C_t - \s_- v^C_\x - u^C_\x = 0, \quad t>0, \quad \x >0, \\
		& u^C_t - \s_- u^C_\x + p(v^C,\theta^C)_\x = \left(\mu \frac{u^C_\x}{v^C}\right)_\x + Q_1^C,\\
		& \frac{R}{\gamma-1}\big(\th^C_t - \s_- \th_\x^C\big) +   p(v^C,\theta^C) u_\x^C = \left(\kappa \frac{\theta^C_\x}{v^C}\right)_\x + \left(\mu \frac{(u^C_\x)^2}{v^C}\right) + Q_2^C,
	\end{cases}
\end{equation}
where $Q_1^C$ and $Q_2^C$ are defined by
\begin{equation}\label{def:QC}
		Q_1^C := \big(u_t^C - \s_- u_\x^C\big) - \left(\mu \frac{u^C_\x}{v^C}\right)_\x, \quad Q_2^C := -\left(\mu \frac{(u^C_\x)^2}{v^C}\right) .
\end{equation}
The following Lemma provides the basic properties of the viscous contact wave.
\begin{lem}\label{lem:CD}\cite{HXY08}
For any $U^* \in CD(U_m)$, let $U^{C}(t,\x)$ be the viscous contact wave connecting $U^*$ and $U_m$ defined in \eqref{def:VC}, with strength $\d_C := |v^* - v_m| \sim |\th^* - \th_m|$. Then, we have the following:
\begin{enumerate}
    \item $v^C$ and $\th^C$ are monotone. More precisely, we have 
    \[
    v^C_\x >0, \, \, \text{and } \, \th^C_\x >0, \,\, \text{or } \,  v^C_\x< 0, \, \, \text{and } \, \th^C_\x<0.
    \]
    \item There exists a constant $C_1>0$ such that the following estimates hold:
\begin{equation}\label{est:contact}
	\begin{aligned}
		&(v^{C} - v_m, u^{C} - u_m, \th^{C} - \th_m) = O(1)\d_C e^{-\frac{C_1 (\x + \s_- t)^2}{1+t}}, \quad \forall \,0 < \x < -\s_- t,\\
		&(v^{C} - v^*, u^{C} - u^*, \th^{C} - \th^*) = O(1)\d_C e^{-\frac{C_1 (\x + \s_- t)^2}{1+t}}, \quad \forall \, \x > -\s_- t,\\
		&(\rd_\x^n v^{C}, \rd_\x^n \th^{C})(t,\x) = O(1) \d_C (1+t)^{-\frac{n}{2}}e^{-\frac{C_1 (\x + \s_- t)^2}{1+t}}, \quad \forall (t,\x) \in \mathbb{R}_+ \times \Rp, \quad n=1,2,\cdots \\
		&\rd_\x^n u^{C}(t,\x) = O(1)\d_C (1+t)^{-\frac{1+n}{2}}e^{-\frac{C_1 (\x + \s_- t)^2}{1+t}} \quad \forall \,(t,\x) \in \mathbb{R}_+ \times \Rp, \quad n=1,2,\cdots.
	\end{aligned}
\end{equation}
\end{enumerate}
\end{lem}
By Lemma \ref{lem:CD}, the error terms $Q_1^C$ and $Q_2^C$ in \eqref{def:QC} can be estimated as follows:
\begin{equation}\label{est:QC}
	\begin{aligned}
		Q_1^C = O(1)\d_C(1+t)^{-\frac{3}{2}}e^{-\frac{2C_1 (\x + \s_- t)^2}{1+t}}, \quad 
		Q_2^C = O(1)\d_C (1+t)^{-2}e^{-\frac{2C_1 (\x + \s_- t)^2}{1+t}}.
	\end{aligned}
\end{equation}

\subsection{Viscous Shock Wave}
Finally, we consider a state $U_+ \in S_3(U^*)$, i.e., $U^*$ and $U_+$ satisfy the Rankine-Hugoniot conditions:
\begin{equation}\label{eq:RH}
	\begin{cases}
		&-\s (v_+ - v^*) - (u_+ - u^*) = 0,\\
		&- \s (u_+ - u^*) + p(v_+,\theta_+) - p(v^*,\th^*) = 0,\\
		&- \s (E_+ - E^*) + \big(p(v_+,\theta_+)u_+ -  p(v^*,\th^*) u^*\big) = 0,
	\end{cases}
\end{equation}
and the entropy condition $v^* < v_+, u^* > u_+$, and $\th^* > \th_+$. It is well-known that there exists a viscous 3-shock profile $U^S(\zeta)= U^S(x-\sigma t) = U^S(\x - (\s - \s_-)t)$ satisfying \eqref{eq:VS}, provided that the shock strength is sufficiently small. Then, the following lemma provides the quantitative estimates for viscous weak shock waves.
\begin{lem}\label{lem:VS} \cite{KVW-NSF}
	For any fixed state $U^*$, there exists a positive constant $\d_0$ such that the following holds. For any end state $U_+$ such that $U_+ \in S_3(U^*)$ with $|U_+ - U^*| \leq \d_0$, there exists a unique solution $U^{S}$ to \eqref{eq:VS} with $U^{S}(0) = \frac{U^* + U_+}{2}$. Let $\d_S$ be the shock strength defined by $\d_S := |u_+ - u^*| \sim |v_+ - v^*| \sim |\th - \th^*|$. Then, there exists a positive constant $C>0$ such that the following estimates hold:
	\begin{equation}\label{est:shock}
		\begin{aligned}
			&v_\zeta^S >0, \quad u_\zeta^S <0, \quad \th_\zeta^S <0, \quad \forall \zeta \in \mathbb{R},\\
			&|(v^S(\zeta) - v^*, u^S(\zeta) - u^*, \th^S(\zeta) - \th^*)| \leq C\d_S e^{-C\d_S |\zeta|}, \quad \zeta < 0,\\
			&|(v^S(\zeta) - v_+, u^S(\zeta) - u_+, \th^S(\zeta) - \th_+)| \leq C\d_S e^{-C\d_S |\zeta|}, \quad \zeta > 0,\\
			&|(v_\zeta^S, u_\zeta^S, \th_\zeta^S)| \leq C\d_S^2 e^{-C\d_S |\zeta|}, \quad \forall \zeta \in \mathbb{R},\\
			&|(v_{\zeta \zeta}^S, u_{\zeta \zeta}^S, \th_{\zeta \zeta}^S)| \leq C\d_S |(v_\zeta^S, u_\zeta^S, \th_\zeta^S)|.
		\end{aligned}
	\end{equation}
In particular, $|v_\zeta^S| \sim |u_\zeta^S| \sim |\th_\zeta^S|$ , and for all $\zeta \in \mathbb{R}$,
\begin{equation}\label{est:shock2}
\begin{aligned}
\left|v^S_\zeta +\frac{1}{\s^*} u^S_\zeta\right| \leq C\d_S |u^S_\zeta|, \quad 
\text{and} \quad \left|\th^S_\zeta - \frac{(\gamma -1) p^*}{R\s^*}u^S_\zeta\right| \leq C\delta_S |u^S_\zeta|,
\end{aligned}
\end{equation}
where
\[
p^* := p(v^*,\th^*), \quad \s^* := \sqrt{\frac{\gamma p^*}{v^*}} = \frac{\sqrt{\gamma R \th^*}}{v^*},
\]
which satisfies
\begin{align*}
|\s - \s^*| \leq C\d_S.
\end{align*}
\end{lem}
\subsection{Poincar\'e-type Inequality}
We present a Poincar\'e-type inequality on any compact subset of $\mathbb{R}$, where the optimal constant $1/2$ is independent of the size of the domain.
\begin{lem}\label{lem:Poincare}
		\cite{HKKL25} For any $c<d$ and function $f:[c,d] \to \mathbb{R}$ satisfying $\int_c^d (y-c)(d-y)|f'(y)|^2\,dy < \infty$,
		\[
		\int_c^d \left|f(y) - \frac{1}{d-c}\int_c^d f(y)\,dy \right|^2\,dy \leq \frac{1}{2}\int_c^d (y-c)(d-y)|f'(y)|^2\,dy.
		\]
\end{lem}

\section{A Priori Estimates and Proof of Main Theorem}\label{Sec:Apr}
\setcounter{equation}{0}
\subsection{Local Existence} 
In this subsection, we state a local-in-time existence result for the inflow problem \eqref{eq:NSFL}. Since the proof is based on a standard iteration argument (see, for example, \cite{Sol76}), we omit the details.
 \begin{proposition}\label{prop:local}
	For any constant $\beta>0$, let $\underline{v}, \underline{u}$ and $\underline{\theta}$ be smooth functions such that
	\[(\underline{v}(x), \underline{u}(x),\underline{\theta}(x))=(v_+,u_+,\theta_+),\quad\mbox{for}\quad x\ge\beta,\quad \underline{v}(0)>0 \quad \text{and} \quad \underline{\theta}(0)>0 .\]
	For any constants $M_0$, $M_1$, $\underline{\kappa}_0$, $\overline{\kappa}_0$, $\underline{\kappa}_1$, and $\overline{\kappa}_1$ with $0<M_0<M_1$ and $0<\underline{\kappa}_1<\underline{\kappa}_0<\overline{\kappa}_0<\overline{\kappa}_1$, there exists a constant $T_0>0$ such that if
	\begin{align*}
		&\|(v_0-\underline{v},u_0-\underline{u}, \theta_0-\underline{\theta})\|_{H^1(\R_+)}\le M_0,\\
		&0<\underline{\kappa}_0\le v_0(x), \theta_0(x)\le \overline{\kappa}_0,\quad x\in\mathbb{R}_+,
	\end{align*}
	the inflow problem \eqref{eq:NSFL} admits a unique solution $(v,u, \theta)$ on $[0,T_0]$ such that
	\begin{align*}
		v -\underline{v}\in C([0,T_0];H^1(\mathbb{R}_+)),\quad(u-\underline{u}, \theta-\underline{\theta}) \in C([0,T_0];H^1(\mathbb{R}_+))\cap L^2(0,T_0;H^2(\mathbb{R}_+)),
	\end{align*}
	and
	\[\|(v-\underline{v},u-\underline{u}, \theta-\underline{\theta})\|_{L^\infty(0,T_0;H^1(\R_+))}\le M_1.\]
	Moreover, for the inflow problem, we have
	\[
	\underline{\kappa}_1\le v(t,x) \le \overline{\kappa}_1, \quad \underline{\kappa}_1\le \theta(t,x)\le \overline{\kappa}_1, \quad \forall t>0, \, \, \forall x\in\mathbb{R}_+,
	\]
	and
	\[
	v(t, 0)=v_->0,  \quad u(t,0)=u_- >0 \, \, ,\quad \theta(t,0) = \theta_->0, \quad \forall t>0.
	\]
\end{proposition}




\subsection{Construction of the Shift and the Weight Function}
In order to establish the stability of the viscous shock, we need to consider a suitable shift of the viscous shock wave. To this end, we construct the shift function $X(t)$ as the solution to the following ODE:
\begin{equation}\label{def:X}
\begin{cases}
	&\dot{X}(t) = -\frac{M}{\d_S}\left(\intRp a \left((u^S)_\x (u-\ubar) + \frac{(v^S)_\x p^*}{v^*}(v-\vbar) + \frac{R}{\gamma - 1}\frac{(\th^S)_\x}{\th^*}(\th-\thbar) \right)\,d\xi\right),\\
   &X(0)=0. 
    \end{cases}
\end{equation}
where $(v,u,\th)$ is a solution to \eqref{eq:NSF}, $a$ is the weight function defined in \eqref{def:a}, and $M>0$ is a constant to be chosen later. By the argument of \cite[Lemma 3.2]{KVW-NSF}, the initial value problem \eqref{def:X} admits a unique Lipschitz continuous solution on $[0,T]$ under the assumptions of Proposition~\ref{prop:ap}.

We then define the weight function $a$ by
\begin{equation}\label{def:a}
	a = a(t,\x) := 1 + \frac{(u^* - u^S(\x - (\s - \s_-)t - X(t) - \beta))}{\sqrt{\d_S}}.
\end{equation}
Since $u^S$ is decreasing, we find that $a$ is increasing. Moreover, observe that 
\(
1 \leq a \leq 1+\sqrt{\d_S},
\)
and
\begin{align}\label{da}
\rd_\x a(t, \x) = -\frac{u^S_\x}{\sqrt{\d_S}} > 0.
\end{align}
We remark that the weight function $a$ will play a crucial role in the $L^2$ estimates by producing the new good term $G^{new}$ in \eqref{Good}.
\subsection{A Priori Estimates}
We now present the \textit{a priori} estimates, which are a key ingredient in proving Theorem~\ref{thm:main}.

For notational convenience, we omit the arguments if there is no confusion:
\[
U^{BL} = U^{BL}(\xi), \quad U^{R} = U^R(t,\xi), \quad U^{C} = U^{C}(t,\xi), \quad U^{S} = U^{S}(\x - (\s - \s_-)t - X(t) - \b).
\]
\begin{proposition}\label{prop:ap}
 Suppose that $1<\gamma \leq 2$. For any $U_* \in \Gamma_{trans}^+$, fix a state $U_m \in \Omega_{super}^+$ with $U_m \in R_1(U_*)$. Then there exist positive constants $\d_0, \e,c$, and $C_0$ such that the following holds.

   For any three states $U_-, U^*, U_+$ with $U_- \in BL(U_*), U^* \in CD(U_m)$, and $U_+ \in S_3(U^*)$ satisfying 
   \[
    |U_- - U_*| + |U_m - U^*| + |U^* - U_+| < \d_0,
   \]
   there exist a sufficiently large constant $\beta>0$ $($depending only on the shock strength $|U^* - U_+|)$ and a sufficiently small constant $0<\e_r<1$ $($depending on $\d_0$ and the rarefaction strength $|U_m - U_*|)$ such that the following holds.

Let $(v,u,\theta)$ be the solution to \eqref{eq:NSF}-\eqref{IC} on $[0,T]$ for some $T>0$, and let $(\bar v, \bar u, \bar \theta)$ be the superposition defined in \eqref{def:Ubar}. 
Let $X(t)$ be the shift function defined in \eqref{def:X}, and let $a$ be the weight function defined in \eqref{def:a}. Assume that
\begin{equation*}
\begin{aligned}
v - \bar v \in C([0,T];H^1(\bbr_+)), \quad (u - \bar u, \theta - \bar \theta) \in C([0,T];H^1(\bbr_+)) \cap L^2(0,T;H^2(\bbr_+)),
\end{aligned}
\end{equation*}
and
\begin{equation}\label{est:apriori}
\|(v-\bar v, u-\bar u, \theta-\bar \theta)\|_{L^\infty(0,T;H^1(\bbr_+))} \le \e.
\end{equation}
Then, the following estimate holds:
\begin{align}
\begin{aligned} \label{estsuppertur}
&\sup_{t \in [0,T]} \|(v-\bar v, u-\bar u, \theta-\bar \theta)(t,\cdot)\|_{H^1(\bbr_+)}^2 
+  \int_0^T  \big(\delta_S|\dot X(t)|^2
+ G^{BL}(U)+ G^R(U) +  G^S(U) \big)\,dt\\
&\qquad
+ \int_0^T \big( D_{v_1}(U) +  D_{u_1}(U) +D_{u_2}(U) + D_{\theta_1}(U) + D_{\theta_2}(U) \big)\,dt\\
&\quad \le 
C_0 \|(v-\bar v, u-\bar u, \theta-\bar \theta)(0,\cdot)\|_{H^1(\bbr_+)}^2 + C_0\d_0^{1/40}
+ C_0 e^{-c\d_S \beta}, 
\end{aligned}
\end{align}
where
\begin{equation*}
\begin{aligned}
&G^{BL}(U):= \int_{\bbr_+} |u^{BL}_\xi|\, |(v-\bar v, \theta-\bar\theta)|^2 \,d\xi, \qquad
G^R(U):= \int_{\bbr_+} |u^R_\xi|\, |(v-\bar v, \theta-\bar\theta)|^2 \,d\xi,\\
&G^S(U):= \int_{\bbr_+} |v^S_\xi|\, |(v-\bar v, u-\bar u, \theta-\bar\theta)|^2 \,d\xi,\\
&D_{v_1}(U):= \int_{\bbr_+} |(v-\bar v)_\xi|^2d\xi, \,\,\,
D_{u_1}(U):= \int_{\bbr_+} |(u-\bar u)_\xi|^2d\xi, \,\,\,
D_{u_2}(U):= \int_{\bbr_+} |(u-\bar u)_{\xi\xi}|^2d\xi,\\
& D_{\theta_1}(U):= \int_{\Rp} |(\theta-\bar\theta)_\xi|^2d\xi, \,\,\, D_{\theta_2}(U):= \int_{\bbr_+} |(\theta-\bar\theta)_{\xi\xi}|^2\,d\xi.
\end{aligned}
\end{equation*}
\end{proposition}

\subsection{Proof of Theorem 1.1}
Based on Proposition~\ref{prop:local} and Proposition~\ref{prop:ap}, we use a continuity argument to prove the global existence result \eqref{GE} for the inflow problem. The detailed proof of \eqref{GE} is deferred to the Appendix~\ref{APP:conti}.

In addition, the time-asymptotic stability \eqref{Asym} can be proved by showing that 
\[
\|(v-\bar{v},u-\ubar,\th - \thbar)_\x(t,\cdot)\|_{L^2(\Rp)}^2 \in W^{1,1}(\Rp).
\]
Since the proof of \eqref{Asym} is standard and similar to that in \cite{HKKKO25}, we omit the details.

Finally, \eqref{Xasym} directly follows from \eqref{def:X} and \eqref{Asym}.
\subsection{Notations}
The following sections are devoted to proving Proposition~\ref{prop:ap}. In what follows, we omit the spatial domain $\Rp$ whenever there is no confusion. For example, we write $L^p$ and $H^1$ for $L^p(\Rp)$ and $H^1(\Rp)$, respectively. Moreover, we introduce the following notations:
\begin{equation}\label{not:vi}
\begin{aligned}
	&(v^{BL},u^{BL}, \th^{BL}) =: (v_1,u_1,\th_1), \quad (v^R,u^R, \th^{R}) =: (v_2,u_2,\th_2)\\
	&(v^C,u^C, \th^{C}) =: (v_3,u_3,\th_3), \quad (v^S,u^S, \th^S) =: (v_4,u_4,\th_4),
\end{aligned}
\end{equation}
and
\[
(v-\vbar, u-\ubar, \th - \thbar):=(\phi,\psi,\chi).
\]
Moreover, the constant $C>0$ may change line to line, which is independent of $\d_0,\e_r,\e,\b$, and time $T>0$. Unless otherwise stated, $C$ may depend on $\d_R$; however, in \eqref{est:Ji}, \eqref{est:ucont}, \eqref{est:wder}, \eqref{JiL2}, and \eqref{est:wder2}, it is independent of $\d_R$. 
\section{Wave Interaction Estimates}\label{Sec:Wint}
\setcounter{equation}{0}
In this section, we provide several useful estimates on the wave interactions. 
To this end, we define $J_i$, $i=1,\cdots,6$, by
\begin{align}\label{def:J}
	\begin{split}
	\begin{aligned}
	&J_1 := |v_\x^{BL}| |v^R - v_*| + |v_\x^R||v^{BL} - v_*|,  && J_2 :=  |v_\x^{BL}| |v^C - v_m| + |v_\x^C||v^{BL} - v_*|,\\
	&J_3 :=  |v_\x^{BL}| |v^S - v^*| + |v_\x^S||v^{BL} - v_*|,  &&J_4 :=  |v_\x^{R}| |v^C - v_m| + |v_\x^C||v^{R} - v_m|,\\
	&J_5 :=  |v_\x^{R}| |v^S - v^*| + |v_\x^S||v^{R} - v_m|, &&J_6 :=  |v_\x^{C}| |v^S - v^*| + |v_\x^S||v^{C} - v^*|.
	\end{aligned}
	\end{split}
\end{align}

\subsection{$L^1$ Interaction Estimates}
First, we provide the $L^1$ wave-interaction estimates. We remark that the estimates in \eqref{est:Ji} of Lemma~\ref{lem:win}, as well as those in \eqref{est:ucont}, and \eqref{est:wder} of Lemma~\ref{lem:estL12} remain valid even when the strength of the boundary layer solution is large.
\begin{lem}\label{lem:win}
Under the assumptions of Proposition~\ref{prop:ap}, there exist positive constants $c$ and $C$ such that the following estimates hold:
\begin{align}\label{est:Ji}
	\begin{split}
		\begin{aligned}
				&\intRp J_1 \,d\x \leq C\left(\frac{(\e_r\d_R)^{1/8}}{(1+t)^{7/8}}\ln(1+\delta_{BL}t) + \frac{\delta_{BL}\delta_R}{1+\delta_{BL} t}\right),  \\
				&  \intRp J_2 \,d\x \leq C\delta_{BL}\d_C \left(e^{-ct} + \frac{1}{1+\delta_{BL}t}\right),  \\
				& \intRp J_3 \,d\x \leq C\delta_{BL}\d_S\left( e^{-c\d_S t} + \frac{1}{1+\delta_{BL}t}\right),    \\
				& \intRp J_4 \,d\x  \leq 
                C\delta_R \delta_C e^{-ct}+ C\delta_C\frac{\delta_R}{\e_r}  e^{- c\e_r t},   \\
				& \intRp J_5 \,d\x \leq C\delta_R\d_S(\e_r  +  \delta_S)e^{-c\delta_S t} + C\delta_R\delta_S \frac{(\e_r  +  \delta_S)}{\e_r} e^{-c\e_r t},\\
				& \intRp J_6 \,d\x \leq C\d_C\d_S \left(e^{-c\delta_St}+e^{-ct}\right).     
		\end{aligned}
	\end{split}
\end{align}
Consequently, we have 
\begin{equation}\label{est:431}
\int_0^t \sum_{i=1}^6 \|J_i\|_{L^1}^{4/3}\,d\t \leq C\d_0 + C(\e_r\d_R)^{1/6} + C\d_0 \left(\frac{\d_R}{\e_r}\right)^{4/3}\frac{1}{\e_r}.
\end{equation}
\end{lem}
\begin{proof}
We estimate $J_1$ following the idea of \cite{MN01}. First, it follows from Lemma~\ref{lem:BL} and Lemma~\ref{lem:rarefaction} that $
v_\x^{BL}(v^R - v_*) > 0$, and $v_\x^R(v_* - v^{BL}) > 0$. 
Therefore, by integration by parts, together with the fact that
$v^{R}(t,0)=v_*=v^{BL}(\infty)$, we obtain
\begin{align*}
	\begin{aligned}
		\intRp J_1 \,d\x 
        &=\left(\int_0^{t}+\int_{t}^\infty\right)\left(v^{BL}_\xi(v^R-v_*)+v^R_\xi(v_*-v^{BL})\right)d\xi\\
		&= \Bigg(\left[(v^{BL}-v_*)(v^R-v_*)\right]^{\xi=t}_{\xi=0}-2\int_0^{t}v^R_{\xi}(v^{BL}-v_*)\,d\xi\\
		&\hspace{1cm}+\left[(v^R-v_*)(v_*-v^{BL})\right]_{\xi=t}^{\xi=\infty}+2\int_{t}^\infty v^{BL}_{\xi}(v^R-v_*)\,d\xi\Bigg)\\
		&\le 2\int_0^{t}v^R_\xi(v_*-v^{BL})\,d\xi +2\int_{t}^\infty v^{BL}_\xi (v^R-v_*)\,d\xi.
	\end{aligned}
\end{align*}
Here, we used $(v^{BL}(\xi)-v_*)(v^R(t,\xi)-v_*)< 0$ for all $t,\x>0$.

We observe from \eqref{est:BL} and \eqref{est:rare1} that, for all $t, \xi>0$,
\begin{equation}\label{est:BLR}
	\begin{aligned}
		&|v^R-v_*|\le C\delta_R,\quad |v^R_\xi|\le \|v^R_\xi\|_{L^\infty}^{1/8}\|v^R_\xi\|^{7/8}_{L^\infty}\le \frac{C(\e_r\delta_R)^{1/8}}{(1+t)^{7/8}},\\
		&|v^{BL}-v_*|\le \frac{C\delta_{BL}}{1+\delta_{BL}\xi},\quad |v^{BL}_\xi|\le \frac{C\delta^2_{BL}}{(1+\delta_{BL}\xi)^2}.
	\end{aligned} 
\end{equation}
Therefore, by \eqref{est:BLR}, we obtain
\begin{align*}	
	\intRp J_1 \,d\x \leq & C\int_0^{t}v^R_\xi(v_*-v^{BL})\,d\xi + C\int_{t}^\infty v^{BL}_\xi(v^R-v_*)\,d\xi\\
	\le  & \frac{C(\e_r\delta_R)^{1/8}\delta_{BL}}{(1+t)^{7/8}}\int_0^{t}\frac{1}{1+\delta_{BL}\xi}\,d\xi+C\delta^2_{BL}\delta_R\int_{t}^\infty\frac{1}{(1+\delta_{BL}\xi)^2}\,d\xi\\
	=  & \frac{C(\e_r\delta_R)^{1/8}}{(1+t)^{7/8}}\ln(1+\delta_{BL}t)+\frac{C\delta_{BL}\delta_R}{1+\delta_{BL}t}.
\end{align*}
For $J_2$, we assume that $v^C$ is monotone increasing. Then, we have
\begin{align*}
	\intRp J_2 \,d\x = &\intRp \big( v_\x^{BL} (v^C - v_m) + v_\x^C (v_* - v^{BL}))\,d\x\\
	= &\left(\int_0^{-\frac{\s_- t}{2}} + \int_{-\frac{\s_- t}{2}}^\infty \right)\big( v_\x^{BL} (v^C - v_m) + v_\x^C (v_* - v^{BL}))\,d\x\\
	= &(v^{BL} - v_*)(v^C - v_m) \Big|_{0}^{-\frac{\s_- t}{2}} - 2\int_0^{-\frac{\s_- t}{2}}v_\x^C (v^{BL}-v_*)\,d\x\\
	& + (v^C - v_m)(v_* - v^{BL})\Big|_{-\frac{\s_- t}{2}}^\infty + 2 \int_{-\frac{\s_- t}{2}}^\infty v_\x^{BL} (v^C - v_m)\,d\x.
\end{align*}
Using $(v^{BL}-v_*)(v^C - v_m) <0$, Lemma \ref{lem:BL} and Lemma \ref{lem:CD}, we have 
\begin{align*}
	&\intRp J_2 \,d\x \leq 
	C\delta_{BL}|v^C - v_m| \Big|_{\xi=0} + 2\int_0^{-\frac{\s_- t}{2}}v_\x^C (v_* - v^{BL})\,d\x + 2 \int_{-\frac{\s_- t}{2}}^\infty v_\x^{BL} (v^C - v_m)\,d\x\\
	&\quad \leq C\delta_{BL}\d_C e^{-\frac{C_1 (\s_- t)^2}{1+t}} + C\frac{\delta_{BL}\d_C}{\sqrt{1+t}}\int_0^{-\frac{\s_- t}{2}} e^{-\frac{C_1 (\x + \s_- t)^2}{1+t}}\,d\x + C\delta_{BL}^2\d_C\int_{-\frac{\s_- t}{2}}^\infty \frac{1}{(1+\delta_{BL}\x)^2}\,d\x\\
	&\quad \leq C\delta_{BL}\d_C \left(e^{-{Ct}} + \frac{1}{1+\delta_{BL}t}\right).
\end{align*}
The case where $v^C$ is decreasing can be treated similarly.

The estimate for $J_3$ follows by the same argument as in \cite[Lemma 4.1]{HKKKO25}; hence we omit the details. 

Next, we estimate the term $J_4$. We first decompose $J_4$ into two terms as follows:
\begin{align*}
	&\intRp J_4 \,d\x = \intRp |v_\x^{R}| |v^C - v_m| + |v_\x^C||v^{R} - v_m|\,d\x\\
 &\,\, =\left(\int_{0\leq \xi \leq  \frac{1}{2}\lambda_1(u_m, \theta_m)(1+t)-\sigma_-t} + \int_{\xi \geq  \frac{1}{2}\lambda_1(u_m, \theta_m)(1+t)-\sigma_-t}\right) |v_\x^{R}| |v^C - v_m| + |v_\x^C||v^{R} - v_m|d\x.
\end{align*}
For $\xi +\sigma_- t\leq  \frac{1}{2}\lambda_1(u_m, \theta_m)(1+t)$, 
it follows from Lemma \ref{lem:CD} that
\begin{equation}\label{est:Cin}
\begin{aligned}
|(v^C - v_m, u^C - u_m, \theta^C - \th_m)| +|v^C_\xi| \leq C\d_C e^{-\frac{C_1 |\xi + \sigma_-t|^2}{1+t}}   \leq  C\d_C e^{-\frac{C_1 |\xi + \sigma_-t|^2}{2(1+t)}}  e^{-Ct}.
\end{aligned}
\end{equation}
Combining this with Lemma~\ref{lem:rarefaction}, we obtain
\begin{align}\label{est:j41}
\begin{aligned}
&\int_{0\leq \xi \leq  \frac{1}{2}\lambda_1(u_m, \theta_m)(1+t)-\sigma_-t}  \left(|v_\x^{R}| |v^C - v_m| + |v_\x^C||v^{R} - v_m| \right)\,d\x  \\
&\quad \leq  C\left(\e_r \delta_R  +\delta_R \right) \d_C\int_{0\leq \xi \leq  \frac{1}{2}\lambda_1(u_m, \theta_m)(1+t)-\sigma_-t}  e^{-\frac{C_1 |\xi + \sigma_-t|^2}{2(1+t)}} e^{-Ct} d\xi \leq C\left(\e_r \delta_R   +\delta_R \right)\delta_C e^{-Ct}.
\end{aligned}
\end{align}
On the other hand, if $\xi +\s_-t \geq \frac{1}{2}\lambda_1(u_m,\theta_m)(1+t) = -\frac{1}{2}|\lambda_1(u_m,\theta_m)|(1+t)$, then we have 
\begin{align}\label{est:Rdecay}
	\begin{aligned}
		e^{-C\e_r|\xi +\sigma_- t -\lambda_1(u_m, \theta_m)(1+t)|} = &e^{-C\e_r|\xi +\sigma_- t- \frac{1}{2}\lambda_1(u_m, \theta_m)(1+t)|} e^{-\frac{1}{2}C\e_r |\lambda_1(u_m, \theta_m)|(1+t)}. 
	\end{aligned}
\end{align}
Therefore, by \eqref{est:Rdecay} and Lemma~\ref{lem:CD}, we obtain
\begin{align}\label{est:j42}
\begin{aligned}
& \int_{\xi \geq  \frac{1}{2}\lambda_1(u_m, \theta_m)(1+t)-\sigma_-t}\left( |v_\x^{R}| |v^C - v_m| + |v_\x^C||v^{R} - v_m|\right)\,d\x  \\
&\,\, \leq C\delta_C \delta_R(\e_r + 1)  \int_{\xi \geq  \frac{1}{2}\lambda_1(u_m, \theta_m)(1+t)-\sigma_-t}  e^{-C\e_r|\xi +\sigma_- t - \frac{1}{2}\lambda_1(u_m, \theta_m)(1+t)|} e^{-C\e_r| \lambda_1(u_m, \theta_m)|(1+t)} d\xi   \\
&\,\, \leq  C\delta_C\frac{\delta_R + \e_r \d_R }{\e_r}  e^{-C\e_r t}.
\end{aligned}
\end{align}
Combining \eqref{est:j41} and \eqref{est:j42}, and using $\e_r \leq 1$, we have
\begin{align*}
&\intRp J_4 \,d\x  \leq C\delta_R \delta_C e^{-Ct} + C\delta_C\frac{\delta_R}{\e_r}  e^{-C\e_r t}.
\end{align*}
For $J_5$, we first split it as:
\begin{align*}
&\intRp J_5 \,d\x = \intRp  |v_\x^{R}| |v^S - v^*| + |v_\x^S||v^{R} - v_m| \, d\xi  \\
&=\left(\int_{0\leq \xi \leq  \frac{1}{2}\lambda_1(u_m, \theta_m)(1+t)-\sigma_-t} + \int_{\xi \geq  \frac{1}{2}\lambda_1(u_m, \theta_m)(1+t)-\sigma_-t}\right) \left(  |v_\x^{R}| |v^S - v^*| + |v_\x^S||v^{R} - v_m|   \right) d\xi.
\end{align*}
Since $\xi +\s_-t \leq \frac{1}{2}\lambda_1(u_m,\theta_m)(1+t)<0$, we obtain 
 \begin{align}\label{est:vsr}
	\xi -(\sigma- \sigma_-)t -X(t) -\beta \leq  \frac{1}{2}\lambda_1(u_m, \theta_m)(1+t) -\frac{3}{4} \sigma t -\beta  \leq -\frac{3}{4}\sigma t <0.
\end{align}
Using an argument similar to that for \eqref{est:Cin}, together with \eqref{est:vsr}, Lemma \ref{lem:rarefaction}, and \eqref{est:shock}, we have 
\begin{align}\label{est:j51}
\begin{aligned}
&\int_{0\leq \xi \leq  \frac{1}{2}\lambda_1(u_m, \theta_m)(1+t)-\sigma_-t} \left(  |v_\x^{R}| |v^S - v^*| + |v_\x^S||v^{R} - v_m|   \right) \, d\xi \\
&\quad \leq C\delta_R\delta_S(\e_r+ \delta_S) \int_{0\leq \xi \leq  \frac{1}{2}\lambda_1(u_m, \theta_m)(1+t)-\sigma_-t}  e^{-\frac{1}{2}C\delta_S|\xi -(\sigma-\sigma_-)t -X(t) -\beta|} e^{-C\delta_S t} d\xi  \\
&\quad \leq C\delta_R \d_S(\e_r+\d_S) e^{-C\delta_S t}.
\end{aligned}
\end{align} 
On the other hand, for $\xi +\s_-t \geq \frac{1}{2}\lambda_1(u_m,\theta_m)(1+t)$, using Lemma \ref{lem:rarefaction}, \eqref{est:Rdecay} and Lemma \ref{lem:VS}, we obtain
\begin{align}\label{est:j52}
\begin{aligned}
&\int_{\xi \geq  \frac{1}{2}\lambda_1(u_m, \theta_m)(1+t)-\sigma_-t}\left(  |v_\x^{R}| |v^S - v^*| + |v_\x^S||v^{R} - v_m|   \right) \, d\xi \\
&\quad \leq C\delta_S\delta_R(\varepsilon_r+\delta_S) e^{-C\e_r t} \int_{\xi \geq  \frac{1}{2}\lambda_1(u_m, \theta_m)(1+t)-\sigma_-t}  e^{-C\e_r|\xi +\sigma_- t - \frac{1}{2}\lambda_1(u_m, \theta_m)(1+t)|}\,d\x\\
&\quad \leq C\delta_S\delta_R\frac{(\varepsilon_r+\delta_S)}{\e_r}e^{-C\e_rt}.
\end{aligned}
\end{align}
Combining \eqref{est:j51} and \eqref{est:j52}, we obtain the desired estimate.

Now, we consider $J_6$. Let $\alpha:=\frac{1}{2}(-\sigma_- + (\s - \s_-)) = \frac{\sigma-2\sigma_-}{2}$. By Lemma~\ref{lem:CD} and Lemma~\ref{lem:VS}, we obtain 
\begin{align*}
	\intRp J_6 &= \left(\int_0^{\alpha t} + \int_{\alpha t}^\infty \right) |v_\x^{C}| |v^S - v^*| + |v_\x^S||v^{C} - v^*|\,d\x\\
	&\leq C\d_C \d_S \int_0^{\alpha t} e^{-C\d_S |\x - (\s - \s_-)t - X(t) - \beta|}\,d\x + C\d_C \d_S \int_{\alpha t}^\infty e^{-\frac{C_1 \left|\x + \s_- t\right|^2}{1+t}}\,d\x\\
	&\leq C\d_C\d_S e^{-C\delta_St}+C\d_C\d_S e^{-Ct}.
\end{align*}
$\bullet$ {\bf Proof of \eqref{est:431}}
From \eqref{est:Ji}, we deduce that 
\[
\begin{aligned}
    &\int_0^t \sum_{i=1}^6 \|J_i\|_{L^1}^{4/3}\,d\t \\
    &\,\, \leq C\big((\e_r\d_R)^{1/6} + \d_{BL}^{1/3}\d_R^{4/3}\big)+ C\big(\d_C^{4/3}+ \d_S^{4/3}\big) \big(\d_{BL}^{4/3} + \d_{BL}^{1/3}\big)\\
    &\quad \,\,+ C\left((\d_R\d_C)^{4/3} +\d_C^{4/3}\left(\frac{\d_R}{\e_r}\right)^{4/3}\frac{1}{\e_r} \right)\\
    &\quad \,\,+C\left(\d_R^{4/3}\d_S^{1/3}(\e_r + \d_S)^{4/3} + C\d_S^{4/3}(\e_r+\d_S)^{4/3}\left(\frac{\d_R}{\e_r}\right)^{4/3}\frac{1}{\e_r}\right)+ C(\d_R^{4/3}\d_S^{1/3}(\e_r +\d_S)^{4/3}).
\end{aligned}
\]
Using the smallness of $\d_0$ and $\e_r$, we obtain
\[
\int_0^t \sum_{i=1}^6 \|J_i\|_{L^1}^{4/3}\,d\t \leq C\d_0 + C(\e_r\d_R)^{1/6} + C\d_0 \left(\frac{\d_R}{\e_r}\right)^{4/3}\frac{1}{\e_r}.
\]
\end{proof}

\begin{lem}\label{lem:estL12}
Under the assumptions of Proposition~\ref{prop:ap}, there exist positive constants $c$ and $C$ such that the following estimates hold:
\begin{align}\label{est:ucont}
	\begin{split}
		&\begin{aligned}
			&\intRp |(u_\x^C,u_{\x\x}^C)||\vbar - v^C| \,d\x\\
            &\qquad  \leq C\d_C\left[\delta_{BL} \left(e^{-ct} + \frac{\ln (1+\d_{BL}t)}{(1+t)}\right) + (\delta_R + \delta_S)e^{-ct} + \d_S^{3/4}e^{-c\d_S t} + \frac{\delta_R}{\e_r}  e^{-c\e_r t}  \right],\\
		\end{aligned}
	\end{split}
\end{align}
\begin{align}\label{est:wder}
	\begin{split}
		&\begin{aligned}
			&\intRp |v_\x^{BL}||v_\x^R|\,d\x \leq C\delta_{BL}\frac{(\d_R\e_r)^{1/8}}{(1+t)^{7/8}},  \\ 
            &\intRp |v_\x^{BL}|\big(|v_\x^C| + |u_\x^C|\big)\,d\x \leq C\delta_{BL}\d_C \left(e^{-ct} + \frac{1}{1+\delta_{BL}t}\right),
		\end{aligned}\\
		&\begin{aligned}
			&\intRp |v_\x^{BL}||v_\x^S|\,d\x \leq  C\delta_{BL} \delta_S\left(\delta_{BL} e^{-c\d_S t} + \frac{\d_S}{1+\delta_{BL}t}\right),  \\
			& \intRp |v_\x^R|\big(|v_\x^C| + |u_\x^C|\big)\,d\x \leq C\delta_C\left(\e_r\d_R  e^{-ct} + \frac{\d_R\e_r}{\e_r} e^{-c\e_rt}\right),\\
			&\intRp |v_\x^R||v_\x^S|\,d\x \leq  C\e_r\delta_R \delta_S e^{-c\delta_St} + C\delta_S^2\frac{\delta_R\e_r}{\e_r}e^{-c\e_rt}, \\
			&  \intRp |v_\x^S| \big(|v_\x^C| + |u_\x^C|\big)\,d\x \leq C\d_S \d_C e^{-ct}.
		\end{aligned}
	\end{split}
\end{align}
Consequently, we have 
\begin{equation}\label{est:432nd}
\begin{aligned}
&\int_0^t \| |(u_\x^C,u_{\x\x}^C)||\vbar - v^C|\|_{L^1}^{4/3} + \sum_{i\neq j}\||(v_i)_\x|(|(v_j)_\x|+|(u_j)_\x|)\|_{L^1}^{4/3}\,d\t\\
&\quad \leq C\d_0\big(1+(\d_R \e_r)^{1/6}\big) + C\d_0 \left(\frac{\d_R}{\e_r}\right)^{4/3}\frac{1}{\e_r},
\end{aligned}
\end{equation}
where $U_i$, $i=1,2,3,4$ are defined in \eqref{not:vi}.
\end{lem}
Since the proof of Lemma~\ref{lem:estL12} is similar to that of Lemma~\ref{lem:win}, we defer it to the Appendix~\ref{App:wint}.
\subsection{$L^2$ Interaction Estimates}
Now, we give the $L^2$ estimates of $J_i$, where $J_i$ are defined in \eqref{def:J}. We again emphasize that the estimates established in \eqref{JiL2} of Lemma~\ref{lem:WInteraction}, as well as those in \eqref{est:wder2} of Lemma~\ref{lem:estL22}, remain valid even when the strength of the boundary layer solution is large.
\begin{lem}\label{lem:WInteraction}
Under the hypotheses as in Proposition \ref{prop:ap}, there exist positive constants $c$ and $C$ such that the following estimates hold:
\begin{equation}\label{JiL2}
	\begin{aligned}
&	\int_{\mathbb{R}_+} J_1^2 d\xi \leq C\frac{\delta_R^{\frac{3}{8}}\e_r^{\frac{1}{8}}\d_{BL}^{\frac{5}{4}}}{(1+t)^{\frac{13}{8}}}
+ C\frac{(\e_r\delta_R)^{\frac{1}{4}}\delta_{BL}}{(1+t)^{\frac{7}{4}}},  \\
&	\int_{\mathbb{R}_+} J_2^2 d\xi \leq      \delta_{BL}^2 \delta_C^2e^{-ct}+ C\delta_C^2 \delta_{BL}^2 \frac{1}{(1+\delta_{BL}t)^2}, \\
&  \int_{\mathbb{R}_+} J_3^2 d\xi  \leq C (\d_{BL}^4\d_S + \delta_{BL}^2 \delta_S^3 ) e^{-c\delta_S t} + C\delta_S^3 \delta_{BL}^2 \frac{1}{(1+\delta_{BL}t)^2},\\
&  \int_{\mathbb{R}_+} J_4^2 d\xi  \leq  C\left(\e_r^2\delta_R^2\delta_C^2 +\delta_R^2 \delta_C^2\right)  e^{-ct}+ C\delta_C^2 \frac{\delta_R^2}{\e_r} e^{-c\e_rt},\\
&  \int_{\mathbb{R}_+} J_5^2 d\xi \leq  \left(\e_r^2\delta_R^2 \delta_S+ \delta_R^2 \delta_S^3\right) e^{-c\delta_St} + C\delta_S^2 \frac{\delta_R^2}{\e_r} e^{-c\e_rt} ,\\
&  \int_{\mathbb{R}_+} J_6^2 d\xi \leq C\delta_C^2\delta_S e^{-c\delta_S t} + \delta_C^2\delta_S^2e^{-ct}, 
\end{aligned}
\end{equation}
where $J_i (i=1, 2, \cdots, 6)$ are defined as in \eqref{def:J}.

Consequently, we have 
\begin{equation}\label{sJiL2}
    \begin{aligned}
        \sum_{i=1}^6 \int_0^t \intRp J_i^2\,d\x \,d\t \leq C\d_0 + C(\d_R\e_r)^2 + C\d_0\left(\frac{\d_R}{\e_r}\right)^2.
    \end{aligned}
\end{equation}
\end{lem}

	\begin{proof}
\noindent $\bullet$ (Estimate of $J_1$) First, we split $J_1$ as  
\begin{align}\label{j1}
 \int_{\mathbb{R}_+} J_1^2 d\xi \leq 2  \int_{\mathbb{R}_+}  |v_\x^{BL}|^2|v^R - v_*|^2 d\xi  + 2 \int_{\mathbb{R}_+} |v_\x^R|^2|v^{BL} - v_*|^2 d\xi.
\end{align}
For the first term on the right-hand side of \eqref{j1}, we use H\"older's inequality to have for any $t>0$ and $y>0$, 
\[
v^R(t,y) - v_* = \int_0^y v_\xi^R\,d\xi \leq \|v_\xi^R\|_{L^{8}(\mathbb{R}_+)} y^{7/8}.
\]
By Lemma \ref{lem:rarefaction}, we have
\begin{align*}
\|v_\xi^R\|_{L^8} &\leq \|v_\xi^R\|_{L^{8}}^{\frac{1}{14}} \|v_\xi^R\|_{L^{8}}^{\frac{13}{14}} \leq C(\delta_R\e_r^{\frac{7}{8}})^{\frac{1}{14}} \left(\delta_R^{\frac{1}{8}} (1+t)^{-\frac{7}{8}})\right)^{\frac{13}{14}} \leq  C\delta_R^{\frac{3}{16}}\e_r^{\frac{1}{16}} \left(1+t\right)^{-\frac{13}{16}}.
\end{align*}
Thus, we obtain 
\begin{align*}
	\int_{\mathbb{R}_+} |v_\xi^{BL}|^2|v^R-v_*|^2\,d\xi &\leq \|v_\xi^R\|_{L^{8}(\mathbb{R}_+)}^2 \int_{\mathbb{R}_+} \xi^{7/4} |v_\xi^{BL}|^2\,d\xi\\
	&\leq \frac{C\delta_R^{\frac{3}{8}}\e_r^{\frac{1}{8}}}{(1+t)^{\frac{13}{8}}} \int_{\mathbb{R}_+} \xi^{7/4} \frac{\delta_{BL}^4}{(1+\delta_{BL}\xi)^4}\,d\xi\le \frac{C\delta_R^{\frac{3}{8}}\e_r^{\frac{1}{8}}\d_{BL}^{\frac{5}{4}}}{(1+t)^{\frac{13}{8}}}.
\end{align*}
For the second term of \eqref{j1}, it follows from \eqref{est:BLR} and Lemma~\ref{lem:BL} that
\begin{align*}
\int_{\mathbb{R}_+} |v_\x^R|^2|v^{BL} - v_*|^2 d\xi  & \leq C\frac{(\e_r\delta_R)^{\frac{1}{4}}}{(1+t)^{\frac{7}{4}}} \int_{\mathbb{R}_+} \frac{\delta_{BL}^2}{\left(1+\delta_{BL} \xi\right)^2} d\xi \leq C\frac{(\e_r\delta_R)^{\frac{1}{4}}\delta_{BL}}{(1+t)^{\frac{7}{4}}}.
\end{align*}
Therefore, we have
\begin{align*}
 \int_{\mathbb{R}_+} J_1^2 d\xi \leq C\frac{\delta_R^{\frac{3}{8}}\e_r^{\frac{1}{8}}\d_{BL}^{\frac{5}{4}}}{(1+t)^{\frac{13}{8}}} + C\frac{(\e_r\delta_R)^{\frac{1}{4}}\delta_{BL}}{(1+t)^{\frac{7}{4}}}.
\end{align*}
\noindent $\bullet$ (Estimate of $J_2$) We split $J_2$ as 
\begin{align}\label{j2}
  \int_{\mathbb{R}_+} J_2^2 d\xi \leq     2\int_{\mathbb{R}_+}|v_\x^{BL}|^2 |v^C - v_m|^2d\xi+   2\int_{\mathbb{R}_+} |v_\x^C|^2|v^{BL} - v_*|^2 d\xi .
\end{align}
For the first term on the right-hand side of \eqref{j2}, we use 
Lemma \ref{lem:BL} and \eqref{est:Cin} to have 
\begin{align*}
\begin{aligned}
&\int_{\mathbb{R}_+}|v_\x^{BL}|^2 |v^C - v_m|^2d\xi = \left(\int_{0}^{-\frac{\sigma_- t}{2}} + \int_{-\frac{\sigma_- t}{2}}^{\infty}\right) |v_\x^{BL}|^2 |v^C - v_m|^2 d\xi  \\
&\quad \leq  C\delta_{BL}^4 \delta_C^2 \int_{0}^{-\frac{\sigma_- t}{2}} e^{-\frac{C_1|\xi+\sigma_- t|^2}{1+t}} e^{-Ct} d\xi + C\delta_{BL}^4 \delta_C^2  \int_{-\frac{\sigma_- t}{2}}^{\infty} \frac{1}{(1+\delta_{BL} \xi)^4} d\xi  \\
&\quad \leq C\delta_{BL}^4 \delta_C^2e^{-ct} + C\delta_{BL}^3 \delta_C^2\frac{1}{(1+\delta_{BL} t)^3}.
\end{aligned}
\end{align*}
Similarly, by Lemma \ref{lem:BL} and \eqref{est:Cin}, we estimate the second term in \eqref{j2} as 
\begin{align*}
\begin{aligned}
&\int_{\mathbb{R}_+} |v_\x^C|^2|v^{BL} - v_*|^2 d\xi = \left(\int_{0}^{- \frac{\sigma_- t}{2}} + \int_{-\frac{\sigma_- t}{2}}^{\infty}\right) |v_\x^C|^2|v^{BL} - v_*|^2 d\xi  \\
&\quad \leq  C\delta_{BL}^2 \delta_C^2 \int_{0}^{-\frac{\sigma_- t}{2}}  e^{-\frac{C_1|\xi+\sigma_- t|^2}{1+t}} e^{-ct} d\xi + C\delta_C^2 \frac{1}{1+t}\int_{-\frac{\sigma_- t}{2}}^{\infty}\frac{\delta_{BL}^2}{(1+\delta_{BL} \xi)^2}  e^{-\frac{2C_1|\xi+\sigma_- t|^2}{1+t}} d\xi \\
&\quad \leq C\delta_{BL}^2 \delta_C^2 e^{-Ct} + C\delta_C^2\delta^2_{BL}\frac{1}{(1+\delta_{BL}t)^{2}}\frac{1}{1+t}\int_{-\frac{\sigma_- t}{2}}^{\infty}  e^{-\frac{2C_1|\xi+\sigma_- t|^2}{1+t}} d\xi \\
&\quad \leq C\delta_{BL}^2 \delta_C^2 e^{-Ct} + C\delta_C^2\delta^2_{BL}\frac{1}{(1+\delta_{BL}t)^{2}}.
\end{aligned}
\end{align*}


Combining the above estimates, we have
\begin{align*}
 \int_{\mathbb{R}_+} J_2^2 d\xi \leq  C\delta_{BL}^2 \delta_C^2e^{-Ct}+ C\delta_C^2 \delta_{BL}^2 \frac{1}{(1+\delta_{BL}t)^2}.
\end{align*}
 \noindent $\bullet$ (Estimate of $J_3$) We split $J_3$ as
 \begin{align}\label{j3}
  \int_{\mathbb{R}_+} J_3^2 d\xi \leq 2   \int_{\mathbb{R}_+} |v_\x^{BL}|^2 |v^S - v^*|^2 d\xi + 2 \int_{\mathbb{R}_+} |v_\x^S|^2 |v^{BL} - v_*|^2 d\xi.
 \end{align}
For the first term of \eqref{j3}, using Lemma~\ref{lem:BL}, Lemma \ref{lem:VS}, and \eqref{est:vsr}, we have 
\begin{align*}
	\begin{aligned}
\int_{\mathbb{R}_+} |v_\x^{BL}|^2 |v^S - v^*|^2 d\xi  &\leq \left(\int_{0}^{bt} + \int_{bt}^{\infty}\right)  |v_\x^{BL}|^2 |v^S - v^*|^2 d\xi   \\
&\leq C\delta_{BL}^4 \delta_S e^{-C\delta_S t} + C\delta_{BL}^3 \delta_S^2 \frac{1}{(1+\delta_{BL}t)^3}.
	\end{aligned}
\end{align*}
Similarly, the second term of \eqref{j3} can be estimated as follows:
\begin{align*}
    \intRp |v_\x^S|^2|v^{BL} - v_*|^2\,d\x  &=  \left( \int_{0}^{bt} +\int_{bt}^{\infty}\right)|v_\x^S|^2|v^{BL} - v_*|^2\,d\x\\
    &\leq C\d_{BL}^2\d_S^3e^{-C\d_St} +\int_{bt}^{\infty}|v_\x^S|^2|v^{BL} - v_*|^2\,d\x,
\end{align*}
where $b = \frac{(\s - \s_-)}{2}$. Now, we use Fubini's theorem to have 
\begin{equation*}
\begin{aligned}
   	\int_{bt}^{\infty}  |(v^S)_\xi|^2 |v^{BL} - v_*|^2 d\xi  = &- \int_{bt}^{\infty}   |(v^S)_\xi^{-X} |^2 \int_{\xi}^{\infty} \left(|v^{BL} - v_*|^2\right)_y (y) dyd\xi  \\
		=& -2  \int_{bt}^{\infty} v^{BL}_y (v^{BL}-v_*)(y) \int_{bt}^{y} |(v^S)_\xi|^2 \, d\xi \, dy \\
		\leq & 2\left(\int_{-\infty}^{\infty} |(v^S)_\xi|^2 d\xi\right) \left(\int_{bt}^{\infty} |v^{BL}_y| |v^{BL} -v_*| dy\right)  \\
		\leq & C\delta^3_S \int_{bt}^{\infty} \frac{\delta_{BL}^3}{(1+\delta_{BL} y)^3} dy \\
        \leq & C\delta_S^3 \delta_{BL}^2 \frac{1}{(1+\delta_{BL}t)^2}. 
\end{aligned}
\end{equation*}
Therefore, we obtain
\begin{equation}\label{l2J3}
\intRp |v_\x^S|^2|v^{BL} -v_*|^2\,d\x \leq C\d_{BL}^2\d_S^3e^{-C\d_St} + C\d_S^3\d_{BL}^2\frac{1}{(1+\d_{BL}t)^2}.    
\end{equation}
Thus, we have 
\begin{align*}
 \int_{\mathbb{R}_+} J_3^2 d\xi  \leq C (\d_{BL}^4\d_S + \delta_{BL}^2 \delta_S^3 ) e^{-C\delta_S t} + C\delta_S^3 \delta_{BL}^2 \frac{1}{(1+\delta_{BL}t)^2}.
\end{align*}
 \noindent $\bullet$ (Estimate of $J_4$) We split $J_4$ as
 \begin{align*}
 \int_{\mathbb{R}_+} J_4^2 d\xi 
 &\,\, \leq  2\int_{0\leq \xi \leq  \frac{1}{2}\lambda_1(u_m, \theta_m)(1+t)-\sigma_-t} \big(|v_\x^{R}|^2 |v^C - v_m|^2 + |v_\x^C|^2 |v^{R} - v_m|^2\big) \,d\x\\
 &\,\, \quad +2\int_{\xi \geq  \frac{1}{2}\lambda_1(u_m, \theta_m)(1+t)-\sigma_-t}\big(|v_\x^{R}|^2 |v^C - v_m|^2 + |v_\x^C|^2 |v^{R} - v_m|^2\big)d\xi.
 \end{align*}
 By Lemma \ref{lem:rarefaction} and \eqref{est:Cin}, we obtain
 \begin{align*}
\begin{aligned}
&\int_{0\leq \xi \leq  \frac{1}{2}\lambda_1(u_m, \theta_m)(1+t)-\sigma_-t}  \left(|v_\x^{R}|^2 |v^C - v_m|^2 + |v_\x^C|^2 |v^{R} - v_m|^2\right)d\xi  \\
&\quad \leq C\left((\e_r\delta_R)^2\delta_C^2 +\delta_R^2 \delta_C^2\right)  e^{-Ct}.
\end{aligned}
 \end{align*}
	On the other hand, it follows from Lemma \ref{lem:rarefaction}, Lemma \ref{lem:CD} and \eqref{est:Rdecay} that
\begin{align*}
\begin{aligned}
&\int_{\xi \geq  \frac{1}{2}\lambda_1(u_m, \theta_m)(1+t)-\sigma_-t}   \left(|v_\x^{R}|^2 |v^C - v_m|^2 + |v_\x^C|^2 |v^{R} - v_m|^2\right)d\xi  \\
&\,\,\leq C\delta_C^2 \delta
_R^2 (1+\e_r^2)\int_{\xi \geq  \frac{1}{2}\lambda_1(u_m, \theta_m)(1+t)-\sigma_-t}  e^{-C\e_r|\xi +\sigma_- t - \frac{1}{2}\lambda_1(u_m, \theta_m)(1+t)|} e^{-C\e_r| \lambda_1(u_m, \theta_m)|(1+t)} d\xi \\
&\,\,\leq C\delta_C^2 \frac{\delta_R^2(1+\e_r^2)}{\e_r}e^{-C\e_rt}.
\end{aligned}
\end{align*}
Therefore, together with the fact that $\e_r\leq 1$, we have
 \begin{align*}
 \int_{\mathbb{R}_+} J_4^2 d\xi  \leq  C\left(\e_r^2\delta_R^2\delta_C^2 +\delta_R^2 \delta_C^2\right)  e^{-Ct}+ C\delta_C^2 \frac{\delta_R^2}{\e_r}e^{-C\e_rt}.
 \end{align*}
  \noindent $\bullet$ (Estimate of $J_5$) First, we decompose $J_5$ as follows:
  \begin{align*}
 &\int_{\mathbb{R}_+} J_5^2 d\xi  \leq   2\int_{\mathbb{R}_+} \left(
 |v_\x^{R}|^2 |v^S - v^*|^2 + |v_\x^S|^2|v^{R} - v_m|^2 \right) d\xi \\
 &\leq  C\left(\int_{0\le\xi \leq \frac{1}{2}\lambda_1(u_m,\theta_m)(1+t)-\s_-t} + \int_{\xi \ge \frac{1}{2}\lambda_1(u_m,\theta_m)(1+t)-\s_-t}\right)|v_\x^{R}|^2 |v^S - v^*|^2 d\xi\\
 &\quad +C\left(\int_{0\le\xi \leq \frac{1}{2}\lambda_1(u_m,\theta_m)(1+t)-\s_-t} + \int_{\xi \ge \frac{1}{2}\lambda_1(u_m,\theta_m)(1+t)-\s_-t}\right) |v_\x^S|^2|v^{R} - v_m|^2  d\xi.
  \end{align*}
  Since $\frac{1}{2}\lambda_1(u_m,\th_m)(1+t) < \beta + \sigma t$, by Lemma~\ref{lem:VS} and \eqref{est:vsr}, we deduce that
\[
\begin{aligned}
   &\int_{0\leq \xi  \leq  \frac{1}{2}\lambda_1(u_m, \theta_m)(1+t)-\sigma_-t} |v_\x^{R}|^2 |v^S - v^*|^2 d\xi + \int_{\xi  \geq  \frac{1}{2}\lambda_1(u_m, \theta_m)(1+t)-\sigma_-t} |v_\x^{R}|^2 |v^S - v^*|^2 d\xi\\
   &\quad \leq C(\e_r\d_R)^2\d_Se^{-C\d_S t} +C\d_S^2\int_{\xi  \geq  \frac{1}{2}\lambda_1(u_m, \theta_m)(1+t)-\sigma_-t} |v_\x^R|^2\,d\xi\\
   &\quad \leq C(\e_r\d_R)^2\d_Se^{-C\d_S t} + C\d_R^2\e_re^{-C\e_r t},
\end{aligned}
\]
and similarly,
\begin{equation}\label{l2J5}
    \begin{aligned}
        &\int_{0\le \xi  \leq  \frac{1}{2}\lambda_1(u_m, \theta_m)(1+t)-\sigma_-t}  |v_\x^S|^2|v^{R} - v_m|^2  d\xi+ \int_{\xi  \geq  \frac{1}{2}\lambda_1(u_m, \theta_m)(1+t)-\sigma_-t} |v_\x^S|^2|v^{R} - v_m|^2  d\xi\\
        &\quad \leq C\d_R^2\d_S^3e^{-C\d_S t} + C\d_S^4\int_{\xi  \geq  \frac{1}{2}\lambda_1(u_m, \theta_m)(1+t)-\sigma_-t} |v^R - v_m|^2\,d\xi\\
        &\quad \leq C\d_R^2\d_S^3e^{-C\d_S t} + C\d_S^4 \frac{\d_R^2}{\e_r}e^{-C\e_r t}.
    \end{aligned}
\end{equation}
Combining the above estimates, we have
\begin{align*}
 \int_{\mathbb{R}_+} J_5^2 d\xi \leq C\left(\e_r^2\delta_R^2 \delta_S+ \delta_R^2 \delta_S^3\right) e^{-C\delta_St} + C\delta_S^2 \frac{\delta_R^2}{\e_r}  e^{-C\e_rt} .
\end{align*}
 
 \noindent $\bullet$ (Estimate of $J_6$) Let $\alpha=\frac{1}{2}(-\s_- + (\s - \s_-)) = \frac{\sigma-2\sigma_-}{2}$. Then, we have \
 \[
 \begin{aligned}
      \int_{\mathbb{R}_+} J_6^2 d\xi & \leq 2\left( \int_{0}^{\alpha t} +\int_{\alpha t}^{\infty}\right) \left( |v_\x^{C}|^2 |v^S - v^*|^2 + |v_\x^S|^2|v^{C} - v^*|^2\right)\,d\x
 \end{aligned}
 \]
  Using Lemma \ref{lem:VS} and Lemma \ref{lem:CD}, we obtain 
  \[
  \begin{aligned}
        &\left( \int_{0}^{\alpha t} +\int_{\alpha t}^{\infty}\right) |v_\x^{C}|^2 |v^S - v^*|^2\,d\x \\
        &\quad \leq C\d_C^2\d_S^2\int_0^{\alpha t} e^{-C\d_S |\x - (\s - \s_-)t - X(t) - \b|}\,d\x + C\d_C^2 \d_S^2\int_{b'  t}^\infty e^{\frac{-C_1|\x  + \s_- t|^2}{1+t}}\,d\x\\
        &\quad \leq C\d_C^2 \d_S e^{-C\d_S t} + C\d_C^2 \d_S^2 e^{-Ct},
  \end{aligned}
  \]
  and similarly,
  \begin{equation}\label{l2J6}
      \begin{aligned}
          \left( \int_{0}^{\alpha t} +\int_{\alpha t}^{\infty}\right) |v_\x^{S}|^2 |v^C - v^*|^2\,d\x \leq \d_S^3\d_C^2 e^{-C\d_S t} + \d_S^4 \d_C^2 e^{-Ct}.
      \end{aligned}
  \end{equation}
  Thus, we obtain
  \[
  \intRp J_6^2\,d\x \leq C\d_C^2 \d_S e^{-C\d_S t} + C\d_C^2\d_S^2e^{-Ct}.
  \]
  $\bullet$ (Proof of \eqref{sJiL2}) From \eqref{JiL2}, we have
  \[
  \begin{aligned}
         \sum_{i=1}^6 \int_0^t \intRp J_i^2\,d\x \,d\t &\leq C(\d_R^{3/8}\e_r^{1/8}\d_{BL}^{5/4} + (\d_R\e_r)^{1/4}\d_{BL}) + C(\d_{BL}^2\d_C^2 + \d_{BL}\d_C^2) \\
         &\quad +C(\d_{BL}^4 + \d_{BL}^2\d_S^2 + \d_S^3\d_{BL}) + C\left(\d_R^2\e_r^2\d_C^2 + \d_R^2\d_C^2 + \d_C^2\left(\frac{\d_R}{\e_r}\right)^2\right)\\
         &\quad + C(\d_R^2\e_r^2+ \d_R^2\d_S^2) + C\d_S^2\left(\frac{\d_R}{\e_r}\right)^2 + C(\d_C^2 + \d_C^2\d_S).
  \end{aligned}
  \]
  Using the smallness of $\d_{BL},\d_C,\d_S$, and $\e_r$, we obtain \eqref{sJiL2}.
  \end{proof}
\begin{lem}\label{lem:estL22}
    Under the assumptions of Proposition~\ref{prop:ap}, there exist positive constants $c$ and $C$ such that the following estimates hold:
\begin{equation}\label{est:wder2}
	\begin{aligned}
		&	\int_{\mathbb{R}_+} (|u^C_\xi|^2 + |u^C_{\x\x}|^2) |\bar{v}-v^C|^2 d\xi \\
        &\quad  \leq C\left(\delta_{BL}^2\delta_C^2  +\delta_R^2 \delta_C^2 + \delta_C^2 \delta_S^2\right)e^{-ct} + C\d_C^2\d_{BL}^2\frac{1}{(1+\d_{BL}t)^2} + C\delta_C^2\delta_Se^{-c\delta_St} + C\delta_C^2 \frac{\delta_R^2}{\e_r} e^{-c\e_rt},\\
		& \int_{\mathbb{R}_+} |v_\xi^{BL}|^2 |v^R_\xi|^2 d\xi \leq C\e_r^{\frac{1}{4}} \delta_R^{\frac{1}{4}} \delta_{BL}^3 \frac{1}{(1+t)^{\frac{7}{4}}}, \\
		&\int_{\mathbb{R}_+} |v_\xi^{BL}|^2 |v^S_\xi|^2 d\xi \leq C\delta_{BL}^4 \delta_S^3e^{-c\delta_St} + C\delta_{BL}^3\delta_S^4 \frac{1}{(1+\delta_{BL}t)^3} ,\\
		& \int_{\mathbb{R}_+} |v_\xi^{R}|^2 |v^S_\xi|^2 d\xi \leq C\e_r^2\delta_R^2 \delta_S^3 e^{-c\delta_S t} + C\delta_S^4 \frac{\delta_R^2}{\e_r} e^{-c\e_rt}, \\
		& \int_{\mathbb{R}_+} |v_\xi^{BL}|^2 \left(|v^C_\xi|^2 +|u^C_\xi|^2 \right) d\xi \leq C\delta_{BL}^4\delta_C^2e^{-ct} + C\delta_{BL}^3\delta_C^2 \frac{1}{(1+\delta_{BL}t)^3},  \\
		& \int_{\mathbb{R}_+} |v_\xi^{R}|^2 \left(|v^C_\xi|^2 +|u^C_\xi|^2 \right) d\xi \leq  C\e_r^2 \delta_R^2 e^{-ct} + C\delta_C^2 \frac{\delta_R^2}{\e_r}  e^{-c\e_rt},  \\
		& \int_{\mathbb{R}_+} |v_\xi^{S}|^2 \left(|v^C_\xi|^2 +|u^C_\xi|^2 \right) d\xi \leq C\delta_C^2 \delta_S^3e^{-c\delta_St} + C\delta_C^2\delta_S^4e^{-ct}.
	\end{aligned}
\end{equation}
Consequently, we have 
\begin{equation}\label{sJiL22}
\begin{aligned}
    &\int_0^t \intRp |(u_\x^C,u_{\x\x}^C)|^2|\vbar - v^C|^2\,d\x\,d\t +C\sum_{i \neq j} \int_0^t \intRp |(u_i)_\x|^2(|(v_j)_\x| + |(u_j)_\x|)^2\,d\x\,d\t\\
    &\quad \leq    C\d_0 + C(\d_R\e_r)^2 + C\d_0\left(\frac{\d_R}{\e_r}\right)^2.
\end{aligned}
\end{equation}
\end{lem}
\begin{proof}
    Since the proof of Lemma~\ref{lem:estL22} is similar to that of Lemma~\ref{lem:estL12}, we omit it.
\end{proof}

\section{Zeroth Order Estimates}\label{Sec:L2}
\setcounter{equation}{0}
In this section, we prove $L^2$ estimates, which is based on the method of $a$-contraction with shifts.
\begin{lem}\label{lem:L2ap}
Under the assumptions of Proposition~\ref{prop:ap}, there exist positive constants $c$ and $C$ such that
\begin{align*}
\begin{aligned}
&\sup_{t \in [0,T]}\norm{(U-\bar U)(t,\cdot)}_{L^2} + \delta_S\int_0^T |\dot{X}|^2 \,dt + \int_0^T (G^{BL} + G^R + G^S +D_{u_1}+ D_{\th_1})\,dt \\
&\quad \leq  C\norm{U_0-\bar U(0,\cdot)}_{L^2} + C(\e^2+\d_C) \int_0^T \left(D_{u_2} + D_{\theta_2}\right) \,dt +C(\e^2 + \delta_0)\int_0^T D_{v_1}\,dt \\
&\qquad + C\d_0^{1/40} + C e^{-c\d_S \beta}.
\end{aligned}
\end{align*}
\end{lem}
\subsection{Relative Entropy Estimates}
First, recall that $\Ubar$ is the superposition of the boundary layer, rarefaction, contact wave and shock wave:
\begin{equation*}
    \Ubar(t,\x) = U^{BL}(\x) + U^R(t,\x) + U^C(t,\x) + U^S(\x - (\s -\s_-)t - X(t)-\beta) - U_* - U_m - U^*.
\end{equation*}
Then, from \eqref{eq:BL}, \eqref{eq:R}, \eqref{eq:contact}, and \eqref{eq:VS}, $\Ubar$ satisfies the following equations:
\begin{equation}\label{eq:composite}
	\begin{cases}
		&\vbar_t - \s_- \vbar_\x - \ubar_\x = -\dot{X}(t)v^S_\x, \quad t>0, \quad \x >0, \\
		&\ubar_t - \s_- \ubar_\x + \pbar_\x = \left(\mu \frac{\ubar_\x}{\vbar}\right)_\x + Q_1 - \dot{X}(t)u_\x^S,\\
		& \frac{R}{\gamma-1}\big(\thbar_t - \s_- \thbar_\x \big) +   \pbar \ubar_\x = \left(\kappa \frac{\thbar_\x}{\vbar}\right)_\x + \left(\mu \frac{(\ubar_\x)^2}{\vbar}\right) + Q_2-\frac{R}{\gamma-1}\dot{X}(t)\th^S_\x.
	\end{cases}
\end{equation}
Here, $\pbar:=\frac{R\thbar}{\vbar}$ and
\begin{equation}\label{dec:Qi}
    Q_i := Q_i^I + Q_i^R + Q_i^C, \quad i=1,2,
\end{equation}
where $Q_i^I$ are error terms related to the wave interaction terms:
\begin{equation}\label{def:QI}
	\begin{split}
		&\begin{aligned}
			Q_1^I := \big(\pbar - p^{BL} - p^R - p^C - p^S)_\x - \m \left(\frac{\ubar_\x}{\vbar} - \frac{u^{BL}_\x}{v^{BL}}  - \frac{u^R_\x}{v^R}  - \frac{u^C_\x}{v^C}  - \frac{u^S_\x}{v^S}  \right)_\x,
		\end{aligned}\\
		&\begin{aligned}
			Q_2^I &:= \Big( \pbar \ubar_\x -p^{BL} u^{BL}_\x - p^R u_\x^R - p^C u_\x^C - p^S_\x u^S_\x) - \kappa \left(\frac{\thbar_\x}{\vbar} - \frac{\th^{BL}_\x}{v^{BL}}  - \frac{\th^R_\x}{v^R}  - \frac{\th^C_\x}{v^C}  - \frac{\th^S_\x}{v^S} \right)_\x\\
			&\quad -\m \left(\frac{\ubar_\x^2}{\vbar} - \frac{(u^{BL}_\x)^2}{v^{BL}}  - \frac{(u^R_\x)^2}{v^R}  - \frac{(u^C_\x)^2}{v^C}  - \frac{(u^S_\x)^2}{v^S}  \right)
		\end{aligned}
	\end{split}
\end{equation}
and $Q_i^R$ are error terms due to the rarefaction wave
\begin{equation}\label{def:QR}
	Q_1^R := -\m \left(\frac{u_\x^R}{v^R}\right)_\x, \quad Q_2^R := -\kappa \left(\frac{\th_\x^R}{v^R}\right)_\x - \m \frac{(u_\x^R)^2}{v^R}.
\end{equation}
For $Q_i^C$, recall from \eqref{est:QC} that 
\[
Q_1^C = O(1)\d_C (1+t)^{-\frac{3}{2}}e^{-2C_1 \frac{(\x + \s_- t)^2}{1+t}}, \quad Q_2^C = O(1)\d_C (1+t)^{-2}e^{-2C_1 \frac{(\x + \s_- t)^2}{1+t}}.
\]
To obtain the $L^2$ energy estimates, we use the method of $a$-contraction theory with shifts, which is based on the relative entropy method. First of all, the (mathematical) entropy is given by 
\[
\eta:= - s = \frac{R}{\gamma - 1}\ln \left(\frac{R}{A} \th v^{\gamma -1}\right) = -R\ln v -\frac{R}{\gamma -1}\ln \th + const.
\]
Then, the relative entropy is given by
\[
\eta(U|\Ubar) = R\Phi\left(\frac{v}{\vbar}\right) + \frac{R}{\gamma - 1}\Phi \left(\frac{\th}{\thbar}\right) + \frac{(u-\ubar)^2}{2\thbar},
\]
where $\Phi(z) := z - 1 - \ln z$ (see \cite{KVW-NSF}).
Thus, the relative entropy weighted by $\thbar$ is given by
\begin{align}\label{def:RE}
\thbar \eta(U|\Ubar) = R\thbar\Phi\left(\frac{v}{\vbar}\right) + \frac{R}{\gamma - 1}\thbar \Phi \left(\frac{\th}{\thbar}\right) + \frac{(u-\ubar)^2}{2}.
\end{align}

\begin{lem}\label{lem:rel}
Let $U$ be a solution to \eqref{eq:NSF} and let $\Ubar$ be the superposed wave defined by \eqref{def:Ubar}. Let $a=a(t,\x)$ be the weight function defined in \eqref{def:a}. Then, we have
\begin{equation}\label{est:rel}
	\begin{aligned}
		&\frac{d}{dt} \intRp a(t,\x) \bar{\theta}(t,\x) \eta(U(t,\x)|\Ubar(t,\x))\,d\x \\
		&\quad = \dot{X}(t) Y(U) + \mathcal{J}^{H}(U)+ \mathcal{J}^{P}(U) +\mathcal{J}^{I}(U) - \mathcal{J}^{good}(U)+\mathcal{J}^{bd}(U),	
	\end{aligned}
\end{equation}
where
\begin{equation*}
	\begin{aligned}
		Y(U) &= -\intRp a_{\xi} \thbar \eta(U|\Ubar)\,d\x + \intRp a \left[ -R(\th^S)_\x \Phi\left(\frac{v}{\vbar}\right) - \frac{R}{\gamma - 1}(\th^S)_\x \Phi\left(\frac{\th}{\thbar}\right)\right]\,d\x\\
		&\quad +\intRp a \left[(u^S)_\x (u-\ubar) + \frac{(v^S)_\x \pbar}{\vbar}(v-\vbar) + \frac{R}{\gamma - 1}\frac{(\th^S)_\x}{\thbar}(\th - \thbar) \right]\,d\xi,
	\end{aligned}
\end{equation*}
\begin{align}\label{Jbad}
	\begin{split}
		&\begin{aligned}
		\mathcal{J}^{H}&:= \intRp a_\x (u-\ubar)(p-\pbar)\,d\x +\intRp a\Big(\thbar_t - \s_- \thbar_\x \Big)\left[R\Phi\left(\frac{v}{\vbar}\right) + \frac{R}{\gamma-1}\Phi\left(\frac{\th}{\thbar}\right) \right]d\x\\
		&\quad +\intRp a \left[-\frac{\pbar \ubar_\x}{v \vbar}(v-\vbar)^2 - \frac{\ubar_\x}{\th}(\th - \thbar)(p - \pbar) + \pbar \ubar_\x \frac{(\th - \thbar)^2}{\th \thbar}\right]d\x,
	\end{aligned}\\
	&\begin{aligned}
		\mathcal{J}^{P}&:= -\intRp a_\x \left[\m (u-\ubar)\left(\frac{u_\x}{v} - \frac{\ubar_\x}{\vbar}\right) + \kappa \frac{\th -\thbar}{\th} \left(\frac{\th_\x}{v} - \frac{\thbar_\x}{\vbar}\right)\right]\,d\x\\
		&\quad  +\intRp a \left[ -\m \ubar_\x \left(\frac{1}{v} - \frac{1}{\vbar}\right)(u-\ubar)_\x + \kappa \frac{\th - \thbar}{\th^2}\th_\x \left(\frac{\th_\x}{v} - \frac{\thbar_\x}{\vbar}\right) - \kappa \frac{(\th - \thbar)_\x}{\th}\thbar_\x \left(\frac{1}{v} - \frac{1}{\vbar}\right) \right.\\
		&\phantom{+\intRp a \Big[ \quad } \left.+\m \frac{\th - \bar{\theta}}{\th}\left(\frac{u_\x^2}{v} - \frac{\ubar_\x^2}{\vbar}\right) - \frac{(\th - \thbar)^2}{\th \thbar}\left(\kappa \left(\frac{\thbar_\x}{\vbar}\right)_\x + \m \frac{\ubar_\x^2}{\vbar}\right)\right]\,d\x\\
	\end{aligned}\\
	&\begin{aligned}
		\mathcal{J}^{I}:= - \intRp a Q_1 (u - \ubar) - a Q_2\left(\frac{\th}{\thbar} - 1\right)\,d\x,
	\end{aligned}
	\end{split}
\end{align}
and
\begin{equation*}
	\mathcal{J}^{good} := \s \intRp a_\x \thbar \eta(U|\Ubar)\,d\x + \intRp a \left[ \frac{\m}{v}|(u-\ubar)_\x|^2 + \frac{\kappa}{v\th}|(\th - \thbar)_\xi|^2\right]\,d\xi,
\end{equation*}
and
\begin{equation}\label{Jbd}
	\begin{aligned}
			\mathcal{J}^{bd} &:= -\s_- a \thbar\eta(U|\Ubar)\Big|_{\x=0} + a(u-\ubar)(p-\pbar) \Big|_{\x=0}\\
			&\,\quad \left.- a\left[\m (u-\ubar)\left(\frac{u_\x}{v} - \frac{\ubar_\x}{\vbar}\right) + \kappa \frac{\th -\thbar}{\th} \left(\frac{\th_\x}{v} - \frac{\thbar_\x}{\vbar}\right)\right] \right|_{\x=0}
	\end{aligned}
\end{equation}
\end{lem}
\begin{rem}
Here, $\mathcal{J}^{H}$ and $\mathcal{J}^{P}$ represent the terms arising from the hyperbolic and parabolic parts of the NSF system, respectively, and $\mathcal{J}^{I}$ denotes the interaction terms between the boundary layer solution and the elementary waves.
\end{rem}
\begin{rem}
From \eqref{def:a}, we have $ \sigma  a_\x> 0$. Therefore, the term 
\[
\s \intRp a_\x \thbar \eta(U|\Ubar)\,d\x,
\]
appearing in $\mathcal{J}^{good}$ is indeed a good term.
\end{rem}

\begin{proof}
	Since the proof of Lemma \ref{lem:rel} is almost the same as \cite[Lemma 4.3]{KVW-NSF} or \cite[Lemma 4.4]{HL25}, we omit the proof here.
\end{proof}
\subsection{Decompositions}
Here, we further decompose the right-hand side of \eqref{est:rel} into good and bad terms. In what follows, the notation $U_i$, $i=1,2,3,4$ introduced in \eqref{not:vi} and the corresponding wave notation $U^{BL},U^R,U^C$, or $U^S$ will both be used.
\subsubsection{Decomposition of $\mathcal{J}^{H}$}
First, based on the identity
\begin{equation}\label{id:p}
	(p - \pbar) = \frac{R}{v}(\th - \thbar) - \frac{\pbar}{v}(v-\vbar),
\end{equation}
and $\|(v-\vbar, u-\ubar, \th - \thbar)\|_{L^\infty} \leq C\e$, the second line of $\mathcal{J}^{H}$ can be estimated as
\begin{align*}
	&\intRp a \left[-\frac{\pbar}{v \vbar}(v-\vbar)^2 - \frac{1}{\th}(\th - \thbar)(p - \pbar) + \pbar  \frac{(\th - \thbar)^2}{\th \thbar}\right]\ubar_\x \, d\x\\
 &\quad \leq \intRp a \ubar_\x\left[-\frac{\pbar}{\vbar^2}|\phi|^2 + \frac{\pbar}{\vbar\thbar}\phi \chi\right] \, d\x + C\e \intRp |\ubar_\x| \big|\big(\phi,\chi \big)\big|^2\,d\x.
\end{align*}
Now, $\mathcal{J}^{H}$ can be decomposed into following terms:
\begin{equation}\label{dec:Jhyp}
	\mathcal{J}^{H} \leq \sum_{i=1}^4 \mathcal{J}^{H}_i + C \e \intRp |\ubar_\x| \big|\big(\phi,\chi \big)\big|^2\,d\x,
\end{equation}
where 
\begin{equation}\label{Jhyp}
	\begin{aligned}
	\mathcal{J}^{H}_i &:= \intRp a\Big(\rd_t \th_i- \s_- \rd_{\x} \th_i \Big)\left[R\Phi\left(\frac{v}{\vbar}\right)
	+ \frac{R}{\gamma-1}\Phi\left(\frac{\th}{\thbar}\right) \right]\,d\x\\
	&\qquad  +\intRp a \left[-\frac{\pbar}{\vbar^2}|\phi|^2 + \frac{\pbar}{\vbar \thbar}\phi \chi\right] \rd_\x u_i \, d\x,
\end{aligned}
\end{equation}
for $i=1,2,3$, and 
\begin{align*}
	\mathcal{J}^{H}_4 &:=  \intRp a_\x (u-\ubar)(p-\pbar)\,d\x + \intRp a\Big(\rd_t \th_4- \s_- \rd_{\x} \th_4 \Big)\left[R\Phi\left(\frac{v}{\vbar}\right)
	+ \frac{R}{\gamma-1}\Phi\left(\frac{\th}{\thbar}\right) \right]\,d\x\\
	&\qquad  +\intRp a \left[-\frac{\pbar}{\vbar^2}|\phi|^2 + \frac{\pbar}{\vbar \thbar}\phi \chi\right] \rd_\x u_4 \, d\x.
\end{align*}
\noindent$\bullet$ {\bf Decomposition of $\mathcal{J}^{H}_1$:} Using $|(\phi,\chi)| \leq C\e$ and by Taylor expansion of $\Phi$ at $z=1$ with $\Phi(1)=\Phi'(1)=0, 
\Phi''(1)=1$, we find that 
\begin{equation}\label{est:TE}
	\Phi\left(\frac{v}{\vbar}\right) = \frac{(v-\vbar)^2}{2\vbar^2} + O\big(|v-\vbar|^3 \big), \quad \Phi\left(\frac{\th}{\thbar}\right) = \frac{(\th-\thbar)^2}{2\thbar^2} + O\big(|\th-\thbar|^3 \big).
\end{equation}
In addition, we observe from \eqref{eq:BL}$_3$ and Lemma \ref{lem:BL} that
\begin{equation}\label{est:BLth}
	\begin{aligned}
	\rd_t \th_1 - \s_- \rd_\x \th_1  &= \frac{\gamma - 1}{R}\left(-p_1 \rd_\x u_1 +\kappa \left(\frac{\rd_\x \th_1}{v_1}\right)_\x + \m \frac{(\rd_\x u_1)^2}{v_1} \right)\\
	&\leq  -\frac{\gamma - 1}{R}\pbar \rd_\x u_1 + C(|\pbar-p_1| + \d_{BL})|\rd_\x u_1|.
\end{aligned}
\end{equation}
Here, we use the notation $p_1=p(v_1,\th_1)$.  Substituting \eqref{est:TE} and \eqref{est:BLth} into \eqref{Jhyp}, we obtain
\begin{equation}\label{est:Jhyp12}
	\begin{aligned}
	\mathcal{J}^{H}_1 &\leq \intRp a \pbar \left(-\frac{\gamma +1}{2\vbar^2} \phi^2 +\frac{1}{\vbar\thbar}\phi \chi - \frac{1}{2\thbar^2}\chi^2 \right)\rd_\x u_1 \,d\x \\
	&\quad + C \intRp (|\pbar-p_1| + \d_{BL})|(\phi,\chi)|^2 \rd_\x u_1\,d\x + \intRp a |(\phi,\chi)|^3 \rd_\x u_1 \,d\x.
\end{aligned}
\end{equation}
Using the inequality
\[
\frac{1}{\vbar \thbar}\phi \chi \leq \frac{3}{4\vbar^2}\phi^2 + \frac{1}{3\thbar^2}\chi^2,
\]
together with $|(\phi,\chi)|\leq C\e$, we obtain
\begin{equation}\label{est:Jhyp11}
	\mathcal{J}^{H}_1 \leq -2 C_{BL} \intRp |(\phi,\chi)|^2 \rd_\x u_1 \,d\x + C \intRp (|\pbar-p_1| + \d_{BL} + \e)|(\phi,\chi)|^2|\rd_\x u_1|\,d\x,
\end{equation}
for some constant $C_{BL}>0$.
Since
\[
|p_1 - \pbar| \leq C\big(|v_1 - \vbar| + |\th_1 - \thbar|\big) \leq C(|v^R - v_*| + |\th^R - \th_*|) + C\d_0,
\]
we have 
\begin{align*}
	 C \intRp (|\pbar-p_1| + \d_{BL})|(\phi,\chi)|^2|\rd_\x u_1|\,d\x & \leq C \intRp |(v^R - v_*,\th^R - \th_*)||(\phi,\chi)|^2 \rd_\x u_1 \,d\x \\
	&\qquad +  C\d_0 \intRp |(\phi,\chi)|^2 \rd_\x u_1 \,d\x.
\end{align*}
Therefore, we obtain 
\begin{equation}\label{dec:J1}
	\mathcal{J}^{H}_1 \leq -C_{BL} \intRp |(\phi,\chi)|^2 \rd_\x u_1 \,d\x + C \intRp |(v^R - v_*,\th^R - \th_*)||(\phi,\chi)|^2 |\rd_\x u_1| \,d\x.
\end{equation}

\noindent$\bullet$ {\bf Decomposition of $\mathcal{J}^{H}_2$:} From \eqref{eq:R}$_3$, we have
\begin{align*}
	\rd_t \th_2 - \s_- \rd_\x \th_2 &= -\frac{\gamma - 1}{R}p_2 \rd_\x u_2 = -\frac{\gamma - 1}{R}\pbar \rd_\x u_2 -\frac{\gamma - 1}{R}(p_2 - \pbar)\rd_\x u_2\\
	&\leq -\frac{\gamma - 1}{R}\bar p \rd_\x u_2 + C\e\rd_\x u_2,
\end{align*}
where $p_2 := p(v_2,\th_2)$. Following the same arguments in deriving \eqref{est:Jhyp11}, we obtain
\begin{equation}\label{dec:J2}
		\mathcal{J}_2^{H} \leq -C_R \intRp |(\phi,\chi)|^2\rd_\x u_2\,d\x,
\end{equation}
for some constant $C_R>0$.

\noindent$\bullet$ {\bf Decomposition of $\mathcal{J}^{H}_3$:} Using \eqref{eq:contact}$_3$ and Lemma \ref{lem:CD}, we have
\[
|\rd_t \th_3 - \s_- \rd_\x \th_3| = C\big(|\rd_\x u_3|+ |\rd_\x \th_3||\rd_\x v_3| + |\rd_{\x\x} \th_3|\big)\leq \frac{C\d_C}{1+t} e^{-\frac{C_1 (\x + \s_- t)^2}{1+t}}.
\]
Thus, we have
\begin{equation}\label{dec:J3}
	\mathcal{J}^{H}_3 \leq C\frac{\d_C}{1+t}\intRp e^{-\frac{C_1 (\x + \s_- t)^2}{1+t}}|(\phi,\chi)|^2\,d\x.
\end{equation}
$\bullet$ {\bf Decomposition of $\mathcal{J}^{H}_4 - \mathcal{J}^{good}$:} We first decompose $\mathcal{J}_4^{H}$ as
\[
\mathcal{J}_4^{H} = \mathcal{J}_{41}^{H} +\mathcal{J}_{42}^{H} , 
\]
where
\begin{align*}
	&\mathcal{J}_{41}^{H}  = \intRp a_\x (u-\ubar)(p-\pbar)\,d\x,\\
	&\mathcal{J}_{42}^{H} =\intRp a\Big(\rd_t \th_4- \s_- \rd_{\x} \th_4 \Big)\left[R\Phi\left(\frac{v}{\vbar}\right)
	+ \frac{R}{\gamma-1}\Phi\left(\frac{\th}{\thbar}\right) \right]\,d\x \\
	&\phantom{\mathcal{J}_{4,2}^{H} =} \quad +\intRp a \left[-\frac{\pbar}{\vbar^2}|\phi|^2 + \frac{\pbar}{\vbar \thbar}\phi \chi\right] \rd_\x u_4 \, d\x.
\end{align*}
\textit{Step I) Decomposition of $\mathcal{J}_{41}^{H} - \mathcal{J}^{good}$}

Using \eqref{id:p} and \eqref{est:TE}, observe that 
\begin{equation}\label{est:acon}
	\begin{aligned}
	&(u-\ubar)(p-\pbar) - \s \thbar \eta(U|\Ubar)\\
	&\quad \leq \psi \left(\frac{R}{v}\chi - \frac{\pbar}{v}\phi\right) - \s \thbar \left(\frac{R}{2\vbar^2}\phi^2 + \frac{R}{\gamma -1}\frac{1}{2\thbar^2}\chi^2 + \frac{\psi^2}{2} \right) + C|(\phi,\chi)|^3\\
	&\quad \leq \psi \left(\frac{R}{\vbar}\chi - \frac{\pbar}{\vbar}\phi\right) - \s \thbar \left(\frac{R}{2\vbar^2}\phi^2 + \frac{R}{\gamma -1}\frac{1}{2\thbar^2}\chi^2 + \frac{\psi^2}{2} \right) + C|(\phi,\psi,\chi)|^3.
\end{aligned}
\end{equation}
Here, the last inequality follows from the fact that 
\[
 \psi\left[ \left(\frac{R}{v}\chi - \frac{\pbar}{v}\phi\right)-  \left(\frac{R}{\vbar}\chi - \frac{\pbar}{\vbar}\phi\right) \right] \leq C|(\phi,\psi,\chi)|^3.
\] 
In order to apply the Poincar\'e-type inequality in Lemma~\ref{lem:Poincare} later, we freeze the variable coefficients in the quadratic form in \eqref{est:acon} at the left end state $U^*$ of the viscous shock. To this end, using
\begin{equation}\label{est:Sfix}
	|(\vbar - v^*, \ubar - u^*, \thbar - \th^*)| \leq |(\vbar - v^S, \ubar - u^S, \thbar - \th^S)| + C\delta_S,
\end{equation}
we have
\begin{align*}
	&(u-\ubar)(p-\pbar) - \s \thbar \eta(U|\Ubar)\\
	&\quad \leq  \psi \left(\frac{R}{v^*}\chi - \frac{p^*}{v^*}\phi\right) - \s^* \thbar \left(\frac{R}{2(v^*)^2}\phi^2 + \frac{R}{\gamma -1}\frac{1}{2(\th^*)^2}\chi^2 + \frac{\psi^2}{2} \right)\\
	&\qquad + C\big(|(\vbar - v^S, \thbar - \th^S)|+ \d_S \big)|(\phi,\psi,\chi)|^2 + C|(\phi,\psi,\chi)|^3\\
	&\quad \leq -\frac{R\s^* \th^*}{2(v^*)^2}\left[\phi + \frac{\psi}{\s^*} \right]^2 - \frac{R\s^*}{2(\gamma -1)\th^*}\left[\chi - \frac{(\gamma - 1)\th^*}{v^*\s^*}\psi \right]^2\\
	&\qquad + C\big(|(\vbar - v^S, \thbar - \th^S)|+ \d_S \big)|(\phi,\psi,\chi)|^2 + C|(\phi,\psi,\chi)|^3.
\end{align*}
For the last inequality, we use the identity
\[
\frac{R\th^*}{2(v^*)^2\s^*} + \frac{R(\gamma -1)\th^*}{2(v^*)^2\s^*} = \frac{(\s^*)^2}{2\s^*} = \frac{\s^*}{2}.
\]
Therefore, we obtain 
\begin{equation}\label{dec:J41}
	\begin{aligned}
		\mathcal{J}^{H}_{41} - \mathcal{J}^{good} &\leq -G_1(U) - G_2(U) - D(U) + C\intRp a_\x |(\vbar - v^S, \thbar - \th^S)||(\phi,\psi,\chi)|^2\,d\x\\
	&\qquad + C\d_S \intRp a_\x|(\phi,\psi,\chi)|^2\,d\x  + C\intRp a_\x |(\phi,\psi,\chi)|^3\,d\x,
	\end{aligned}
\end{equation}
where 
\begin{equation*}
	\begin{aligned}
		&G_1(U)=\frac{R\s^* \th^*}{2(v^*)^2}\intRp a_\x \left(\phi + \frac{\psi}{\s^*} \right)^2 d\x, \\
		&G_2(U)= \frac{R\s^*}{2(\gamma -1)\th^*}\intRp a_\x \left(\chi - \frac{(\gamma - 1)\th^*}{v^*\s^*}\psi \right)^2 d\x,\\
		&D(U) = \intRp a \frac{\m}{v}|\psi_\x|^2 \,d\x + \intRp a \frac{\kappa}{v\th}|\chi_\xi|^2\,d\xi =: D_{u_1} + D_{\th_1}.
	\end{aligned}
\end{equation*}

\noindent \textit{Step II) Decomposition of $\mathcal{J}_{42}$}

Now, it remains to decompose $\mathcal{J}_{42}$. Observe from \eqref{eq:VS}$_3$ that,
\begin{equation*}
	\begin{aligned}
	\rd_t \th_4 - \s_- \rd_\x \th_4  &= \frac{\gamma - 1}{R}\left(-p_4 \rd_\x u_4 +\kappa \left(\frac{\rd_\x \th_4}{v_4}\right)_\x + \m \frac{(\rd_\x u_4)^2}{v_4} \right)-\dot{X}\rd_\x \th_4\\
	&\leq  -\frac{\gamma - 1}{R}\pbar \rd_\x u_4 + C(|\pbar-p_4| + \d_{S})|\rd_\x u_4|+|\dot{X}||\rd_\x \th_4|,
\end{aligned}
\end{equation*}
where $p_4 = p(v_4,\th_4)$. Following the calculations in deriving \eqref{est:Jhyp12}, we find that
\begin{equation*}
	\begin{aligned}
	\mathcal{J}^{H}_{42} &\leq \intRp a \pbar \left(-\frac{\gamma +1}{2\vbar^2} \phi^2 +\frac{1}{\vbar\thbar}\phi \chi - \frac{1}{2\thbar^2}\chi^2 \right)\rd_\x u_4 \,d\x \\
	&\quad + C \intRp (|\pbar-p_4| + \d_{S})|(\phi,\chi)|^2 \rd_\x u_4\,d\x + \intRp a |(\phi,\chi)|^3 \rd_\x u_4 \,d\x\\
    &\quad + C|\dot{X}|\intRp |(\phi,\chi)|^2|\rd_\x \th_4|\,d\x.
\end{aligned}
\end{equation*}
For the last term of the above inequality, we use 
\[
\intRp |(\phi,\chi)|^2|\rd_\x \th_4|\,d\x \leq C\e^2\intRp |\rd_\x \th_4|\,d\x \leq C\e^2\d_S
\]
to have
\begin{equation}\label{est:Xhyp}
\begin{aligned}
    |\dot{X}|\intRp |(\phi,\chi)|^2|\rd_\x \th_4|\,d\x &\leq \frac{\d_S}{4M}|\dot{X}|^2 + \frac{C}{\d_S}\left(\intRp |(\phi,\chi)|^2|\rd_\x \th_4|\,d\x\right)^2\\
    &\leq  \frac{\d_S}{4M}|\dot{X}|^2 + C\e^2 \intRp |(\phi,\chi)|^2|\rd_\x \th_4|\,d\x,
\end{aligned}
\end{equation}
where $M>0$ is defined in \eqref{def:X}. Using \eqref{est:Sfix}, \eqref{est:Xhyp}, $|\pbar - p_4| \leq C\big(|\vbar - v_4| + |\thbar - \th_4|\big)$ and $|(\psi,\phi,\chi)|\leq C\e$, we obtain
\begin{equation}\label{dec:J42}
	\begin{aligned}
	\mathcal{J}^{H}_{42} &\leq \intRp a p^* \left(-\frac{\gamma +1}{2(v^*)^2} \phi^2 +\frac{1}{v^*\th^*}\phi \chi - \frac{1}{2(\th^*)^2}\chi^2 \right)\rd_\x u_4 \,d\x \\
	&\quad + C \intRp (|(\vbar - v_4, \thbar - \th_4)| + \d_{S} + \e)|(\phi,\chi)|^2 |\rd_\x u_4|\,d\x+\frac{\d_S}{4M}|\dot{X}|^2.
\end{aligned}
\end{equation}
\subsubsection{Decomposition of $Y$}
We decompose $Y(U)$ into three terms by
\[
Y(U) := Y_1(U) + Y_2(U) + Y_3(U),
\]
where 
\begin{equation}\label{dec:Y}
\begin{aligned}
	&Y_1(U) = \intRp a \left[(u^S)_\x \psi + \frac{(v^S)_\x p^*}{v^*}\phi + \frac{R}{\gamma - 1}\frac{(\th^S)_\x}{\th^*}\chi \right]\,d\xi, \\
	&Y_2(U) = \intRp a(v^S)_\x  \left(\frac{\pbar}{\vbar} - \frac{p^* }{v^*}\right) \phi \,d\x + \intRp a \frac{R}{\gamma - 1}(\th^S)_\x \left(\frac{1}{\thbar} - \frac{1}{\th^*}\right)\chi \,d\x,\\
	&Y_3(U) = -\intRp a_{\xi} \thbar \eta(U|\Ubar)\,d\x + \intRp a \left[ -R(\th^S)_\x \Phi\left(\frac{v}{\vbar}\right) - \frac{R}{\gamma - 1}(\th^S)_\x \Phi\left(\frac{\th}{\thbar}\right)\right]\,d\x.\\
\end{aligned} 
\end{equation}
Recall from \eqref{def:X} that the shift function $X(t)$ is defined as
\begin{equation}\label{rel:XY}
    \dot{X} = -\frac{M}{\d_S}Y_1.
\end{equation}
This implies
\begin{equation*}
	\dot{X}Y = -\frac{\d_S}{M}|\dot{X}|^2 + \dot{X}(Y_2+Y_3) \leq  -\frac{3}{4M}\d_S|\dot{X}|^2 + \frac{C}{\d_S}\big(|Y_2|^2+|Y_3|^2\big). 
\end{equation*}
Combining \eqref{dec:Jhyp}, \eqref{dec:J1}, \eqref{dec:J2}, \eqref{dec:J3}, \eqref{dec:J41}, \eqref{dec:J42}, and \eqref{dec:Y}, together with the bound $|u^S_\x| \leq C\d_S^{1/2}|a_\x| \leq C|a_\x|$,
we obtain the following lemma.
\begin{lem}\label{lem:decomp}
Under the assumptions of Proposition~\ref{prop:ap}, there exist positive constants $C$ and $C_2$ such that the following holds.
\begin{align*}
	\frac{d}{dt}\intRp a \thbar \eta(U|\Ubar)\,d\x &\leq -\frac{\d_S}{2M}|\dot{X}(t)|^2 + CB^C(U) + B^S(U) + B^{rem}(U) + B^{I}(U) + \mathcal{J}^{bd}(U)\\
	&\quad - C_2 (G^{BL}(U) + G^R(U)) - G^{new}(U) - D(U),
\end{align*}
where
\begin{equation*}
	\begin{split}
		&\begin{aligned}
			&B^C(U) = \frac{\d_C}{1+t}\intRp e^{-\frac{C_1 (\x + \s_- t)^2}{1+t}}|(\phi,\psi,\chi)|^2\,d\x,\\
			&B^S(U) =  \intRp a p^* \left(-\frac{\gamma +1}{2(v^*)^2} \phi^2 +\frac{1}{v^*\th^*}\phi \chi - \frac{1}{2(\th^*)^2}\chi^2 \right) u_\x^S\,d\x,
		\end{aligned}\\
		&\begin{aligned}
			B^{rem}(U) &= \frac{C}{\d_S}|Y_3(U)|^2 + C\intRp a_\x|(\phi,\psi,\chi)|^3\,d\x + C\d_S \intRp a_\x|(\phi,\psi,\chi)|^2\,d\x\\
			&\quad + C\d_S\intRp |(\phi,\chi)|^2|u_\x^S| \,d\x + C\e\intRp |(\phi,\chi)|^2|\ubar_\x| \,d\x + \mathcal{J}^{P}(U),
		\end{aligned}\\
		&\begin{aligned}
			B^{I}(U) &= \frac{C}{\d_S}|Y_2(U)|^2 -\intRp a Q_1 (u - \ubar) \,d\x - \intRp a Q_2\left(\frac{\th}{\thbar} - 1\right)\,d\x\\
			&\quad + C\intRp \big(|(v^R - v_*,\th^R - \th_*)||u^{BL}_\x| + |(\vbar - v^S, \thbar - \th^S)| |a_\x|\big)|(\phi,\psi,\chi)|^2\,d\x,
		\end{aligned}
	\end{split}
\end{equation*}
and 
\begin{equation}\label{Good}
	\begin{split}
		&\begin{aligned}
			&G^{BL}(U) = \intRp |(\phi,\chi)|^2 u^{BL}_\x\,d\x, \quad G^R(U) =\intRp |(\phi,\chi)|^2 u^{R}_\x\,d\x,
		\end{aligned}\\
		&\begin{aligned}
			G^{new}(U) &=\frac{R\s^* \th^*}{2(v^*)^2}\intRp a_\x \left(\phi + \frac{\psi}{\s^*} \right)^2 d\x + \frac{R\s^*}{2(\gamma -1)\th^*}\intRp a_\x \left(\chi - \frac{(\gamma - 1)\th^*}{v^*\s^*}\psi \right)^2 d\x\\
			& = G_1(U) + G_2(U),
		\end{aligned}\\
		&\begin{aligned}
			D(U) = \intRp a \frac{\m}{v}|\psi_\x|^2 \,d\x + \intRp a \frac{\kappa}{v\th}|\chi_\xi|^2\,d\xi = D_{u_1} + D_{\th_1},
		\end{aligned}
	\end{split}
\end{equation}
where $C_1>0$ is the constant defined in Lemma~\ref{lem:CD}, $\mathcal{J}^{P}, \mathcal{J}^{bd}$ are given in \eqref{Jbad} and \eqref{Jbd}, and $Y_2, Y_3$ are defined in \eqref{dec:Y}.
\end{lem}
\begin{remark}
We note that, although $B^C(U)$ is defined in terms of $(\phi,\psi,\chi)$ for convenience, only the terms involving $\phi$ and $\chi$ arise in the decomposition, as seen in \eqref{dec:J3}.
\end{remark}
\subsection{Leading Order Estimates}
\begin{lem}\label{lem:leading}
There exists a constant $C_S>0$ such that
\[
-\frac{\d_S}{4M}|\dot{X}(t)|^2 + B^S(U) - G^{new}(U) - \frac{3}{4}D(U) \leq -\frac{1}{4}G^{new}(U) -C_S G^S(U),
\]
where
\[
G^S(U) := \intRp |(\phi,\psi,\chi)|^2|u^S_\x|\,d\x.
\]
\end{lem}
\noindent \textit{Proof.} We introduce the following new variable $y$:
\begin{equation}\label{def:y}
		y(t,\x):= \frac{u^* - u^S(\xi - (\s - \s_-)t - X(t) - \b)}{\d_S}.
\end{equation}
For a fixed $t>0$, since $u^S$ is monotone decreasing, the map $\x \mapsto y(t,\x)$ is well-defined, and 
\[
\frac{dy}{d\xi}= -\frac{(u^S)'(\x - (\s - \s_-)t - X(t) - \beta}{\d_S}>0, \quad \lim_{\x \to +\infty} y(t,\xi) = 1.
\]
Now, we define $y_0=y_0(t)$ as 
\[
y_0(t) := \lim_{\x \to 0+} y(t,\xi) = \frac{u^* - u^S(- (\s - \s_-)t - X(t) - \b)}{\d_S}>0.
\]
Using $\sigma>0$ and $\sigma_- <0$, together with \eqref{def:X} and \eqref{est:apriori}, we obtain
\begin{equation*}
   |\dot{X}(t)|\leq \frac{C}{\d_S}\|(\phi,\psi,\chi)\|_{L^\infty(\Rp)}\intRp |(\rho^S_\x, u^S_\x, \th^S_\x)|\,d\x \leq C\e \leq \frac{\s - \s_-}{2}, \quad \forall t<T. 
\end{equation*}
This implies
\[
|X(t)| \leq  \frac{(\s - \s_-)}{2}t, \quad \forall t<T,
\]
and so
\begin{equation}\label{y0ngtv}
    - (\s - \s_-)t - X(t) - \beta < -\frac{(\s-\s_-)}{4}t - \beta < 0, \quad \forall t<T.
\end{equation}
Therefore, by \eqref{lem:VS} and \eqref{y0ngtv}, we obtain
\[
y_0(t) \leq Ce^{-C\d_S|(\s - \s_-)t + X(t) + \b|} \leq Ce^{-C\d_S t}e^{-C\d_S \b} \leq  Ce^{-C\d_S \b}.
\]
Thus, we choose $\beta>0$ sufficiently large so that $0<y_0<\frac{1}{9}$ for all $t<T$.

We also observe that
\[
a(t,\x) = 1 + \sqrt{\d_S}y, \quad \text{and} \quad \d_S \frac{dy}{d\x} = -u^S_\x.
\]

\noindent{\bf $\bullet$ Estimate of $-\frac{M}{\d_S}|\dot{X}|^2$:}
First, using \eqref{est:shock2}, we find that
\begin{align*}
    \left|Y_1 - \intRp a\left(\psi - \frac{p^*}{v^* \s^*}\phi + \frac{p^*}{\th^* \s^*}\chi \right) u_\x^S\,d\x \right| \leq C\d_S \intRp |u_\x^S||(\phi,\chi)|\,d\x.
\end{align*}
This implies that 
\begin{align*}
    &\left|Y_1 - \intRp a\left(\psi - \frac{p^*}{v^* \s^*}\left(-\frac{\psi}{\s^*}\right) + \frac{p^*}{\th^* \s^*}\left(\frac{(\gamma -1)\th^*}{v^* \s^*}\psi \right) \right) u_\x^S\,d\x \right|\\
    &\quad \leq C\d_S \intRp |u_\x^S||(\phi,\chi)|\,d\x + C\intRp\left(\left|\phi + \frac{\psi}{\s^*} \right| + \left|\chi-\frac{(\gamma-1)\th^*}{v^*\s^*}\psi\right| \right) |u_\x^S|\,d\x.
\end{align*}
Using
\[
1 - \frac{p^*}{v^* \s^*}\left(-\frac{1}{\s^*}\right) + \frac{p^*}{\th^* \s^*}\left(\frac{(\gamma -1)\th^*}{v^* \s^*} \right) = 2
\]
together with $1\leq a \leq 1+C\sqrt{\d_S}$, we have 
\begin{align*}
    \left|Y_1 + 2\d_S \int_{y_0}^1 \psi\,dy \right| &\leq C\d_S^{1/2} \intRp |u_\x^S||\psi|\,d\x + C\d_S \intRp |u_\x^S||(\phi,\chi)|\,d\x\\
    &\quad + C\intRp\left(\left|\phi + \frac{\psi}{\s^*} \right| + \left|\chi-\frac{(\gamma-1)\th^*}{v^*\s^*}\psi\right| \right) |u_\x^S|\,d\x\\
    &\leq C\d_S^{3/2}\int_{y_0}^1 |\psi|\,dy  + C\intRp\left(\left|\phi + \frac{\psi}{\s^*} \right| + \left|\chi-\frac{(\gamma-1)\th^*}{v^*\s^*}\psi\right| \right) |u_\x^S|\,d\x.
\end{align*}
From \eqref{rel:XY}, we obtain
\begin{align*}
    \left|\dot{X} - 2M\int_{y_0}^1 \psi dy \right| \leq C\d_S^{1/2}\int_{y_0}^1 |\psi| dy + C\d_S^{-1/2}\intRp |a_\x|\left(\left|\phi + \frac{\psi}{\s^*} \right| + \left|\chi-\frac{(\gamma-1)\th^*}{v^*\s^*}\psi\right| \right)d\x.
\end{align*}
This implies
\[
\left(\left|2M\int_{y_0}^1 \psi\,dy\right| - |\dot{X}| \right)^2 \leq  C\d_S\int_{y_0}^1 |\psi|^2 dy + C\d_S^{-1}G^{new}\intRp |a_\x|\,d\x.
\]
Using the inequality $\frac{a^2}{2} - b^2 \leq (a-b)^2$ for $a,b\in \mathbb{R}$, we obtain
\begin{equation}\label{est:X2}
    -\frac{\d_S}{4M}|\dot{X}|^2 \leq -\frac{M}{2}\delta_S\left(\int_{y_0}^1 \psi dy\right)^2 + C\d_S^2 \int_{y_0}^1 |\psi|^2 dy + C\d_S^{1/2}G^{new}.
\end{equation}

{\bf $\bullet$ Estimate of $B^S$:} First, note that 
\begin{align*}
    B^S(U) = - \intRp a \frac{p^* \gamma}{2(v^*)^2}\phi^2 u_\x^S \,d\x - \intRp a \frac{p^*}{2}\left(\frac{\phi}{v^*} + \frac{\chi}{\th^*} \right)^2 u_\x^S \,d\x =: B^S_1(U) + B^S_2(U).
\end{align*}
Using the following algebraic inequalities for any $a,b, c \in \mathbb{R}$:
\begin{equation}\label{ineq:al}
\begin{aligned}
    &(a+b)^2 \leq (1+\d_S^{-1/4})a^2 + (1+\d_S^{1/4})b^2,\\
    &(a+b+c)^2 \leq (2+\d_S^{-1/4})a^2+ (2+\d_S^{-1/4})b^2 + (1+2\d_S^{1/4})c^2,
\end{aligned}
\end{equation}
we obtain 
\begin{align*}
    B_1^S(U) &=- \intRp a \frac{p^* \gamma}{2(v^*)^2}\left(\phi + \frac{\psi}{\s^*} + \left(- \frac{\psi}{\s^*}\right) \right)^2 u_\x^S \,d\x\\
    &\leq -(1+\d_S^{-1/4})\intRp a \frac{p^* \gamma}{2(v^*)^2}\left(\phi + \frac{\psi}{\s^*}\right)^2 u_\x^S \,d\x -(1+\d_S^{1/4})\intRp a \frac{p^* \gamma}{2(v^* \s^*)^2}\psi^2 u_\x^S \,d\x\\
    &\leq C\d_S^{1/4}G_1 + \d_S(1+C\d_S^{1/4})\frac{p^* \gamma}{2(v^* \s^*)^2}\int_{y_0}^1 \psi^2\,dy,
\end{align*}
and
\begin{align*}
    B_2^S(U) &=  - \intRp a \frac{p^*}{2}\left(\left(\frac{\phi}{v^*} + \frac{\psi}{v^*\s^*} \right) + \left(\frac{\chi}{\th^*}-\frac{\gamma-1}{v^*\s^*}\psi\right) + \left(\frac{\gamma -2}{v^*\s^*}\psi\right) \right)^2 u_\x^S \,d\x\\
    &\leq C\d_S^{1/4}G^{new} - (1+C\d_S^{1/4})\intRp a\frac{p^*(\gamma-2)^2}{2(v^* \s^*)^2}\psi^2 u_\x^S\,d\x\\
    &\leq C\d_S^{1/4}G^{new} +\d_S(1+C\d_S^{1/4})\frac{p^*(\gamma-2)^2}{2(v^* \s^*)^2}\int_{y_0}^1  \psi^2 \,dy.
\end{align*}
Thus, we have 
\begin{equation}\label{est:BS}
\begin{aligned}
    B^S(U) &\leq C\d_S^{1/4}G^{new} + C\d_S(1+C\d_S^{1/4})\frac{p^*(\gamma^2-3\gamma+4)}{2(v^*\s^*)^2}\int_{y_0}^1  \psi^2 \,dy\\
    &= C\d_S^{1/4}G^{new} + C\d_S(1+C\d_S^{1/4})\frac{(\gamma^2-3\gamma+4)}{2\gamma v^*}\int_{y_0}^1  \psi^2 \,dy.
\end{aligned}
\end{equation}
{\bf $\bullet$ Estimates on diffusion terms:} Recall that we have two diffusion terms
\[
D_{u_1} := \m \intRp \frac{a}{v}|\psi_\x|^2\,d\xi, \quad D_{\th_1} := \k \intRp \frac{a}{v\th}|\chi_\x|^2\,d\xi.
\]
In order to use the Poincar\'e-type inequality, the diffusion coefficients should be written in terms the left-end state $U^*$ of the viscous shock. To this end, we use the following Lemma.
\begin{lem}\label{lem:Dmon}
\begin{enumerate}
	\item For any $\gamma >1$, we have
	\begin{equation}\label{est:Du1}
		v^{R} - v_m \leq 0.
	\end{equation}
	\item For any $1 < \gamma \leq 2$, 
	\begin{equation}\label{est:Dth1}
		v^{R}\th^{R} - v_m \th_m \leq 0
	\end{equation}
	holds.
\end{enumerate}
\end{lem}
\begin{proof}
First, \eqref{est:Du1} directly follows from $v^R_\x >0$. Thus, it remains to prove the estimate \eqref{est:Dth1}.

From \eqref{def:s}, we have
\[
v^R\th^R = \frac{A}{R}(v^R)^{2-\g}\exp\left(\frac{\gamma -1}{R}(s(v^R,\th^R))\right).
\]
Since the entropy $s(v^R,\th^R)$ is constant along the rarefaction curve, i.e.,
\[
s(v^R,\th^R) = s(v_*,\th_*) = s(v_m,\th_m)=:s^r,
\]
it follows that 
\[
v^R\th^R = \frac{A}{R}(v^R)^{2-\g}\exp\left(\frac{\gamma -1}{R}s^r\right).
\]
Since $v^R$ is strictly increasing, we find that the monotonicity of $v^R \th^R$ is determined by the exponent $2 - \g$. Therefore, $v^R\th^R$ is strictly increasing when $1<\g<2$, and it remains constant when $\g=2$. This completes the proof of \eqref{est:Dth1}.
\end{proof}
Thanks to Lemma \ref{lem:Dmon}, the coefficients for the diffusion term $D_{u_1}$ can be controlled as follows. First, note that
\[
\frac{1}{v} = \frac{1}{v^S} + \left(\frac{1}{v} - \frac{1}{\vbar} \right) + \left(\frac{1}{\vbar} - \frac{1}{v^S} \right).
\]
For the last term of the right-hand side of the above equality, we use Lemma \ref{lem:Dmon} to have
\begin{align*}
	\frac{1}{\vbar} - \frac{1}{v^S} &= - \frac{\vbar - v^S}{\vbar v^S} = - \frac{v^{BL} - v_* + v^R - v_m + v^C - v^*}{\vbar v^S}\\
	&\quad \geq  - \frac{v^{BL} - v_* + v^C - v^*}{\vbar v^S} \geq -C\d_0,
\end{align*}
together with $a\geq 1$, we obtain 
\begin{equation}\label{est:Du1xi}
	D_{u_1} \geq \m(1-C(\e + \d_0))\intRp \frac{1}{v^S}|(u-\ubar)_\x|^2\,d\x,
\end{equation}
Now, we estimate the diffusion term $D_{\th_1}$. Observe that
\[
\frac{1}{v \th} = \frac{1}{v^S \th^S} + \left( \frac{1}{v \th} - \frac{1}{\vbar \thbar}\right) + \left(\frac{1}{\vbar \thbar} - \frac{1}{v^S \th^S} \right).
\]
For the last term of the right-hand side of the above equality, we have
\begin{align*}
\frac{1}{\vbar \thbar} - \frac{1}{v^S \th^S} &= \frac{-\Big((v^R + \vtil - v_m)(\th^R + \thtil - \th_m) - v_m \th_m + (v_m \th_m - v^S \th^S)\Big)}{\vbar \thbar v^S \th^S},
\end{align*}
where $\Util$ denotes the superposition of the boundary layer, viscous contact wave, and viscous shock wave:
\[
\Util(t,\x) = U^{BL}(\x) + U^C(t,\x) + U^S(\x - (\s -\s_-)t - X(t)-\beta) - U_* - U^*.
\]
Then, using Lemma \ref{lem:Dmon}, and the fact that
\(
\Util - U_m = O(\delta_0),
\)
the numerator of the right-hand side of the above equality becomes
\begin{align*}
	&- \Big((v^R + \vtil - v_m)(\th^R + \thtil - \th_m) - v_m \th_m + (v_m \th_m - v^S \th^S)\Big)\\
&\quad = -\Big( (v^R \th^R - v_m \th_m) + v^R(\thtil - \th_m) + \th^R(\vtil - v_m) + (\vtil - v_m)(\thtil - \th_m)+(v_m \th_m - v^S \th^S) \Big)\\
&\quad \geq  -\Big(v^R(\thtil - \th_m) + \th^R(\vtil - v_m) + (\vtil - v_m)(\thtil - \th_m) +(v_m \th_m - v^S \th^S)\Big) \geq -C\d_0.
\end{align*}
Therefore, we obtain
\begin{equation}\label{est:Dth1xi}
	D_{\th_1} \geq \kappa \big(1- C(\e + \d_0))\intRp \frac{1}{v^S \th^S} |(\th - \thbar)_\x|^2\,d\x.
\end{equation}
In order to apply the Poincar\'e-type inequality with respect to the variable $y$, we need to estimate the Jacobian $\frac{dy}{d\x}$. This is provided by the following Lemma.
\begin{lem}\label{lem:Jac}
There exists a positive constant $C>0$ such that
\[
\left|\frac{\mu}{y(1-y)}\frac{dy}{d\x} - \d_S \frac{(\gamma+1)}{2}\frac{\mu R \gamma}{\m R \gamma + \kappa (\gamma - 1)^2} \right| \leq C\d_S^2.
\]
\end{lem}
\begin{proof}
	From \cite[(4.59)]{KVW-NSF}, we have 
	\begin{equation}\label{est:Jac1}
	\left|\frac{(v_+ - v^*)^2}{(v^S-v^*)(v_+ - v^S)}\frac{\m}{v^S}\frac{v^S_\x}{v_+ - v^*} - \frac{\gamma (\gamma + 1)p^*}{2 (v^*)^2 \sigma^*} \frac{\m R \gamma}{\m R \gamma + \kappa (\gamma - 1)^2}(v_+ - v^*)  \right| \leq C\d_S^2.
	\end{equation}
Integrating \eqref{eq:VS}$_1$, we obtain
\begin{equation}\label{eq:VSint}
	-\s(v^S - v^*) - (u^S - u^*) = 0 \, \, \text{and }  -\s(v^S - v_+) - (u^S - u_+) = 0.
\end{equation}
Then using \eqref{eq:VS}$_1$, \eqref{eq:VSint}, and \eqref{eq:RH}$_1$, the left-hand side of \eqref{est:Jac1} can be rewritten as 
\[
\text{l.h.s. of \eqref{est:Jac1}} = \left|\frac{1}{y(1-y)}\frac{\m}{v^S}\frac{dy}{d\x} - \frac{\gamma (\gamma + 1)p^*}{2 (v^*)^2 \sigma^*}\frac{\d_S}{\s} \frac{\m R \gamma}{\m R \gamma + \kappa (\gamma - 1)^2}  \right|.
\]
Using $|v^S - v^*| + |\s - \s^*| \leq C\d_S$, together with the identity
\[
\frac{\gamma (\gamma + 1)p^*}{2 (v^*\sigma^*)^2} = \frac{(\gamma+1)}{2v^*},
\]
we then obtain the desired estimate.
\end{proof}
Using Lemma \ref{lem:Jac} in \eqref{est:Du1xi} and \eqref{est:Dth1xi}, we obtain 
\begin{align*}
\begin{aligned}
	D_{u_1} &\geq \m(1-C(\e + \d_0))\int_{y_0}^1 |\psi_y|^2\left(\frac{dy}{d\x}\right)\,dy\\
	&\geq \frac{(\gamma+1)}{2v^*} \frac{\m R \gamma}{\m R \gamma + \kappa (\gamma -1)^2}(1-C(\e + \d_0))\d_S\int_{y_0}^1 y(1-y)|\psi_y|^2\,dy\\
	&\geq \frac{(\gamma+1)}{2v^*} \frac{\m R \gamma}{\m R \gamma + \kappa (\gamma -1)^2}(1-C(\e + \d_0))\d_S\int_{y_0}^1 (y-y_0)(1-y)|\psi_y|^2\,dy,
    \end{aligned}
\end{align*}
and likewise, we have
\begin{align*}
	D_{\th_1} \geq \frac{(\gamma+1)}{2v^*\th^*} \frac{\k R \gamma}{\m R \gamma + \kappa (\gamma -1)^2} (1-C(\e + \d_0))\d_S \int_{y_0}^1 (y-y_0)(1-y)|\chi_y|^2\,dy.
\end{align*}
Using Poincar\'e-type inequality in Lemma~\ref{lem:Poincare}, we obtain
\begin{align*}
    D &\geq \frac{(\gamma+1)}{v^*} \frac{\m R \gamma}{\m R \gamma + \kappa (\gamma -1)^2}(1-C(\e + \d_0))\d_S\left[\int_{y_0}^1 \psi^2 \,dy - \frac{1}{1-y_0}\left(\int_{y_0}^1 \psi\,dy\right)^2 \right]\\
    &\qquad +\frac{(\gamma+1)}{v^*\th^*} \frac{\k R \gamma}{\m R \gamma + \kappa (\gamma -1)^2} (1-C(\e + \d_0))\d_S \left[\int_{y_0}^1 \chi^2 \,dy - \frac{1}{1-y_0}\left(\int_{y_0}^1 \chi\,dy\right)^2 \right].
\end{align*}
From the inequality $(a+b)^2 \geq (1-\d_S^{-1/4})a^2 + (1-\d_S^{1/4})b^2$, we obtain
\begin{equation}\label{chisquare}
\begin{aligned}
   \int_{y_0}^1 \chi^2 \,dy &= \int_{y_0}^1 \left(\chi - \frac{(\gamma -1)\th^*}{v^* \s^*}\psi +\frac{(\gamma -1)\th^*}{v^* \s^*}\psi \right)^2 \,dy \\
   &\geq (1-\d_S^{-1/4}) \int_{y_0}^1 \left(\chi - \frac{(\gamma -1)\th^*}{v^* \s^*}\psi \right)^2\,dy + (1-\d_S^{1/4})\int_{y_0}^1 \left(\frac{(\gamma -1)\th^*}{v^* \s^*}\psi \right)^2\,dy\\
   &\geq -C\d_S^{-3/4}G^{new} +(1-\d_S^{1/4})\left(\frac{(\gamma -1)\th^*}{v^* \s^*}\right)^2 \int_{y_0}^1 \psi^2\,dy.
\end{aligned}
\end{equation}
Moreover, using \eqref{ineq:al}$_1$ and H\"older's inequality, we have 
\[
\begin{aligned}
    \Big(\int_{y_0}^1 \chi dy\Big)^2 &\leq   C\d_S^{-1/4}(1-y_0)  \int_{y_0}^1 \Big(\chi - \frac{(\gamma -1)\th^*}{v^* \s^*}\psi \Big)^2 dy + (1+\d_S^{1/4}) \Big(\int_{y_0}^1 \frac{(\gamma -1)\th^*}{v^* \s^*}\psi dy\Big)^2\\
    &\leq C\d_S^{-3/4}G^{new} + (1+\d_S^{1/4})\left(\frac{(\gamma -1)\th^*}{v^* \s^*}\right)^2\left(\int_{y_0}^1 \psi \,dy \right)^2.\\
\end{aligned}
\]
Therefore, we finally obtain
\begin{equation}\label{est:D}
    \begin{aligned}
            D &\geq  \frac{\gamma +1}{v^*}\d_S (1-C\d_S^{1/4} - C(\e + \d_0))\int_{y_0}^1 \psi^2\,dy \\
      &\quad - \frac{\gamma +1}{(1-y_0)v^*}\d_S (1+C\d_S^{1/4} +C(\e + \d_0))\left(\int_{y_0}^1 \psi\,dy\right)^2 - C\d_S^{1/4}G^{new}.
    \end{aligned}
\end{equation}
{\bf $\bullet$ Proof of Lemma~\ref{lem:leading}:} Combining \eqref{est:X2}, \eqref{est:BS}, and \eqref{est:D}, and using $0<y_0<\frac{1}{9}$, we obtain
\begin{equation*}
    \begin{aligned}
        &-\frac{\d_S}{4M}|\dot{X}(t)|^2 + B^S - G^{new} - \frac{3}{4}D \\
        &\quad \leq \d_S \left(-\frac{M}{2} + \frac{7}{8}\frac{(\gamma+1)}{v^*} \right)\left(\int_{y_0}^1 \psi dy\right)^2+ \d_S \left(\frac{\gamma^2 - 3\gamma + 4}{2\gamma v^*} - \frac{2}{3}\frac{\gamma+1}{v^*}\right)\int_{y_0}^1 \psi^2 dy - \frac{2}{3}G^{new}. \\
    \end{aligned}
\end{equation*}
Choosing $M = \frac{2(\gamma+1)}{v^*}$ and using the inequality
\[
\frac{\gamma^2 - 3\gamma + 4}{2\gamma} - \frac{2}{3}(\gamma+1) \leq -\frac{1}{6\gamma}, \quad \text{for any } \gamma >1,
\]
we obtain
\begin{equation}\label{est:lot2}
\begin{aligned}
    -\frac{\d_S}{4M}|\dot{X}(t)|^2 + B^S - G^{new} - \frac{3}{4}D \leq -\frac{1}{6\gamma v^*}\intRp |\psi|^2|u_\x^S|\,d\x-\frac{2}{3}G^{new}.
\end{aligned}
\end{equation}
Following the same argument as in \eqref{chisquare}, and using the inequality $(a+b)^2 \geq (1-\d_S^{-1/4})a^2 + (1-\d_S^{1/4})b^2$, we deduce that
\begin{equation}\label{der:GS}
\intRp |\psi|^2|u_\x^S|\,d\x \geq -C\d_S^{3/4}G^{new} + \widetilde{C}_S(1-\d_S^{1/4})\intRp |(\phi,\chi)|^2|u_\x^S|\,d\x,
\end{equation}
for some positive constant $\widetilde{C}_S$. Combining \eqref{est:lot2} and \eqref{der:GS}, we conclude that 
\[
-\frac{\d_S}{4M}|\dot{X}(t)|^2 + B^S - G^{new} - \frac{3}{4}D \leq -C_SG^S-\frac{1}{2}G^{new},
\]
for some positive constant $C_S$.\\
\qed
\subsection{Estimates of the Remainder Terms}
In what follows, we estimate the remaining bad terms in Lemma~\ref{lem:decomp}.
\subsubsection{Estimates of $B^{rem}(U)$ and $B^C(U)$}
\begin{lem}\label{lem:Brem}
     Under the assumptions of Proposition~\ref{prop:ap}, there exists a positive constant $C$ such that the following estimate holds.
    \[
    \begin{aligned}
        |B^{rem}(U)| &\leq \left(\frac{1}{50}+C\e\right)(D_{u_1} + D_{\th_1}) + C\e G^{new} + C(\d_0+ \e)(G^{BL}+G^S+B^C)\\
        &\qquad + C(\e_r\d_R + \e)G^R.
    \end{aligned}
    \]
\end{lem}
\begin{proof}
First, we split $B^{rem}(U)$ as
\begin{align*}
	B^{rem}(U) &= \frac{C}{\d_S}|Y_3(U)|^2 + C\intRp a_\x|(\phi,\psi,\chi)|^3\,d\x + C\d_S \intRp a_\x|(\phi,\psi,\chi)|^2\,d\x\\
	&\quad + C\d_S\intRp |(\phi,\chi)|^2|u_\x^S| \,d\x + C\e\intRp |(\phi,\chi)|^2|\ubar_\x| \,d\x + \mathcal{J}^{P}(U),  \\
    &:=\frac{C}{\d_S}|Y_3(U)|^2 + B^{rem}_1 + B^{rem}_2 + B^{rem}_3 + B^{rem}_4 + \mathcal{J}^{P}(U).
\end{align*}

\noindent$\bullet$ {\bf Estimates of $Y_3(U)$:}
From $|u_\x| \leq C\d_S^{1/2}|a_\x|$, note that
\begin{align*}
|Y_3(U)| &= \left|-\intRp a_{\xi} \thbar \eta(U|\Ubar)\,d\x + \intRp a \left[ -R(\th^S)_\x \Phi\left(\frac{v}{\vbar}\right) - \frac{R}{\gamma - 1}(\th^S)_\x \Phi\left(\frac{\th}{\thbar}\right)\right]\,d\x\right| \\
&\leq C\intRp a_\x |(\phi,\psi,\chi)|^2\,d\x. 
\end{align*}
Using \eqref{def:a}, Lemma~\ref{lem:VS}, and \eqref{est:apriori}, we obtain
\begin{align*}
	\frac{C}{\delta_S}|Y_{3}|^2 \le &\frac{C}{\delta_S} \left(\intRp  |a_\xi| |(\phi, \psi, \chi)|^2 d\xi\right)^2
	 \le \frac{C}{\delta_S^2} \left(\intRp|u^S_\xi|  |(\phi, \psi, \chi)|^2 d\xi\right)^2\\
	 \le& \|(\phi, \psi, \chi)\|_{L^2}^2 \intRp |u^S_\xi| |(\phi, \psi, \chi)|^2d\xi \le C\eps^2 G^S.\\
\end{align*}
Here, we used the fact that 
\[
\intRp |u_\x^S||(\phi,\psi,\chi)|^2\,d\x \leq C\d_S^2\|(\phi, \psi, \chi)\|_{L^2}^2.
\]
\noindent$\bullet$ {\bf Estimates of $B^{rem}_1, B^{rem}_2, B^{rem}_3$:} 
First, by \eqref{da} and the interpolation inequality, we have
\begin{align*}
		\begin{aligned}
		B^{rem}_1=& \intRp a_\xi  |(\phi,\psi,\chi)|^3 d\x \\
		 \le & C \intRp  |a_\xi |\Big|\phi+\frac{\psi}{\sigma^*}\Big| ^3d\xi+C\intRp |a_\xi||\psi|^3d\xi +C \intRp |a_\xi |\Big|\chi-\frac{(\gamma-1)\theta^*}{v^*\s^*}\psi \Big|^3 d\xi \\
	\le &C\e (G_1+G_2)+C \frac{1}{\sqrt{\delta_S}} \intRp |u^S_\xi|  \|\psi\|_{L^\infty}^2 |\psi| d\xi \\
		\le&  C \e(G_1+G_2)+C\frac{1}{\sqrt{\delta_S}}  \|\psi_\xi\|_{L^2} \|\psi\|_{L^2}  \sqrt{\intRp  |u^S_\xi| \psi^2 d\xi} \sqrt{\intRp  |u^S_\xi |  d\xi} \\
	\le &C \e(G_1+G_2)+C\e \sqrt{D} \sqrt{G^S} \le C \e(G_1+G_2 + D+ G^S).
		\end{aligned}
\end{align*}
Again using $|u_\x| \leq C\d_S|a_\x|$ and \eqref{da}, we obtain
\[
	B^{rem}_2 + B^{rem}_3 \leq C\delta_S \intRp a_\xi |(\phi,\psi,\chi)|^2 \,d\x \le C\sqrt{\delta_S} G^S,
	\]
Using $\ubar_\x= u^{BL}_\x + u^R_\x + u^C_\x + u^S_\x$ and Lemma \ref{lem:CD}, we have 
\begin{align*}
    B^{rem}_4  \le C\e\intRp |(\phi,\chi)|^2|\ubar_\x| \,d\x \leq C\e \left(G^{BL} +G^R + G^S\right) + C\e B^C.
\end{align*}
$\bullet$ {\bf Estimates of  $\mathcal{J}^{P}(U)$:} We first decompose $\mathcal{J}^{P}(U)$ as follows:
\begin{align*}
		\mathcal{J}^{P}&= -\intRp a_\x \left[\m (u-\ubar)\left(\frac{u_\x}{v} - \frac{\ubar_\x}{\vbar}\right) + \kappa \frac{\th -\thbar}{\th} \left(\frac{\th_\x}{v} - \frac{\thbar_\x}{\vbar}\right)\right]\,d\x\\
		&\quad  +\intRp a \left[ -\m \ubar_\x \left(\frac{1}{v} - \frac{1}{\vbar}\right)(u-\ubar)_\x + \kappa \frac{\th - \thbar}{\th^2}\th_\x \left(\frac{\th_\x}{v} - \frac{\thbar_\x}{\vbar}\right) - \kappa \frac{(\th - \thbar)_\x}{\th}\thbar_\x \left(\frac{1}{v} - \frac{1}{\vbar}\right) \right.\\
		&\phantom{+\intRp a \Big[ \quad } \left.+\m \frac{\th - \bar{\theta}}{\th}\left(\frac{u_\x^2}{v} - \frac{\ubar_\x^2}{\vbar}\right) - \frac{(\th - \thbar)^2}{\th \thbar}\left(\kappa \left(\frac{\thbar_\x}{\vbar}\right)_\x + \m \frac{\ubar_\x^2}{\vbar}\right)\right]\,d\x\\
        &:=\mathcal{J}_1^{P} + \mathcal{J}_2^{P}.
	\end{align*}
Using Young's inequality and \eqref{da}, we obtain
\begin{align*}
	\begin{aligned}
		\mathcal{J}_1^{P} &\le C \intRp|a_\xi |\psi| \Big(  |\psi_\xi| + |\bar u_\xi| |\phi| \Big) d\xi + C \intRp |a_\xi | |\chi| \Big(  |\chi_\xi| + |\bar \theta_\xi| |\phi| \Big) d\xi \\
		&\le \frac{1}{100} D_{u_1}+\frac{1}{100} D_{\theta_1} +  C\frac{1}{\delta_S} \intRp |u^S_\xi|^2 |(\psi, \chi) |^2 d\xi  + C\intRp ( |\bar u_\xi|^2 + |\bar \theta_\xi|^2)  |\phi|^2 d\xi \\
		&\le  \frac{1}{100} D_{u_1}+\frac{1}{100} D_{\theta_1} + C\delta_S G^S  +  C\intRp ( |\bar u_\xi|^2 + |\bar \theta_\xi|^2)  |\phi|^2 d\xi.
	\end{aligned}
\end{align*}
Similarly, applying Young's inequality, we have the following estimate for $\mathcal{J}_2^{P}$:
\begin{align*}
\begin{aligned}
\mathcal{J}_2^{P} \le & \intRp \Big[ |\bar u_\xi| |\phi|  |\psi_\xi| +  |\chi| \Big(  |\chi_\xi| + |\bar \theta_\xi| \Big) \Big(  |\chi_\xi| + |\bar \theta_\xi| |\phi| \Big) \Big]d\xi   \\
& +  \intRp \Big[  |\chi_\xi| |\bar \theta_\xi| |\phi| +|\chi| \Big( |\psi_\xi|^2 +  |\bar u_\xi| |\psi_\xi| + |\bar u_\xi|^2 |\phi|  \Big) \Big] d\xi \\
& +\intRp \left(|\bar \theta_{\xi\xi}|+ |\bar v_\x \bar\theta_\x| + |\bar u_\xi|^2\right) |\chi|^2 \,d\x\\
\leq &\frac{1}{100}(D_{u_1} + D_{\th_1}) +  C \intRp ( |\bar u_\xi|^2 + |\bar \theta_\xi|^2 + |\bar \th_{\x\x}|)  |(\phi, \chi)|^2 d\xi.  
\end{aligned}
\end{align*}
By Lemma \ref{lem:BL}, Lemma \ref{lem:rarefaction}, \eqref{est:contact}, \eqref{est:shock}, we have
\begin{align*}
	\begin{aligned} 
    \intRp ( |\bar u_\xi|^2 + |\bar \theta_\xi|^2 + |\bar \th_{\x\x}|)  |\phi|^2 d\xi \leq
	C \left(\delta_{BL} B^{BL} + \e_r\delta_R G^R +\delta_S G^S\right)+ C\delta_CB^C.
	\end{aligned}
\end{align*}
Thus, we have 
\[
\mathcal{J}^{P} =\mathcal{J}^{P}_1 +\mathcal{J}^{P}_2 \leq \frac{1}{50}(D_{u_1} + D_{\th_1}) + C(\d_{BL}G^{BL} + \e_r\d_RG^R + \d_S G^S) + C\d_C B^C.
\]
Combining all these estimates, we obtain the desired result in Lemma~\ref{lem:Brem}.
\end{proof}
The following lemma provides an estimates for $B^C$.
\begin{lem}\label{lem:L2CD}
    Under the assumptions of Proposition~\ref{prop:ap}, there exists a positive constant $C$ such that the following estimate holds.
    \[
    \begin{aligned}
        \int_0^T |B^C|\,dt &\leq C\d_C\sup_{t \in [0,T]} \|(U- \Ubar)(t,\cdot)\|_{L^2}^2 
        + C\d_C\int_0^T (\d_S|\dot{X}|^2+G^S +G^R+G^{BL})\,dt\\
        &\qquad +C\d_C\int_0^T (D_{v_1} + D_{u_1} + D_{\th_1} + D_{u_2} + D_{\th_2}) \,dt + C\d_C. 
    \end{aligned}
    \]
\end{lem}
\begin{proof}
    Since the proof of Lemma~\ref{lem:L2CD} is almost similar to those in \cite{HLM10} and \cite[Appendix D]{KVW-NSF}, we defer it to Appendix~\ref{App:CD}.
\end{proof}

\subsubsection{Estimates of $ B^{I}(U)$}
\begin{lem}\label{lem:BIU}
Under the assumptions of Proposition~\ref{prop:ap}, there exists a positive constant $C$ such that, for any $t\in[0,T]$,
    \[
\int_0^t |B^I|\,d\t \leq C\d_0^{1/3} + C(\e_r\d_R)^{1/9} + C\d_0\left(\frac{\d_R}{\e_r}\right)^{4/3}\frac{1}{\e_r} + C\e^2\int_0^t D_{v_1} \,d\t + \frac{1}{40}\int_0^t (D_{u_1} + D_{\th_1})\,d\t.
\]
\end{lem}
\begin{proof}
We decompose $ B^{I}(U)$ as follows:
\begin{align*}
  B^{I}(U) &= \frac{C}{\d_S}|Y_2(U)|^2 -\intRp a Q_1 (u - \ubar) \,d\x - \intRp a Q_2\left(\frac{\th}{\thbar} - 1\right)\,d\x\\
  &\quad + C\intRp \big(|(v^R - v_*,\th^R - \th_*)||u^{BL}_\x| + |a_\x||(\vbar - v^S, \thbar - \th^S)|\big)|(\phi,\chi)|^2\,d\x \\
  &:=  \frac{C}{\d_S}|Y_2(U)|^2 +  B^{I}_1 +  B^{I}_2 +  B^{I}_3 + B^{I}_4.
\end{align*}

\noindent$\bullet$ {\bf Estimates of $Y_2(U)$:} Observe that 
\begin{align*}
    |Y_2(U)| &= \left|\intRp a(v^S)_\x  \left(\frac{\pbar}{\vbar} - \frac{p^* }{v^*}\right) \phi \,d\x + \intRp a \frac{R}{\gamma - 1}(\th^S)_\x \left(\frac{1}{\thbar} - \frac{1}{\th^*}\right)\chi \,d\x \right|\\
    &\leq C\intRp |(v^S_\x,\th^S_\x)||(\vbar - v^*,\thbar - \th^*)||(\phi,\chi)|\,d\x\\
    &\leq C\intRp |(v^S_\x,\th^S_\x)||(\vbar - v^S,\thbar - \th^S)||(\phi,\chi)|\,d\x + C\d_S \intRp  |(v^S_\x,\th^S_\x)||(\phi,\chi)|\,d\x.
\end{align*}
Using \eqref{def:Ubar}, $|v_\x^S|\sim |u_\x^S| \sim |\th_\x^S|$ and the fact that 
\[
\begin{aligned}
    |v^{BL} - v_*| \sim |\th^{BL} - \th_*|, \quad  |v^R - v_m| \sim |\th^R - \th_m|, \quad |v^{C} - v^*| \sim |\th^C - \th^*|,
\end{aligned}
\]
we have 
\[
\begin{aligned}
    |Y_2(U)| &\leq C\|(\phi,\chi)\|_{L^\infty}\intRp |v_\x^S|\big(|v^{BL} - v_*| + |v^R - v_m| + |v^C - v^*| \big)\,d\x\\
    &\quad +  C\d_S \intRp  |u_\x^S||(\phi,\chi)|\,d\x\\
    &\leq C\|(\phi,\chi)\|_{L^\infty}\intRp (J_3 + J_5 + J_6 \big)\,d\x +  C\d_S \intRp  |u^S_\x||(\phi,\chi)|\,d\x.
\end{aligned}
\]
where $J_3, J_5$, and $J_6$ are defined in \eqref{def:J}. Using the interpolation inequality, \eqref{est:apriori}, and H\"older's inequality, we obtain
\[
\begin{aligned}
    \frac{C}{\d_S}Y_{2}^2 &\leq \frac{C}{\d_S}\|(\phi,\chi)\|_{L^2}\|(\phi_\x,\chi_\x)\|_{L^2}\left(\intRp (J_3 + J_5 + J_6)\,d\x \right)^2 + C\d_S \left(\intRp |u^S_\x||(\phi,\chi)|^2\,d\x\right)^2\\
    &\leq \frac{C}{\d_S}\e\|(\phi_\x,\chi_\x)\|_{L^2}\left(\intRp (J_3 + J_5 + J_6)\,d\x \right)^2 + C\d_S \intRp |u^S_\x|\,d\x \intRp |u^S_\x||(\phi,\chi)|^2\,d\x\\
    &\leq  C\e^2(D_{v_1} + D_{\th_1}) + \frac{C}{\d_S^2}\left(\intRp (J_3 + J_5 + J_6)\,d\x \right)^4 + C\d_S^2 G^S.
\end{aligned}
\]
By Lemma~\ref{lem:win}, we find that 
\[
\begin{aligned}
  &\frac{C}{\d_S^2}\int_0^t \left(\intRp (J_3 + J_5 + J_6)\,d\x \right)^4 \,d\t \\
  &\quad \leq \frac{C}{\d_S^2}\left[(\d_{BL}^4\d_S^3 + \d_{BL}^3\d_S^4) + \left(\d_R^4\d_S^3(\e_r+\d_S)^4 +  \frac{(\d_R\d_S)^4}{\e_r^5}(\e_r + \d_S)^4\right) + (\d_C^4\d_S^3 + \d_C^4 \d_S^4)\right]\\
  &\quad \leq C\d_S + C \d_S^2 \frac{\d_R^4}{\e_r^5}.
\end{aligned}
\]
\noindent$\bullet$ {\bf Estimates of $B^{I}_1,   B^{I}_2$:}
From \eqref{dec:Qi},
\[
\begin{aligned}
    B_1^I + B_2^I &= \intRp aQ_1(u-\ubar)\,d\x + \intRp a Q_2\left(1-\frac{\thbar}{\th}\right)\,d\x\\
    &\leq C\intRp \sum_{i=1}^2 |Q_i^I||(\psi,\chi)|\,d\x + C\intRp \sum_{i=1}^2 |Q_i^R||(\psi,\chi)|\,d\x + C\intRp \sum_{i=1}^2 |Q_i^C||(\psi,\chi)|\,d\x\\
    &:= I_1 + I_2 + I_3.
\end{aligned}
\]
For $I_1$, we use the interpolation inequality, Young's inequality, and \eqref{est:apriori} to have 
\[
\begin{aligned}
    |I_1| &\leq \|(\psi,\chi)\|_{L^\infty}\intRp \sum_{i=1}^2 |Q_i^I|\,d\x \leq \|(\psi,\chi)\|_{L^2}^{1/2}\|(\psi_\x,\chi_\x)\|_{L^2}^{1/2}\intRp \sum_{i=1}^2 |Q_i^I|\,d\x\\
    &\leq \frac{1}{100}(D_{u_1} + D_{\th_1}) + C\e^{2/3}\sum_{i=1}^2\|Q_i^I\|_{L^1}^{4/3}.
\end{aligned}
\]

Recall that the term $Q_1^I$ in \eqref{def:QI} defined by 
\[
Q_1^I = \big(\pbar - p^{BL} - p^R - p^C - p^S \big)_\x - \m \left(\frac{\ubar_\x}{\vbar} - \frac{u^{BL}_\x}{v^{BL}}  - \frac{u^R_\x}{v^R}  - \frac{u^C_\x}{v^C}  - \frac{u^S_\x}{v^S}  \right)_\x.
\]
For the first term of $Q_1^I$, using $|(v_i)_\x| \sim |(\th_i)_\x|$ and $|\vbar - v_i| \sim |\thbar - \th_i|$ for $i=1,2,3,4$, we find that 
\[
\big(\pbar - p^{BL} - p^R - p^C - p^S\big)_\x \leq C\sum_{i=1}^4 |(v_i)_\xi||\vbar - v_i|. 
\]
Moreover, we obtain 
\begin{align*}
	\left(\frac{\ubar_\x}{\vbar} - \frac{u^{BL}_\x}{v^{BL}}  - \frac{u^R_\x}{v^R}  - \frac{u^C_\x}{v^C}  - \frac{u^S_\x}{v^S}  \right)_\x \leq \sum_{i=1}^4 \big(|(u_i)_{\x\x}|,|(u_i)_\x||(v_i)_\x|\big)|\vbar - v_i| + \sum_{i\neq j}|(u_i)_\x||(v_j)_\x|.
\end{align*}
Therefore, we have 
\begin{equation}\label{est:Q1}
	|Q_1^I| \leq C\sum_{i=1}^4 |(v_i)_\xi||\vbar - v_i| + |u_{\x\x}^C||\vbar - v^C| + \sum_{i\neq j}|(u_i)_\x||(v_j)_\x|.
\end{equation}
In the above estimate, we have used the estimates $|(u_i)_{\x\x}| \leq C|(u_i)_{\x}|$ and $|(u_i)_{\x}| \sim |(v_i)_{\x}|$ for $i=1,2,4$, which follow directly from Lemmas \ref{lem:BL}, \ref{lem:rarefaction}, and \ref{lem:VS}.

In addition, we recall that 
\begin{align*}
		Q_2^I &:= \Big( \bar{p} \bar{u}_\x -p^{BL} u^{BL}_\x - p^R u_\x^R - p^C u_\x^C - p^S_\x u^S_\x) - \kappa \left(\frac{\thbar_\x}{\vbar} - \frac{\th^{BL}_\x}{v^{BL}}  - \frac{\th^R_\x}{v^R}  - \frac{\th^C_\x}{v^C}  - \frac{\th^S_\x}{v^S} \right)_\x\\
			&\quad -\m \left(\frac{\ubar_\x^2}{\vbar} - \frac{(u^{BL}_\x)^2}{v^{BL}}  - \frac{(u^R_\x)^2}{v^R}  - \frac{(u^C_\x)^2}{v^C}  - \frac{(u^S_\x)^2}{v^S}  \right) =: L_1 + L_2 + L_3.
\end{align*}
Note that
\begin{align*}
	|L_1| \leq \sum_{i=1}^4 |(u_i)_\x||\pbar - p_i| \leq C \sum_{i=1}^4 |(v_i)_\x||\vbar - v_i| + |u_\x^C||\vbar - v^C|.
\end{align*}
Using calculations similar to those for $Q_1^I$, we obtain
\[
|L_2| \leq  C\sum_{i=1}^4 |(v_i)_\xi||\vbar - v_i| + C(|v_\x^C|^2+|u_\x^C|)|\vbar - v^C| + \sum_{i\neq j}|(v_i)_\x||(v_j)_\x|,
\]
where we have used  $|(v_i)_\x| \sim |(\th_i)_\x|$, $|(v_i)_{\x\x}| \sim |(\th_i)_{\x\x}|$ for $i=1,2,3,4$, together with $|(v_i)_{\x\x}| \leq C|(v_i)_\x|$ for $i=1,2,4$, and for the viscous contact wave, 
\[
|\th^C_{\x\x}| \leq C|\th^C_{\x}|^2 + C|u^C_\x|,
\]
which follows from \eqref{eq:VCD}$_2$. Finally, we have 
\begin{align*}
	|L_3| &\leq C\sum_{i=1}^4 |(u_i)_\x|^2|\vbar - v_i| + C\sum_{i\neq j}|(u_i)_\x||(u_j)_\x|\\
	&\leq  C\sum_{i=1}^4 |(v_i)_\x||\vbar - v_i| + C|u_\x^C|^2|\vbar - v^C| + C\sum_{i\neq j}|(u_i)_\x||(u_j)_\x|.
\end{align*}
Therefore, we obtain 
\begin{equation}\label{est:Q2}
	|Q_2^I| \leq C\sum_{i=1}^4 |(v_i)_\x||\vbar - v_i| + C|u_\x^C||\vbar - v^C| + C\sum_{i\neq j}|(u_i)_\x|\big(|(v_j)_\x| + |(u_j)_\x|\big).
\end{equation}
Using the fact that 
\[
 \sum_{i=1}^4 |(v_i)_\x||\vbar - v_i| \leq C\sum_{i=1}^6 J_i,
\]
where $J_i$ are defined in \eqref{def:J},
and combining \eqref{est:Q1} and \eqref{est:Q2}, we obtain 
\[
|Q_1^I| + |Q_2^I| \leq C\sum_{i=1}^6 J_i + C|(u_\x^C,u_{\x\x}^C)||\vbar - v^C| + C\sum_{i\neq j}|(u_i)_\x|\big(|(v_j)_\x| + |(u_j)_\x|\big).
\]
Therefore, using \eqref{est:431} and \eqref{est:432nd}, we obtain 
\[
\int_0^t \||Q_1^I| + |Q_2^I|\|_{L^1}^{4/3}\,d\t \leq C\d_0 + C(\e_r \d_R)^{1/6} + C\d_0\left(\frac{\d_R}{\e_r}\right)^{4/3}\frac{1}{\e_r}.
\]
Thus, we have 
\[
\int_0^t |I_1|\,d\t\leq \frac{1}{100}(D_{u_1} + D_{\th_1}) + C\d_0 + C(\e_r \d_R)^{1/6} + C\d_0\left(\frac{\d_R}{\e_r}\right)^{4/3}\frac{1}{\e_r}.
\]
For $I_2$, following the calculations in $I_1$, we find that 
\[
|I_2| \leq \frac{1}{100}(D_{u_1} + D_{\th_1}) + C\e^{2/3}\sum_{i=1}^2 \|Q_i^R\|_{L^1}^{4/3}.
\]
It follows from \eqref{def:QR} that
\begin{equation}\label{est:QiR}
    \begin{aligned}
        |Q^R_1| + |Q_2^R| \le C &\left[|u^R_{\xi\xi}|+|u^R_{\xi}||v^R_{\xi}|+|\theta^R_{\xi\xi}|+|\theta^R_{\xi}||v^R_{\xi}|+|u^R_{\xi}|^2\right]\\
\le C &\left[\big|\big(u^R_{\xi\xi},\theta^R_{\xi\xi}\big)\big|+\big|\big(v^R_{\xi},u^R_{\xi},\theta^R_{\xi}\big)\big|^2\right],
    \end{aligned}
\end{equation}
 which implies
$$
\||Q^R_1| + |Q^R_2|\|_{L^1}\leq C\Big[\big\|\big(u^R_{\xi\xi},\theta^R_{\xi\xi}\big)\big\|_{L^1}+\big\|\big(v^R_{\xi},u^R_{\xi},\theta^R_{\xi}\big)\big\|_{L^2}^2\Big].
$$
By virtue of Lemma \ref{lem:rarefaction}, we have 
\begin{align*}
\big\|\big(u^R_{\xi\xi},\theta^R_{\xi\xi}\big)\big\|_{L^1}\le  \big\|\big(u^R_{\xi\xi},\theta^R_{\xi\xi}\big)\big\|^{\frac{1}{9}}_{L^1} \big\|\big(u^R_{\xi\xi},\theta^R_{\xi\xi}\big)\big\|^{\frac{8}{9}}_{L^1} \leq C(\e_r\delta_R)^{\frac{1}{9}}(\delta_R^{\frac{8}{9}}+\delta_R^{\frac{1}{9}})(1+t)^{-\frac{7}{9}},    
\end{align*}
where we have taken $q=8$ in Lemma \ref{lem:rarefaction}, and
\begin{equation*}
\begin{aligned}
\big\|\big(v^R_{\xi},u^R_{\xi},\theta^R_{\xi}\big)\big\|^2_{L^2} \le \big\|\big(v^R_{\xi},u^R_{\xi},\theta^R_{\xi}\big)\big\|^{\frac{1}{4}}_{L^2} \big\|\big(v^R_{\xi},u^R_{\xi},\theta^R_{\xi}\big)\big\|^{\frac{7}{4}}_{L^2} \leq C(\e_r^{\frac{1}{2}}\delta_R)^{\frac{1}{4}}\delta_R^{\frac{7}{8}}(1+t)^{-\frac{7}{8}}.
\end{aligned}
\end{equation*}
Combining the above estimates, we obtain
\begin{align*}
   \||Q^R_1| + |Q^R_2|\|^{\frac{4}{3}}_{L^1} \leq&C(\e_r\delta_R)^{\frac{4}{27}}\delta_R^{\frac{32}{27}}(1+t)^{-\frac{28}{27}}+ C(\e_r\delta_R)^{\frac{1}{6}}\delta_R^{\frac{4}{3}}(1+t)^{-\frac{7}{6}}, 
\end{align*}
and therefore,
\[
\int_0^t  \||Q^R_1| + |Q^R_2|\|^{\frac{4}{3}}_{L^1}\,d\t \leq C(\e_r\d_R)^{1/9}.
\]
Finally, for $I_3$, we use H\"older's inequality and \eqref{est:apriori} to have 
\[
\begin{aligned}
 |I_3| \leq C\intRp \sum_{i=1}^2|Q_i^C||(\phi,\chi)|\,d\x \leq C\sum_{i=1}^2\|Q_i^C\|_{L^2} \|(\phi,\chi)\|_{L^2} \leq C\e\sum_{i=1}^2\|Q_i^C\|_{L^2}.
\end{aligned}
\]
Since it follows from \eqref{est:QC} that
\begin{align*}
   \|Q_1^C\|_{L^2}\le C \delta_C (1+t)^{-\frac 54},\qquad \|Q_2^C\|_{L^2}\le C \delta_C (1+t)^{-\frac 74},
\end{align*} 
we obtain 
\[
\int_0^t |I_3|\,d\x \leq C\d_C.
\]
Combining the above estimates, we have 
\[
\int_0^t |B_1^I + B_2^I|\,d\t \leq \frac{1}{50}(D_{u_1} + D_{\th_1}) + C\d_0 + C(\e_r \d_R)^{1/9} + C\d_0\left(\frac{\d_R}{\e_r}\right)^{4/3}\frac{1}{\e_r}.
\]
$\bullet$ {\bf Estimates of $B^{I}_3$:} By the interpolation inequality, we obtain 
\[
\begin{aligned}
|B_3^I| &\leq C\||u_\x^{BL}||(v^R-v_*,\th^R -\th_*)|\|_{L^2}\|(\phi,\chi)\|_{L^4}^2 \\
&\leq C\||u_\x^{BL}||(v^R-v_*,\th^R -\th_*)|\|_{L^2}\|(\phi,\chi)\|_{L^2}^{3/2}\|(\phi_\x,\chi_\x)\|_{L^2}^{1/2}   \\
&\leq  C\|J_1|\|_{L^2}^{4/3} + C\e^6(D_{v_1} + D_{\th_1}),
\end{aligned}
\]
where $J_1$ is defined in \eqref{def:J}. Using Lemma~\ref{lem:WInteraction}, we have 
\[
\begin{aligned}
 \int_0^t |B_3^I|\,d\t &\leq C\d_R^{1/6}(\d_R\e_r)^{1/12}\d_{BL}^{5/6} + C(\e_r\d_R)^{1/4}\d_{BL} +    C\e^6\int_0^t (D_{v_1} + D_{\th_1})\,d\t\\
  &\leq C\d_0^{5/6} +  C\e^6\int_0^t (D_{v_1} + D_{\th_1})\,d\t.
\end{aligned}
\]
$\bullet$ {\bf Estimates of $B^{I}_4$:}
By \eqref{def:a} and the interpolation inequality, we have
\begin{align*}
    |B_4^I| &\leq \frac{C}{\d_S^{1/2}}\||u^S_\x||(\vbar - v^S,\thbar - \th^S)|\|_{L^2}\|(\phi,\chi)\|_{L^4}^2\\
    &\leq \frac{C}{\d_S^{1/2}}\||u^S_\x||(\vbar - v^S,\thbar - \th^S)|\|_{L^2}\|\|(\phi,\chi)\|_{L^2}^{3/2}\|(\phi_\x,\chi_\x)\|_{L^2}^{1/2}\\
    &\leq \frac{C}{\d_S^{2/3}}\||u^S_\x||(\vbar - v^S,\thbar - \th^S)|\|_{L^2}^{4/3} + \e^6\|(\phi_\x,\chi_\x)\|_{L^2}^2.
\end{align*}
Using the same arguments as for $Y_2$, we obtain
\[
\begin{aligned}
   |B_4^I| &\leq \frac{C}{\d_S^{2/3}}\||u^S_\x|(|v^{BL}-v_*| + |v^R - v_m| + |v^C - v^*|)\|_{L^2}^{4/3} + C\e^6(D_{v_1} + D_{\th_1}).
\end{aligned}
\]
Using \eqref{l2J3}, \eqref{l2J5} and \eqref{l2J6}, we obtain
\[
\begin{aligned}
\int_0^t |B_4^I|d\t &\leq \frac{C}{\d_S^{2/3}}(\d_{BL}^{4/3}\d_S + \d_{BL}^{1/3}\d_S^2) + \frac{C}{\d_S^{2/3}}\left(\d_R^{4/3}\d_S + \d_S^{8/3}\left(\frac{\d_R^2}{\e_r}\right)^{4/3}\frac{1}{\e_r} \right)
\\
&\quad + \frac{C}{\d_S^{2/3}}( \d_S \d_C^{8/3}+\d_C^{4/3}\d_S^{8/3})\\
&\leq C\d_0^{1/3} + C\d_0\d_R^{4/3}\left(\frac{\d_R}{\e_r}\right)^{4/3}\frac{1}{\e_r} + C\e^6 \int_0^t (D_{v_1} + D_{\th_1})\,d\t.
\end{aligned}
\]
Combining the above estimates and using $\e_r <1$, we conclude that 
\[
\int_0^t |B^I|\,d\t \leq C\d_0^{1/3} + C(\e_r\d_R)^{1/9} + C\d_0\left(\frac{\d_R}{\e_r}\right)^{4/3}\frac{1}{\e_r} + C\e^2\int_0^t D_{v_1} \,d\t + \frac{1}{40}\int_0^t (D_{u_1} + D_{\th_1})\,d\t.
\]
\end{proof}
\subsubsection{Estimates of $\mathcal{J}^{bd}$}

\begin{lem}\label{lem:bd}
Under the assumptions of Proposition~\ref{prop:ap}, there exist positive constants $c$ and $C$ such that
	\[
		\int_0^t \mathcal{J}^{bd} \,d\t \leq C\d_C + C e^{-c\d_S \beta} + C\e^2 \int_0^t \left(\norm{\psi_{\x\x}}_{L^2}^2+ \norm{\chi_{\x\x}}_{L^2}^2 \right)\,d\t.
	\]
\end{lem}
\begin{proof}
First, we decompose $\mathcal{J}^{bd}$ into three terms as follows:
\begin{align*}
& \mathcal{J}^{bd}_1 :=-\s_- a \bar{\theta}\eta(U|\Ubar)\Big|_{\x=0} + a(u-\ubar)(p-\pbar) \Big|_{\x=0}, \\
&\mathcal{J}^{bd}_2:= -a\m \left[(u-\ubar)\left(\frac{u_\x}{v} - \frac{\ubar_\x}{\vbar}\right) \right] \Big|_{\x=0}, &&\mathcal{J}^{bd}_3 :=-a\kappa \left[ \frac{\th -\thbar}{\th} \left(\frac{\th_\x}{v} - \frac{\thbar_\x}{\vbar}\right)\right] \Big|_{\x=0}.
\end{align*}
By \eqref{def:RE}, we have
\begin{align}\label{est:p1}
\int_{0}^{t}\mathcal{J}_1^{bd} d\t \leq C\int_{0}^{t} |(\phi,\psi,\chi)|^2\big|_{\x=0} d\t.
\end{align}
By definition of $\overline{U}$ in \eqref{def:Ubar}, we have
\begin{align*}
|U-\overline{U}|\big|_{\xi=0} = |U^C + U^S - U_m - U^*|\big|_{\x = 0} \leq  |U^C-U_m|\big|_{\xi=0} + |U^S-U^*|\big|_{\xi=0}.
\end{align*}
From Lemma~\ref{lem:CD}, it follows that
\begin{align}\label{est:CDbd}
|U^C-U_m|\big|_{\xi=0} \leq C\delta_Ce^{\frac{-C_1(\sigma_-t)^2}{1+t}} \leq C\delta_Ce^{-Ct}.
\end{align}
On the other hand, by $|X(t)| \leq C\varepsilon t$ and taking $\e>0$ sufficiently small, we obtain
\[-(\sigma-\sigma_-)t-X(t)-\beta \le -\frac{(\sigma-\sigma_-)t}{2}-\beta<0.\]
Therefore, by Lemma~\ref{lem:VS}, we have 
\begin{equation}\label{est_boundary}
	|U^S-U^*|\big|_{\xi=0}\le |U^S(-(\sigma-\sigma_-)t-X(t)-\beta)-U^*|\le C\delta_Se^{-C\delta_S t}e^{-C\delta_S\beta},
\end{equation}
Substituting \eqref{est:CDbd} and \eqref{est_boundary} into \eqref{est:p1}, we obtain
\begin{align}\label{est:Jbd1}
\int_{0}^{t}\mathcal{J}^{bd}_1 d\t \leq C\delta_C^2 + C\d_S e^{-C\delta_S\beta}.
\end{align}
For $\mathcal{J}^{bd}_2$, observe that
\begin{align}\label{p3}
\begin{aligned}
\int_{0}^{t} \mathcal{J}^{bd}_2 d\t \leq C \int_{0}^{t} |\psi \psi_\x|\big|_{\xi=0} d\t+ C \int_{0}^{t} |\psi \phi\bar{u}_\xi|\big|_{\xi=0} d\t.
\end{aligned}
\end{align}
For the first term on the right-hand side of \eqref{p3}, we use the interpolation inequality, Young's inequality, and \eqref{est:apriori} to obtain
\begin{align*}
& \int_{0}^{t} |\psi \psi_\x|\big|_{\xi=0} d\t \leq \int_0^t \|\psi_\x\|_{L^\infty}|\psi|\big|_{\x=0}\,d\t \leq C\int_0^t\|\psi_\x\|_{L^2}^{1/2}\|\psi_{\x\x}\|_{L^2}^{1/2}|\psi|\big|_{\x=0}\,d\t\\
&\quad \leq C\int_{0}^{t} |\psi(\t,0)|^{4/3} d\t + C\int_{0}^{t} \|\psi_\xi\|^2_{L^2} \|\psi_{\xi\xi}\|^2_{L^2} d\t  \\
&\quad \leq C(\delta_C + \d_S^{1/3}e^{-C\delta_S \beta}) + C\e^2 \int_{0}^{t} \|\psi_{\x\x}\|^2_{L^2} d\t.
\end{align*}
For the last inequality, we used \eqref{est:CDbd} and \eqref{est_boundary}, which yield
\[
\int_{0}^{t} |\psi(\t,0)|^{4/3} d\t \leq C\int_0^t |u^C-u_m|^{4/3} + |u^S-u^*|^{4/3}\,d\t \leq C\d_C^{4/3} + C\d_S^{1/3}e^{-C\d_S\b},
\]
For the second term on the right-hand side of \eqref{p3}, we use \eqref{est:p1} and \eqref{est:Jbd1} to have
\begin{align*}
\int_{0}^{t} |\psi \phi\bar{u}_\xi|\big|_{\xi=0} \,d\t \leq C\int_{0}^{t} |(\phi,\psi)|^2\big|_{\x=0}\,d\t  \leq C\delta_C^2 + Ce^{-C\delta_S\beta}.
\end{align*}
Therefore, we have
\begin{align*}
\int_{0}^{t} \mathcal{J}^{bd}_2 d\t  \leq  C(\delta_C + \d_S^{1/3}e^{-C\delta_S \beta}) + C\e^2 \int_{0}^{t} \|(u-\bar{u})_{\xi\x}\|^2_{L^2} d\t.
\end{align*}
Similarly, $\mathcal{J}^{bd}_3$ can be estimated as 
\begin{align*}
\int_{0}^{t} \mathcal{J}^{bd}_3 d\t  \leq  C(\delta_C + \d_S^{1/3}e^{-C\delta_S \beta}) + C\e^2 \int_{0}^{t} \|(\theta-\bar{\theta})_{\xi\x}\|^2_{L^2} d\t.
\end{align*}
\end{proof}

\subsubsection{Conclusion}
By Lemma~\ref{lem:decomp} and Lemma~\ref{lem:leading}, we have 
\begin{align*}
	\frac{d}{dt}\intRp a \thbar \eta(U|\Ubar)\,d\x &\leq -\frac{\d_S}{2M}|\dot{X}(t)|^2 + CB^C(U) + B^S(U) + B^{rem}(U) + B^{I}(U) + \mathcal{J}^{bd}(U)\\
	&\quad - C_2 (G^{BL}(U) + G^R(U)) - G^{new}(U) - D(U)\\
    &\leq -\frac{\d_S}{4M}|\dot{X}(t)|^2 + B^{rem}(U) + B^I(U) \\
    &\quad + \mathcal{J}^{bd}(U) -C_2(G^{BL}(U) + G^R(U)) - \frac{1}{4}G^{new}(U) - \frac{1}{4}D(U).
\end{align*}
Using Lemma~\ref{lem:Brem} and the smallness of $\d_0, \e_0$, and $\e_r$, we have
\begin{align*}
\begin{aligned}
    \frac{d}{dt}\intRp a \thbar \eta(U|\Ubar)\,d\x 
    \leq &  -\frac{\d_S}{4M}|\dot{X}
    (t)|^2-\frac{1}{10}G^{new}(U)-\frac{C_S}{2} G^S(U) - \frac{C_2}{2}(G^{BL}(U) + G^R(U))  \\
    &\quad - \frac{1}{10}D(U)  + B^I(U) + \mathcal{J}^{bd}(U)+CB^C(U).
\end{aligned}
\end{align*}
Integrating the above inequality over $[0,T]$, and using Lemma~\ref{lem:BIU} and Lemma~\ref{lem:bd} we obtain
\begin{align*}
\begin{aligned}
&\sup_{t\in [0,T]}\intRp \eta(U|\Ubar)\,d\x + \delta_S\int_0^{T} |\dot{X}|^2 \,dt + \int_0^T (G^{BL} + G^R + G^S +D_{u_1}+ D_{\th_1})\,dt \\
&\quad \leq  \intRp \eta(U_0|\Ubar(0,\cdot))\,d\x + C\d_0^{1/3} + C(\e_r\d_R)^{1/9} + C\d_0\left(\frac{\d_R}{\e_r}\right)^{4/3}\frac{1}{\e_r} + Ce^{-c\d_S\b} \\
&\quad \quad + C\e^2\int_0^t (D_{v_1} + D_{u_2} + D_{\th_2})\,dt+C\int_0^T B^C\,dt.
\end{aligned}
\end{align*}
Now, choosing $\e_r>0$ as 
\begin{equation}\label{def:er}
   \e_r := \d_0^{1/4},
\end{equation}
and using Lemma~\ref{lem:L2CD} and the fact that 
\(
\intRp \eta(U|\Ubar)\,d\x \sim \|U - \Ubar\|_{L^2}^2,
\)
we obtain the desired estimate in Lemma~\ref{lem:L2ap}.

\section{Higher Order Estimates}\label{Sec:H1}
\setcounter{equation}{0}
In this section, we provide the higher order estimates and complete the proof of Proposition~\ref{prop:ap}. 

It follows from \eqref{eq:NSF} and \eqref{eq:composite} that the equations for the perturbations $(\phi, \psi, \chi)$ are given by
\begin{equation}\label{eq:NSFp}
	\begin{cases}
		&\phi_t - \s_- \phi_\x - \psi_\x = \dot{X}(t)v^S_\x, \quad t>0, \quad \x >0, \\
		&\psi_t - \s_- \psi_\x + (p - \pbar)_\x = \mu\left(\frac{u_\x}{v} - \frac{\ubar_\x}{\vbar}\right)_\x - Q_1 + \dot{X}(t)u_\x^S,\\
		& \frac{R}{\gamma-1}\big(\chi_t - \s_- \chi_\x \big) +   (p u_\x - \pbar \ubar_\x) = \kappa \left(\frac{\th_\x}{v} - \frac{\thbar_\x}{\vbar}\right)_\x + \mu \left(\frac{u_\x^2}{v} - \frac{\ubar_\x^2}{\vbar}\right) - Q_2+\frac{R}{\gamma-1}\dot{X}(t)\th^S_\x.
	\end{cases}
\end{equation}
\subsection{Higher Order Estimates on $v-\vbar$}
\begin{lem}\label{lem:hv}
Under the assumptions of Proposition~\ref{prop:ap}, and for any $\nu>0$, there exist positive constants $c$, $C>0$ (independent of $\nu$), and $C_\nu>0$ such that
\begin{align*}
    &\sup_{t \in [0,T]}\|\phi_\x(t,\cdot)\|_{L^2}^2 + \int_0^T D_{v_1} \,dt \\
	&\quad \leq C\big(\|\psi(0,\cdot)\|_{L^2}^2 + \|\phi_\x(0,\cdot)\|_{L^2}^2 \big) + C\int_0^T (\d_S|\dot{X}|^2 + G^{BL} + G^R + G^S + D_{\th_1})\,dt \\
	&\qquad +C_\nu \int_0^T D_{u_1}\,dt + C(\nu+\e) \int_0^T D_{u_2}\,dt + C\int_0^T B^C\,dt + C\int_0^T \intRp |Q_1|^2\,d\x \,dt\\
	&\qquad + C(\d_R\e_r^{1/2} + \d_C^2 + e^{-c\d_S \b}).
\end{align*}
\end{lem}
\begin{proof}
Before proving Lemma~\ref{lem:hv}, we denote 
\[
D_v := \intRp \frac{R\th \phi_\x^2}{v}\,d\x  \sim D_{v_1}.
\]
Differentiating \eqref{eq:NSFp}$_1$ with respect to $\x$ and multiplying $\mu \phi_\x$, we have 
\begin{equation}\label{est:hv1}
	\mu \left(\frac{\phi_\x^2}{2}\right)_t - \s_- \mu \left(\frac{\phi_\x^2}{2}\right)_\x = \mu \dot{X}v^S_{\x\x}\phi_\x + \mu \phi_\x \psi_{\x\x}.
\end{equation}
In order to remove the term $\mu \phi_\x \psi_{\x\x}$, we make use of the momentum equation as follows.

Multiplying \eqref{eq:NSFp}$_2$ by $-v\phi_\x$, we obtain
\begin{equation}\label{est:hv2}
\begin{aligned}
	-v\phi_\x \psi_t + \s_- v \phi_\x \psi_\x &= -\dot{X}v\phi_\x u^S_\x + v\phi_\x (p-\pbar)_\x - \mu \phi_\x \psi_{\x\x} + \mu \frac{\phi_\x \psi_\x v_\x}{v} \\
	&\qquad \quad + \mu \frac{\ubar_{\x\x}\phi_\x \phi}{\vbar} + \mu v \phi_\x \ubar_\x \left(\frac{v_\x}{v^2} - \frac{\vbar_\x}{\vbar^2}\right) + v\phi_\x Q_1.
\end{aligned}
\end{equation}
Here, we used the fact that 
\[
\left(\frac{u_\x}{v} - \frac{\ubar_\x}{\vbar}\right)_\x = \frac{\psi_{\x\x}}{v} - \frac{\psi_\x v_\x}{v^2} - \frac{\ubar_{\x\x}\phi}{v\vbar} -\ubar_\x \left(\frac{v_\x}{v^2} - \frac{\vbar_\x}{\vbar^2}\right).
\]
By adding \eqref{est:hv1} and \eqref{est:hv2}, we obtain
\begin{equation}\label{est:hv3}
	\begin{aligned}
	&\left(\mu \frac{\phi_\x^2}{2} - v\psi \phi_\x\right)_t - \left(\mu \s_- \frac{\phi_\x^2}{2} - v \psi \phi_t \right)_\x + (v_t - \s_- v_\x)\psi \phi_\x - (v_\x \psi + v \psi_\x)(\phi_t - \s_- \phi_\x)\\
	&\,\, = \dot{X}\phi_\x \big(\m v^S_{\x\x} - vu_\x^S)  +v\phi_\x (p-\pbar)_\x + \mu \frac{\phi_\x \psi_\x v_\x}{v} + \mu \frac{\ubar_{\x\x}\phi_\x \phi}{\vbar} + \mu v \phi_\x \ubar_\x \left(\frac{v_\x}{v^2} - \frac{\vbar_\x}{\vbar^2}\right)+v\phi_\x Q_1.
\end{aligned}
\end{equation}
From \eqref{eq:NSF} and \eqref{eq:NSFp}, we find that 
\[
(v_t - \s_- v_\x)\psi \phi_\x - (v_\x \psi + v \psi_\x)(\phi_t - \s_- \phi_\x) = -\dot{X}v_\x^S(v_\x\psi + v\psi_\x) + (u_\x\psi \phi_\x - v_\x \psi \psi_\x - v\psi_\x^2).
\]
Combining this with the identity:
\[
(p-\pbar)_\x = -\frac{R\th \phi_\x}{v^2}+\frac{R\chi_\x}{v} - \frac{R\thbar_\x \phi}{v\vbar} - R\vbar_\x\left(\frac{\th}{v^2} - \frac{\thbar}{\vbar^2}\right),
\]
and integrating \eqref{est:hv3} over $ [0,t]\times\Rp$, we obtain
\begin{align*}
	&\intRp \left(\mu\frac{|\phi_\x|^2}{2} - v\psi \phi_\x\right)\,d\x + \int_0^t\intRp \frac{R\th \phi_\x^2}{v}\,d\x\,d\tau + \int_0^t\left. \left(\mu \s_- \frac{\phi_\x^2}{2} - v \psi \phi_t  \right)\right|_{\x=0}\,d\tau\\
    &\quad =:\intRp \left(\mu\frac{|\phi_\x|^2}{2} - v\psi \phi_t\right)\bigg|_{\xi=0}\,d\tau +\sum_{i=1}^5 N_i,\\
\end{align*}
where
\begin{align*}
    &\begin{aligned}
        &N_1 = \int_0^t\dot{X} \intRp \Big( v_\x^S(v_\x\psi + v\psi_\x) + \phi_\x ( \mu v^S_{\x\x} - vu_\x^S)\Big)\,d\x\,d\tau,\\
        &N_2=\int_0^t\intRp (-u_\x \psi \phi_\x + v_\x \psi \psi_\x + v\psi_\x^2)\,d\x\,d\tau,
    \end{aligned}\\
    &\begin{aligned}
        &N_3= \int_0^t\intRp \left[R\phi_\x \chi_\x - \frac{R\thbar_\x \phi \phi_\x}{\vbar} - R\phi_\x \vbar_\x \left(\frac{\th}{v^2} - \frac{\thbar}{\vbar^2} \right)+\mu \frac{\ubar_{\x\x} \phi_\x \phi}{\vbar}\right.\\
        &\phantom{N_3= \int_0^t\intRp \Big[} \left.+ \mu v \phi_\x \ubar_\x \left(\frac{v_\x}{v^2} - \frac{\vbar_\x}{\vbar^2}\right)\right]d\x d\tau,
    \end{aligned}\\
    &\begin{aligned}
        &N_4=  \int_0^t\intRp  \mu \frac{\phi_\x \psi_\x v_\x}{v} \,d\x\,d\tau, &&  N_5= \int_0^T\intRp v\phi_\x Q_1 \,d\x\,d\tau. 
    \end{aligned}
\end{align*}
$\bullet$ (Estimates of boundary terms) 
We first estimate the boundary terms. In order to control the term $\phi_\x^2\big|_{\x=0}$ we use the mass equation \eqref{eq:NSFp}$_1$ to have
\[
|\phi_\x| \big|_{\x=0} \leq C|\phi_t|\big|_{\x=0} + C|\psi_\x| \big|_{\x=0} + C|\dot{X}(t)||v^S_\x|\big|_{\x=0}.
\]
Using $|\dot{X}|\leq C\e$ and Young's inequality, we obtain 
\[
 \left. \left(\mu \s_- \frac{\phi_\x^2}{2} - v \psi \phi_t  \right)\right|_{\x=0} \leq C|\psi|^2 \big|_{\x=0} + C|\phi_t|^2\big|_{\x=0} + C|\psi_\x|^2 \big|_{\x=0} + C\e_1^2|v^S_\x|^2\big|_{\x=0}.
\]
We can easily find that 
\begin{equation}\label{est:hvbd}
	\int_0^t |\psi|^2 \big|_{\x=0}\,d\tau \leq C(\d_C^2 + \d_S e^{-C\d_S \b}), \quad \text{and} \quad C\e^2\int_0^t|v_\x^S|^2 \big|_{\x=0}\,d\tau \leq C\e^2\d_S^3e^{-C\d_S \b}.
\end{equation}
Moreover, thanks to the interpolation inequality, we have 
\[
\int_0^t |\psi_\x|^2 \big|_{\x=0}\,d\tau \leq \int_0^t \|\psi_\x\|_{L^2} \|\psi_{\x\x}\|_{L^2}\,d\tau \leq C_\nu \int_0^t D_{u_1}\,d\tau + \nu \int_0^t D_{u_2}\,d\tau,
\]
where $\nu>0$ is an arbitrary constant 
and $C_\nu$ denotes a constant depending on $\nu>0$. In order to estimate $|\phi_t|^2\big|_{\x=0}$, we use the inflow condition as follows:
\[
|\phi_t|^2\big|_{\x=0} = |(\vbar(t,0) - v_-)_t| \leq |(v^C - v_m)_t|^2\big|_{\x=0} + |(v^S - v^*)_t|^2\big|_{\x=0}.
\]
Since we have 
\[
|v^C_t| \big|_{\x=0} \leq \left|\s_- v^C_\x + u^C_\x \right| \leq \left. \frac{C\d_C}{\sqrt{1+t}}e^{-\frac{2C_1(\x + \s_- t)^2}{1+t}}\right|_{\x=0} \leq C\d_C e^{-Ct},
\]
and 
\[
|(v^S -v^*)_t|\Big|_{\x=0} \leq C|(v^S)'(-(\s - \s_-)t - X(t) - \b)||\dot{X} + (\s - \s_-)| \leq C\d_S^2e^{-C\d_S((\s - \s_-)t + \b)},
\]
we obtain
\[
\int_0^T |\phi_t|^2\big|_{\x=0} \,dt \leq C(\d_C^2 + \d_S e^{-C\d_S \b}).
\]
Hence, we have 
\[
\int_0^t \left. \left(\mu \s_- \frac{\phi_\x^2}{2} - v \psi \phi_t  \right)\right|_{\x=0}\,d\t \leq C(\d_C^2 + \d_S e^{-C\d \b}) + C_\nu \int_0^t D_{u_1} \,d\t+ \nu \int_0^t D_{u_2} \,d\t. 
\]
$\bullet$ (Estimates of $N_1$) 
Using $|\vbar_\x|\leq C(\d_0 + \d_R\e_r) \leq C$ and $|\psi|\leq C\e$, we have 
\begin{align*}
	v_\x^S(v_\x\psi + v\psi_\x) + \phi_\x ( \mu v^S_{\x\x} - vu_\x^S) &\leq C|v_\x^S|\big(|\phi_\x|(1+|\psi|) + |\vbar_\x||\psi| + |\psi_\x| \big)\\
	&\leq C|v_\x^S|\big(|\phi_\x| + |\psi| + |\psi_\x|).
\end{align*}
Using Young's inequality, we obtain 
\[
|N_1| \leq C\d_S \int_0^t |\dot{X}|^2\,d\t + C\d_S\int_0^t \big(G^S + D_v + D_{u_1})\,d\t.
\]
$\bullet$ (Estimates of $N_2$) 
For $N_2$, we use Young's inequality to have 
\begin{align*}
	|N_2| &= \left|\int_0^t\intRp (v_\x \psi \psi_\x + v\psi_\x^2 - u_\x \psi \phi_\x)\,d\x\,d\t \right| \\
	&\leq C \int_0^t \intRp |\phi_\x||\psi||\psi_\x|\,d\x\,d\t + C \int_0^t\intRp |(\vbar_\x, \ubar_\x)||\psi||\psi_\x|\,d\x\,d\t + C\int_0^t \intRp |\psi_\x|^2 \,d\x\,d\t\\
	&\leq \frac{1}{100} \int_0^t D_v\,d\t+ C\int_0^t D_{u_1}\,d\t + C\int_0^t\intRp |(\vbar_\x, \ubar_\x)|^2|\psi|^2\,d\x\,d\t\\
	&\leq \frac{1}{100}\int_0^t D_v\,d\t + C\int_0^t(D_{u_1} + G^S)\,d\t + C\d_C\int_0^t B^C\,d\t\\
	&\qquad + C\int_0^t\intRp (|v^{BL}_\x|^2 + |v^R_\x|^2)|\psi|^2\,d\x\,d\t.
\end{align*}
Now, it remains to control the last term of the above inequality. First, we use the interpolation inequality to have
\begin{align*}
	\int_0^t\intRp |v^R_\x|^2|\psi|^2\,d\x\,d\t&\leq \int_0^t\|v^R_\x\|_{L^2}^2\|\psi\|_{L^\infty}^2\,d\t \leq C\int_0^t\|v^R_\x\|_{L^2}^2\|\psi\|_{L^2}\|\psi_{\x}\|_{L^2}\,d\t\\
	&\leq C\e \int_0^t\|v^R_\x\|_{L^2}^4 \,d\t+ C\e \int_0^t D_{u_1}\,d\t.
\end{align*}
By Lemma~\ref{lem:rarefaction}, we have
\[
\|v^R_\x\|_{L^2}^4 \leq \|v^R_\x\|_{L^2}\|v^R_\x\|_{L^2}^3 \leq C\d_R\e_r^{1/2} (\d_R)^{3/2}(1+t)^{-3/2}.
\]
Consequently, we obtain
\[
\int_0^t \intRp |v^R_\x|^2|\psi|^2\,d\x\,d\tau \leq C\d_R^{5/2}\e_r^{1/2} +  C\e \int_0^t D_{u_2}\,d\t.
\]
On the other hand, in order to control the term related to the boundary layer solution, as in \cite{KawaOutBL}, we use the Poincar\'e-type inequality as follows. First, note that 
\begin{equation}\label{est:KPoin}
	\psi(\tau,\x) = \psi(0, \x) + \int_0^\x \psi_\x(\tau,\zeta)\,d\zeta \leq \psi(0, \x) + \sqrt{\x}\|\psi_\x\|_{L^2(\Rp)}.
\end{equation}
Using \eqref{est:KPoin}, \eqref{est:hvbd}, and Lemma~\ref{lem:BL}, we obtain
\begin{align*}
	\int_0^t \intRp |v^{BL}_\x|^2|\psi|^2\,d\x \,d\t &\leq C \int_0^t  |\psi(0, \x)|^2\,d\tau + C\int_0^t \|\psi_\x\|_{L^2}^2 \intRp \x|v^{BL}_\x|^2\,d\x\\
	&\leq C(\d_C^2 + \d_S e^{-C\d_S \b}) + C\int_0^t D_{u_1} \intRp \frac{\d_{BL}^4 \x}{(1+\d_{BL}\x)^4}\,d\x\,d\t\\
	&\leq C(\d_C^2 + \d_S e^{-C\d_S \b}) + C\d_{BL}^3 \int_0^t D_{u_1}\,d\t.
\end{align*}
Thus, we have 
\[
\begin{aligned}
    |N_2| &\leq  \frac{1}{100}\int_0^t D_v\,d\t  + C\int_0^t (D_{u_1} + G^S)\,d\t+C\e\int_0^t D_{u_2}\,d\t + C\d_C\int_0^t B^C\,d\t\\
&\quad +C(\d_R \e_r^{1/2} + \d_C^2 + \d_Se^{-C\d_S\b}).
\end{aligned}
\]
$\bullet$ (Estimates of $N_3$) 
Observe that 
\begin{align*}
\begin{aligned}
	|N_3| \leq& C\int_0^t\intRp |\phi_\x||\chi_\x|\,d\x\,d\t + C\int_0^t\intRp |\phi_\x||(\thbar_\x,\vbar_\x,\ubar_{\x\x},\ubar_\x \vbar_\x)||(\phi, \chi)| \,d\x\,d\t \\
    &+ C\int_0^t\intRp |\ubar_\x||\phi_\x|^2\,d\x\,d\t.
    \end{aligned}
\end{align*}
Using Young's inequality and $|u_{\x\x}^{BL}| \leq C|u_\x^{BL}|$,  $|u_{\x\x}^R| \leq C|u_\x^R|$, $|u_{\x\x}^S| \leq C|u_\x^S|$, and Lemma~\ref{lem:CD}, we obtain
\begin{align*}
	|N_3| \leq \big(\frac{1}{100} + C\d_0 + C\d_R \e_r\big) \int_0^t D_v d\t+ C\int_0^t (G^{BL} + G^R + G^S + D_{\th_1})d\t +C\d_C\int_0^t B^Cd\t.
\end{align*}
$\bullet$ (Estimates of $N_4$)
We use the interpolation inequality to obtain
\begin{align*}
	|N_4| &\leq \int_0^t \intRp |\phi_\x|^2|\psi_\x| \,d\x\,d\t + \int_0^t \intRp |\phi_\x||\psi_\x||\vbar_\x|\,d\x\,d\t\\
	&\leq \int_0^t \|\phi_\x\|_{L^2}^2\|\psi_\x\|_{L^2}^{1/2}\|\psi_{\x\x}\|_{L^2}^{1/2}\,d\t + \int_0^t \intRp |\phi_\x||\psi_\x||\vbar_\x|\,d\x\,d\t\\
	&\leq C\e^{4/3}\int_0^t  D_{u_2}\,d\t + C \int_0^t D_{u_1}\,d\t + \frac{1}{100}\int_0^t D_v\,d\t.
\end{align*}
$\bullet$ (Estimates of $N_5$) 
Finally, we have 
\begin{align*}
   |N_5|\leq &C \int_0^t \intRp  |Q_1|^2\,d\x\,d\t+ \frac{1}{100}\int_0^t D_v \,d\t.
\end{align*}
 
Therefore, combining all these estimates, together with the fact:
\[
\intRp \left(v\psi \phi_\x -  v_0\psi_0 \phi_{0\x}\right) \,d\x \leq \frac{\mu}{4}\intRp |\phi_\x|^2\,d\x + C\big(\|\psi\|_{L^2}^2 + \|(\phi_{0\x},\psi_0)\|_{L^2}\big),
\]
we obtain the desired estimate in Lemma \ref{lem:hv}.
\end{proof}
\subsection{Higher Order Estimates on $u-\ubar$}
\begin{lem}\label{lem:hu}
Under the assumptions of Proposition~\ref{prop:ap}, there exist positive constants $c$ and $C$ such that
\begin{align*}
	&\sup_{t \in [0,T]} \|\psi_{\x}(t,\cdot)\|_{L^2}^2 + \int_0^T D_{u_2}\,dt \\
	&\quad \leq C\|\psi_{\x}(0,\cdot)\|_{L^2}^2 +C \int_0^T \big(\d_S|\dot{X}|^2 + G^{BL} + G^R + G^S + D_{v_1} + D_{u_1} + D_{\th_1}\big)\,dt \\
	&\qquad + C\int_0^T B^C\,dt + C\int_0^T |Q_1|^2\,dt + Ce^{-c\d_S \b}.
\end{align*}
\end{lem}
\begin{proof}
Throughout the proof of Lemma~\ref{lem:hu}, denote 
\[
\mathcal{D}_{u_2}:= \mu \intRp \frac{1}{v}|\psi_{\x\x}|^2\,d\x \sim D_{u_2}.
\]
Multiplying \eqref{eq:NSFp}$_2$ by $-\psi_{\x\x}$ and integrating over $ [0,t]\times\Rp$, we have 

\begin{equation*}
	\begin{aligned}
		&\intRp \frac{|\psi_\x|^2}{2} \,d\xi + \int_0^t\left. \left( \psi_t \psi_\x - \s_- \frac{|\psi_\x|^2}{2} \right)\right|_{\x = 0} \,d\t\\
		&\quad = \intRp \frac{|\psi_\x(0
        ,\x)|^2}{2} \,d\xi -\int_0^t\dot{X}\intRp u^S_\x \psi_{\x\x}\,d\x\,d\t + \int_0^t\intRp (p - \pbar)_\x \psi_{\x\x}\,d\x\,d\t   \\
        &\qquad - \int_0^t\intRp \mu\left(\frac{u_\x}{v} - \frac{\ubar_\x}{\vbar}\right)_\x \psi_{\x\x}\,d\x\,d\t + \int_0^T\intRp Q_1 \psi_{\x\x}\,d\x\,d\t\\
		&\quad =:\intRp \frac{|\psi_\x(0
        ,\x)|^2}{2} \,d\xi + I_1 + I_2 + I_3 + I_4.
	\end{aligned}
\end{equation*}
First, for the boundary terms, we use $\s_-<0$ and Young's inequality to have 
\[
\left. \left( \psi_t \psi_\x - \s_- \frac{|\psi_\x|^2}{2} \right)\right|_{\x = 0} \leq \left. \left(C|\psi_t|^2 -\frac{\s_-}{4}|\psi_\x|^2 + \frac{\s_-}{2}|\psi_\x|^2 \right)\right|_{\x = 0} \leq C |\psi_t|^2 \Big|_{\x = 0}.
\]
Since
\[
|\psi_t| \Big|_{\x=0} \leq |(u^C - u_m)_t| \Big|_{\x=0} + |(u^S - u^*)_t|\Big|_{\x=0},
\]
and 
\begin{align*}
&|(u^C - u_m)_t| \Big|_{\x=0} \leq \left.\Big|\s_- u_\x^C + \mu \left(\frac{u_\x^C}{v^C}\right)_\x + Q_1^C \Big|\right|_{\x=0} \leq C \d_C \left.(1+t)^{-1}e^{-\frac{C_1(\x + \s_- t)^2}{1+t}}\right|_{\x=0}\\
&\phantom{|(u^C - u_m)_t| \Big|_{\x=0}} \leq C\d_C e^{-Ct},\\
&|(u^S -u^*)_t|\Big|_{\x=0} \leq C|(u^S)'(-(\s - \s_-)t - X(t) - \b)||\dot{X} + (\s - \s_-)| \leq C\d_S^2e^{-C\d_S((\s - \s_-)t + \b)},
\end{align*}
we obtain
\begin{equation}\label{est:Hubd}
	C \int_0^t |\psi_t|^2 \Big|_{\x = 0}\,d\t \leq C\big(\d_C^2 + \d_S^3 e^{-C \d_S \b} \big).
\end{equation}
For $I_1$, we use Young's inequality to have 
\[
|I_1| \leq \int_0^t|\dot{X}|\intRp |u_\x^S||\psi_{\x\x}|\,d\x\,d\t\leq C\d_S^2\int_0^t|\dot{X}|^2 \,d\t+ \frac{1}{100}\int_0^t \mathcal{D}_{u_2}\,d\t.
\]
To estimate $I_2$, note that 
\begin{align*}
	|(p-\pbar)_\x| = R\left|\th \left(\frac{1}{v} - \frac{1}{\vbar}\right)_\x +\frac{1}{\vbar}(\th - \thbar)_\x\right| \leq C\big(|\phi_\x| + |\vbar_\x||\phi| + |\chi_\x| \big).
\end{align*}
Using Young's inequality, we obtain 
\begin{align*}
	|I_2| =\left|\int_0^t\intRp (p - \pbar)_\x \psi_{\x\x}\,d\x\,d\t\right| \leq &C\int_0^t(D_{v_1} + D_{\th_1})\,d\t + C\int_0^T\intRp |\vbar_\x|^2|\phi|^2\,d\x\,d\t \\
    &+ \frac{1}{100}\int_0^T \mathcal{D}_{u_2}\,d\t.
\end{align*}
Then, we use Lemma~\ref{lem:BL}, Lemma~\ref{lem:rarefaction1}, Lemma~\ref{lem:CD}, and Lemma~\ref{lem:VS} to have 
\[
\begin{aligned}
\int_0^t\intRp |\vbar_\x|^2|\phi|^2\,d\x\,d\t &\leq C\int_0^t\big(\d_{BL}G^{BL} + \d_R \e_r G^{R} + \d_S^2 G^S \big)\,d\t + C\d_C \int_0^t B^C\,d\t.
\end{aligned}
\]
Thus, we have 
\begin{align*}
   |I_2| \leq &C\int_0^t\big(\d_{BL}G^{BL} + \d_R \e_r G^{R} + \d_S^2 G^S +D_{v_1} + D_{\th_1})\,d\t + \frac{1}{100}\int_0^t \mathcal{D}_{u_2}\,d\t + C\d_C \int_0^t B^C\,d\t. 
\end{align*}
For $I_3$, observe that 
\begin{align*}
	I_3 &= -\mu \int_0^t\intRp \frac{1}{v}|\psi_{\x\x}|^2\,d\x\,d\t - \mu \int_0^t\intRp \left(\frac{1}{v}\right)_\x \psi_\x \psi_{\x\x}\,d\x\,d\t\\
    &\quad - \mu\int_0^t  \intRp \ubar_{\x\x}\left(\frac{1}{v} - \frac{1}{\vbar}\right)\psi_{\x\x}\,d\x\,d\t - \mu \int_0^t\intRp \ubar_\x \left(\frac{1}{v} - \frac{1}{\vbar}\right)_\x \psi_{\x\x}\,d\x\,d\t\\
	&\leq -\int_0^t \mathcal{D}_{u_2}\,d\t + C\int_0^t\intRp \big(|\phi_\x|+|\vbar_\x|\big)|\psi_\x||\psi_{\x\x}|\,d\x\,d\t + C\int_0^t\intRp |\ubar_{\x\x}||\phi||\psi_{\x\x}|\,d\x\,d\t \\
	&\quad + C\int_0^t\intRp |\ubar_\x|\big(|\phi_\x| + |\phi||\vbar_\x|)|\psi_{\x\x}|\,d\x\,d\t =: -\int_0^t \mathcal{D}_{u_2}\,d\t + I_{31} + I_{32} + I_{33}.
\end{align*}
Using the interpolation inequality, we have 
\begin{align*}
	|I_{31}| &\leq \int_0^t \|\phi_\x\|_{L^2}\|\psi_{\x}\|_{L^\infty}\|\psi_{\x\x}\|_{L^2}\,d\t + \int_0^t\|\vbar_\x\|_{L^\infty}\|\psi_{\x}\|_{L^2}\|\psi_{\x\x}\|_{L^2}\,d\t\\
	&\leq C\e \int_0^t\|\psi_{\x}\|_{L^2}^{1/2}\|\psi_{\x\x}\|_{L^2}^{3/2}\,d\t + C(\d_R\e_r + \d_0)\int_0^t\|\psi_{\x}\|_{L^2}\|\psi_{\x\x}\|_{L^2}\,d\t\\
	&\leq C(\e + \d_R\e_r + \d_0)\int_0^t\big(D_{u_1} + D_{u_2}\big)\,d\t.
\end{align*}
Since $|u^{BL}_{\x\x}| \leq C|u^{BL}_\x|$, $|u^R_{\x\x}| \leq C|u^R_\x|$ and $|u^{S}_{\x \x}| \leq C|u^S_{\x}|$, we obtain
\begin{align*}
	|I_{32}| &\leq C\int_0^t\intRp |\ubar_{\x\x}||\phi||\psi_{\x\x}|\,d\x\,d\t  \\
    & \leq \frac{1}{100}\int_0^t \mathcal{D}_{u_2}\,d\t + C\int_0^t(G^{BL} + G^R + G^S) \,d\t+ C \int_0^t\intRp |u^C_{\x\x}|^2 |\phi|^2\,d\x\,d\t\\
	&\leq \frac{1}{100}\int_0^t \mathcal{D}_{u_2}\,d\t + C\int_0^t(G^{BL} + G^R + G^S)\,d\t + C\d_C \int_0^t B^C\,d\t.
\end{align*}
Next, using Young's inequality, we have 
\begin{align*}
	|I_{33}| &\leq C\int_0^t\intRp |\ubar_\x|\big(|\phi_\x| + |\phi||\vbar_\x|)|\psi_{\x\x}|\,d\x\,d\t \\
	&\leq \frac{1}{100}\int_0^t \mathcal{D}_{u_2}\,d\t + C(\d_0 + \d_R\e_r)\int_0^t D_{v_1}\,d\t + C\int_0^t (G^{BL} + G^R + G^S +\d_C B^C)\,d\t
\end{align*}
Thus, using $\d_R\e_r <1$, we obtain 
\begin{equation}\label{est:hI3}
	\begin{aligned}
	I_3 &\leq -\frac{1}{2}\int_0^t \mathcal{D}_{u_2}\,d\t + C\int_0^t(G^{BL} + G^R + G^S)\,d\t + C(\e + \d_R\e_r + \d_0)\int_0^t(D_{v_1} + D_{u_1})\,d\t\\
	&\quad + C\int_0^t B^C\,d\t.
\end{aligned}
\end{equation}
Finally, again we use Lemma \ref{lem:WInteraction} to have
\begin{align*}
    |I_4| =\left|\int_0^t \intRp Q_1 \psi_{\x\x}\,d\x\,d\t \right| \leq &C \int_0^t\intRp |Q_1|^2\,d\x\,d\t + \frac{1}{100}\int_0^t \mathcal{D}_{u_2}\,d\t.
\end{align*}
Combining all these estimates, we obtain Lemma \ref{lem:hu}.
\end{proof}
\subsection{Higher Order Estimates on $\th-\thbar$}
Finally, we present the higher order estimates on $\chi$.
\begin{lem}\label{lem:hth}
Under the assumptions of Proposition~\ref{prop:ap}, there exist positive constants $c$ and $C$ such that
\begin{align*}
	&\sup_{t \in [0,T]}\|\chi_\x(t,\cdot)\|_{L^2}^2 + \int_0^T D_{\th_2}\,dt \\
	&\leq C\|\chi_\x(0,\cdot)\|_{L^2}^2 + C\int_0^T (\d_S |\dot{X}|^2 + G^{BL} + G^R + G^S + D_{v_1} + D_{u_1} + D_{\th_1} + D_{u_2})\,dt \\
	&\qquad +C\int_0^T |Q_2|^2 \,dt + Ce^{-c\d_S \b}.
\end{align*}
\end{lem}
Since the proof of this lemma is similar to that of Lemma \ref{lem:hu}, we provide a sketch of proof.
\begin{proof}
Throughout the proof of Lemma~\ref{lem:hth}, denote 
\[
\mathcal{D}_{\th_2} := \kappa \intRp \frac{1}{v}|\chi_{\x\x}|^2\,d\x \sim D_{\th_2}.
\]
Multiplying \eqref{eq:NSFp} by $-\chi_{\x\x}$ and integrating over $[0,t]\times \Rp$, we obtain 
\begin{align*}
	&\frac{R}{\g - 1}\intRp \frac{|\chi_\x|^2}{2}\,d\x + \frac{R}{\g -1}\int_0^t \left. \left(\chi_t \chi_\x - \s_- \frac{|\chi_\x|^2}{2} \right)\right|_{\x = 0}\,d\tau\\
	&\quad =\frac{R}{\g - 1}\intRp \frac{|\chi_\x(0,\x)|^2}{2}\,d\x-\frac{R}{\g -1}\int_0^t\dot{X}\intRp \th_\x^S \chi_{\x\x}\,d\x\,d\tau + \int_0^t\intRp (pu_\x - \pbar \ubar_\x)\chi_{\x\x}\,d\x\,d\tau \\
    &\qquad -\kappa \int_0^t\intRp \left(\frac{\th_\x}{v} - \frac{\thbar_\x}{\vbar}\right)_\x\chi_{\x\x}\,d\x\,d\tau  - \mu \int_0^t\intRp \left(\frac{u_\x^2}{v} - \frac{\ubar_\x^2}{\vbar}\right)\chi_{\x\x}\,d\x\,d\tau \\
    &\qquad + \int_0^t\intRp Q_2 \chi_{\x\x}\,d\x\,d\tau =:\frac{R}{\g - 1}\intRp \frac{|\chi_\x(0,\x)|^2}{2}\,d\x+ \sum_{i=1}^5 J_i. 
\end{align*}
First, for the boundary terms, by following the same argument as in \eqref{est:Hubd} together with 
\[
(\th^C - \th_m)_t \Big|_{\x = 0} \leq \left. \left(\s_- \th_\x + C \left(\frac{\th^C_\x}{\th^C}\right)_\x \right) \right|_{\x=0} \leq C\d_C\left. (1+t)^{-1/2}e^{-\frac{C_1(\x + \s_-t)^2}{1+t}}\right|_{\x=0} \leq C\d_C e^{-Ct},
\]
we obtain 
\[
\int_0^t \left. \left(\chi_t \chi_\x - \s_- \frac{|\chi_\x|^2}{2} \right)\right|_{\x = 0} \,d\tau \leq C\int_0^t |\chi_t|^2 \Big|_{\x=0}\,d\tau  \leq C(\d_C^2 + \d_S^3 e^{-C\d_S \b}).
\]
For $J_1$, we use Young's inequality to have 
\[
|J_1| \leq C\d_S^2 \int_0^t|\dot{X}|^2\,d\tau + \frac{1}{100}\int_0^t \mathcal{D}_{\th_2}\,d\tau.
\]
To estimate $J_2$, note that 
\[
|pu_\x - \pbar \ubar_\x| \leq |p\psi_\x| + |\ubar_\x (p - \pbar)| \leq C|\psi_\x| + C|\ubar_\x|(|\phi| + |\chi|). 
\]
Therefore, we have
\begin{align*}
	&|J_2| = \left|\int_0^t\intRp (pu_\x - \pbar \ubar_\x)\chi_{\x\x}\,d\x\,d\tau\right|\\
	&\, \, \leq \frac{1}{100}\int_0^t \mathcal{D}_{\th_2}d\tau + C\int_0^t  D_{u_1}d\tau + C(\d_0 + \e_r\d_R)\int_0^t(G^{BL} + G^R + G^S) d\tau  + C\d_C \int_0^t B^Cd\t.
\end{align*}
Now, for $J_3$, observe that 
\begin{align*}
	J_3 &= -\kappa \int_0^t \intRp \frac{1}{v}|\chi_{\x\x}|^2\,d\x \,d\t - \kappa \int_0^t \intRp \left(\frac{1}{v}\right)_\x  \chi_\x \chi_{\x\x}\,d\x\,d\t\\
    &\quad \quad - \kappa \int_0^t \intRp \thbar_{\x\x}\left(\frac{1}{v} - \frac{1}{\vbar}\right)\chi_{\x\x}\,d\x\,d\t  - \kappa \int_0^t \intRp \thbar_\x \left(\frac{1}{v} - \frac{1}{\vbar}\right)\chi_{\x\x}\,d\x\,d\t.
\end{align*}
In a similar way to the estimate of $I_3$ in \eqref{est:hI3}, we derive
\begin{align*}
	J_3 &\leq -\frac{1}{2}\int_0^t \mathcal{D}_{\th_2}\,d\tau + C\int_0^t (G^{BL} + G^R + G^S)\,d\tau + C(\e + \d_R \e_r + \d_0)\int_0^t(D_v + D_{\th_1})\,d\tau  \\
	&\qquad  +C\int_0^t B^C\,d\t.
\end{align*}
To estimate $J_4$, note that 
\begin{align*}
	\frac{u_\x^2}{v} - \frac{\ubar_\x^2}{\vbar} = \frac{1}{v}(u_\x^2 - \ubar_\x^2) + \ubar_\x^2\left(\frac{1}{v} -\frac{1}{\vbar}\right) \leq C(|\psi_\x|^2 + |\psi_\x||\ubar_\x|) + C\ubar_\x^2 |\phi|.
\end{align*}
This yields 
\begin{align*}
	|J_4| &\leq C \int_0^t\intRp (|\psi_\x|^2 + |\psi_\x||\ubar_\x|)|\chi_{\x\x}|\,d\x\,d\tau + C\int_0^t\intRp |\ubar_\x|^2 |\phi||\chi_{\x\x}|\,d\x\,d\tau\\
	&\leq \frac{1}{100}\int_0^t \mathcal{D}_{\th_2}\,d\tau +  C \int_0^t\intRp |\psi_\x|^2|\chi_{\x\x}|\,d\x\,d\tau +C\int_0^t D_{u_1}\,d\t\\ 
    &\qquad +C(\d_0 + \e_r \d_R)\int_0^t (G^{BL} + G^R + G^S)\,d\tau +C\d_C^3 \int_0^t B^C\,d\t. 
\end{align*}
On the other hand, using the interpolation inequality and Young's inequality, we find that 
\begin{align*}
\int_0^t\intRp |\psi_\x|^2|\chi_{\x\x}|\,d\x\,d\tau & \leq \int_0^t \|\psi_\x\|_{L^\infty} \|\psi_\x\|_{L^2} \|\chi_{\x\x}\|_{L^2}\,d\tau \leq \int_0^t\|\psi_\x\|_{L^2}^{3/2}\|\psi_{\x\x}\|_{L^2}^{1/2} \|\chi_{\x\x}\|_{L^2}\,d\tau\\
&
\leq C\e^4 \int_0^t D_{u_1}\,d\tau  + \int_0^t D_{u_2}\,d\tau + \frac{1}{100}\int_0^t \mathcal{D}_{\th_2}\,d\tau.
\end{align*}
Thus, we obtain 
\begin{align*}
	J_4 &\leq \frac{1}{50}\int_0^t \mathcal{D}_{\th_2}\,d\tau +  C(\d_0 + \e_r \d_R)\int_0^t (G^{BL} + G^R + G^S)\,d\tau+C \int_0^t D_{u_1}\,d\tau + C\int_0^t D_{{u_2}}\,d\tau \\
	&\qquad + C\d_C^3 \int_0^t B^C\,d\t. 
\end{align*}
Finally, applying Lemma \ref{lem:WInteraction}, we have  
\begin{align*}
   |J_5| \leq \int_{0}^{t}\intRp |Q_2| |\chi_{\x\x}|\,d\x\,d\tau \leq &C\int_{0}^{t}\intRp |Q_2|^2 \,d\x\,d\tau+ \frac{1}{100}\int_{0}^{t}\mathcal{D}_{\th_2}\,d\tau.
\end{align*}

By combining all these estimates, we obtain the desired estimates in Lemma~\ref{lem:hth}.
\end{proof}
\begin{lem}\label{lem:QiL2}
    Under the assumptions of Proposition~\ref{prop:ap}, there exists a positive constant $C$ such that, for any $0<t<T$,
    \[
    \int_0^t \intRp  \big(|Q_1|^2 + |Q_2|^2\big)\,d\x\,d\t \leq C\d_0 +C(\d_R\e_r) + C\d_0 \left(\frac{\d_R}{\e_r}\right)^2.
    \]
\end{lem}
\begin{proof}
Recall from \eqref{dec:Qi} that 
\[
\sum_{i=1}^2 Q_i =\sum_{i=1}^2 (Q_i^I + Q_i^R + Q_i^C),
\]
where $Q_i^I, Q_i^R$, and $Q_i^C$ are defined in \eqref{def:QI}, \eqref{def:QR}, and \eqref{def:QC}.
    Using \eqref{est:Q1}, \eqref{est:Q2}, \eqref{sJiL2}, and \eqref{sJiL22}, we have 
    \begin{align*}
        \int_0^t \intRp  (|Q_1^I|^2 + |Q_2^I|^2)d\x d\t &\leq C\sum_{i=1}^6 \int_0^t \intRp |J_i|^2d\x d\t + C\int_0^t \intRp |(u_\x^C,u_{\x\x}^C)|^2|\vbar - v^C|^2d\x d\t \\
        &\quad +C\sum_{i \neq j} \int_0^t \intRp |(u_i)_\x|^2(|(v_j)_\x| + |(u_j)_\x|)^2\,d\x\,d\t\\
        &\leq C\d_0 + C(\d_R\e_r)^2 + C\d_0\left(\frac{\d_R}{\e_r}\right)^2.
    \end{align*}
    Next, from \eqref{est:QiR} and Lemma~\ref{lem:rarefaction}, we have
    \begin{align*}
        &\int_0^t \intRp \big(|Q_1^R|^2 + |Q_2^R|^2\big)\,d\x\, d\t \leq \int_0^t \big(\|u_{\x\x}^R\|_{L^2}^2 + \|u_{\x}^R\|_{L^4}^4\big)\,d\t\\
        &\quad \leq \int_0^t \big(\|u_{\x\x}^R\|_{L^2}^{1/2}\|u_{\x\x}^R\|_{L^2}^{3/2} + \|u_{\x}^R\|_{L^4}^2\|u_{\x}^R\|_{L^4}^2\big)\,d\t\\
        &\quad \leq C\d_R^2 \e_r((\d_R)^{1/2} + (\d_R)^{1/8})^{3/2}\int_0^t(1+\t)^{-21/16}\,d\t + C(\d_R^2\e_r^{3/2})\d_R^{1/2}\int_0^t(1+\t)^{-3/2}\,d\t\\
        &\quad \leq  C\d_R^2 \e_r((\d_R)^{1/2} + (\d_R)^{1/8})^{3/2} + C(\d_R^2\e_r^{3/2})\d_R^{1/2} \leq C\d_R^{5/2}\e_r.
    \end{align*}
    Here, we take $q=8$ in Lemma~\ref{lem:rarefaction}.
Finally, by \eqref{est:QC}, we have
    \begin{align*}
        \int_0^t \intRp \big(|Q_1^C|^2 + |Q_2^C|^2\big)\,d\t \leq C\d_C^2.
    \end{align*}
    Combining all these estimates and using the smallness of $\d_0$ and $\e_r$, we obtain the desired estimate in Lemma~\ref{lem:QiL2}.
\end{proof}
$\bullet$ Proof of Proposition~\ref{prop:ap}: Combining Lemma~\ref{lem:hv}--\ref{lem:QiL2}, and choosing $\nu>0$ sufficiently small in Lemma~\ref{lem:hv}, we obtain
\[
\begin{aligned}
&\sup_{t \in [0,T]}\|(\phi_\x,\psi_\x,\chi_\x)(t,\cdot)\|_{L^2}^2 + \int_0^T (D_{v_1} + D_{u_2} + D_{\th_2})\,dt\\
&\quad  \leq C\|(\psi, \phi_\x,\psi_\x,\chi_\x)(0,\cdot)\|_{L^2}^2 + C\int_0^T (\d_S |\dot{X}|^2 + G^{BL} + G^R + G^S + D_{u_1} + D_{\th_1})\,dt \\
&\quad \quad + C\int_0^T B^C\,dt + 
+ C\left(\d_0 + (\d_R\e_r) + \d_0\left(\frac{\d_R}{\e_r}\right)^2\right)+ Ce^{-C\d_S\b}.
\end{aligned}
\]
Applying Lemma~\ref{lem:L2CD} and then combining the above estimates with Lemma~\ref{lem:L2ap}, we obtain the desired estimate in Proposition~\ref{prop:ap} by choosing $\e_r:=\d_0^{1/4}$ as defined in \eqref{def:er}. \qed
\begin{appendix}
\section{Proof of Lemma~\ref{lem:L2CD}}\label{App:CD}
In this section, we provide a sketch of the proof of Lemma~\ref{lem:L2CD}. It suffices to establish the following estimate:
\begin{equation}\label{lem:HK}
    \begin{aligned}
&\int_{0}^{t} \frac{1}{(1+\tau)}\int_{\mathbb{R}_+} e^{-\frac{C_1|\xi+\sigma_-\tau|^2}{1+\tau}}\left(\left(\pbar \phi + \frac{R}{\gamma-1}\chi\right)^2 + \frac{(R\chi - \pbar \phi)^2}{2v} + \frac{\gamma\pbar}{2}\psi^2 \right) d\xi d\tau   \\
&\quad \le C \sup\limits_{\tau\in \left[0, T\right]}||(U-\overline{U})(\tau,\cdot)||^2_{L^2(\mathbb{R}_+)}+C\delta_S\int_{0}^{t}|\dot{X}(\tau)|^2d\tau+C\int_{0}^{t}(G^S+G^R+G^{BL})d\tau\\
&\qquad +C\int_{0}^{t}(D_{v_1}+D_{u_1}+D_{\theta_1}+D_{u_2} + D_{\th_2})d\tau+C.
\end{aligned}
\end{equation}
Indeed, there exists $\widetilde{C}>0$ such that 
\[
(\phi^2+ \psi^2+\chi^2) \leq \widetilde{C}\left(\left(\pbar \phi + \frac{R}{\gamma-1}\chi\right)^2 + \frac{(R\chi - \pbar \phi)^2}{2v} + \frac{\gamma\pbar}{2}\psi^2 \right).
\]
Combining this inequality with \eqref{lem:HK} directly yields the desired estimate in Lemma~\ref{lem:L2CD}. Thus, it remains to prove \eqref{lem:HK}. For simplicity, we introduce the following notation. Unless otherwise specified, all functions below are defined on $\Rp$, whereas $\mathcal{H}$ is defined on the whole line $\mathbb{R}$.
\begin{align*}
&P := \pbar \phi + \frac{R}{\gamma-1}\chi, && \widetilde{P} = R\chi - \pbar \phi, && \mathcal{H}(t,\xi):=\frac{1}{\sqrt{1+t}}e^{-\frac{b|\xi+\sigma_-t|^2}{1+t}},\\
&b:=\frac{C_1}{2}, && H_1(t,\xi):=\int_{-\infty}^{\xi}\mathcal{H}(t,y)\,dy, && H_2(t,\xi):=\int_{-\infty}^{\xi}|\mathcal{H}(t,y)|^2\,dy.
\end{align*}
Note that $H_1$ satisfies $(H_1)_t - \s_- (H_1)_\x = \frac{1}{4b}(H_1)_{\x\x}$, and 
\begin{equation} \label{properties H_2}
|H_1(t,\x)| \leq \int_{\mathbb{R}} |\mathcal{H}(t,\x)|\,d\x \leq \sqrt{\frac{\pi}{b}}, \quad |H_2(t, \xi)|\le \frac{C}{\sqrt{1+t}}, \quad |(H_2)_t-\sigma_-(H_2)_\xi|\le \frac{C}{(1+t)^{\frac{3}{2}}}.
\end{equation}
We follow the same calculations in \cite[Appendix D]{KVW-NSF} to have 
\[
\begin{aligned}
  &\frac{1}{4b}\int_0^t \intRp P^2\mathcal{H}^2\,d\x\,d\t = \intRp \frac{P(0,\x)^2 H_1(0,\x)^2}{2}\,d\x - \intRp \frac{P^2H_1^2}{2}\,d\x \\
  &\quad + \int_0^t\intRp  \dot{X} 
  \Big(\pbar v_\x^S + \frac{R\th_\x^S}{\gamma-1}\Big)PH_1^2\,d\x\,d\t -\frac{1}{4b}\int_0^t\intRp (P^2)_\x H_1 \mathcal{H}\,d\x\,d\t\\
  &\quad - \int_0^t \intRp (p-\pbar)\psi_\x P H_1^2\,d\x\,d\t \\
  &\quad + \int_0^t \intRp \left((\pbar_t - \s_-\pbar_\x)\phi - (p-\pbar)\ubar_\x + \mu\left(\frac{u_\x^2}{v} - \frac{\ubar_\x^2}{\vbar}\right) - Q_2 \right)PH_1^2\,d\x\,d\t \\
  &\quad +\int_0^t \intRp \kappa\left(\frac{\th_\x}{v} - \frac{\thbar_\x}{\vbar} \right)_\x PH_1^2\,d\x\,d\t - \int_0^t \left.\left( \s_- \frac{P^2H_1^2}{2} - \frac{1}{4b}P^2H_1\mathcal{H}\right)\right|_{\x = 0} =: \sum_{i=1}^8 Z_i.
\end{aligned}
\]
Since the estimates of $Z_1,\cdots,Z_4$ are essentially the same as those in \cite{KVW-NSF}, we omit the details and focus on the remaining terms. First, we decompose $Z_5$ as 
\[
\begin{aligned}
    Z_5 &= \int_0^t \intRp \left(\frac{\gamma \pbar}{2v}\phi^2 PH_1^2 - \frac{(\gamma-1)}{v}\phi P^2 H_1^2 \right)_t - \s_-\left(\frac{\gamma \pbar}{2v}\phi^2 PH_1^2 - \frac{(\gamma-1)}{v}\phi P^2 H_1^2 \right)_\x \,d\x\,d\t\\
    &- \int_0^t \intRp \left(\frac{\gamma \pbar}{v}\phi^2 P - \frac{2(\gamma-1)}{v}\phi P^2\right)H_1[(H_1)_t - \s_- (H_1)_\x]\,d\x\,d\t\\
    &- \int_0^t \intRp \left((\rd_t - \s_- \rd_\x)\frac{\gamma \pbar}{v}P\right)\frac{\phi^2 H_1^2}{2}\,d\x\,d\t - \int_0^t \intRp \left((\rd_t - \s_- \rd_\x)\frac{(\gamma-1)}{v}P^2\right)\phi H_1^2\,d\x\,d\t\\
    &-\int_0^t \intRp \dot{X}\frac{(\pbar \phi - R\chi)}{v}v_\x^S PH_1^2\,d\x\,d\t =: \sum_{i=1}^5 Z_{5,i}.
\end{aligned}
\]
Since the estimates of $Z_{5,2}, Z_{5,5}$ are essentially same as in \cite{KVW-NSF}, we omit the details and focus on the remaining terms. For $Z_{5,1}$, we have 
\[
|Z_{5,1}| \leq  C\sup_{\t \in [0,T]} \|(\phi,\chi)\|_{L^2}^2 + \int_0^t |(\phi,\chi)|^2\,d\t \leq C\sup_{\t \in [0,T]} \|(\phi,\chi)\|_{L^2}^2 + C(\d_C^2 + e^{-C\d_S\b}).
\]
We further decompose $Z_{5,3}$ as
\[
\begin{aligned}
    Z_{5,3} &= -\int_0^t \intRp \dot{X}\left(v_\x^S \pbar + \frac{R\th_\x^S}{\gamma -1}\right)\frac{\gamma \pbar \phi^2 H_1^2}{2v}d\x d\t + \int_0^t \intRp \kappa \left(\frac{\th_\x}{v} - \frac{\thbar_\x}{\vbar}\right)\left(\frac{\gamma \pbar \phi^2 H_1^2}{2v} \right)_\x d\x d\t\\
    &\quad -\int_0^t \intRp \left(-(p-\pbar)\psi_\x - (p-\pbar)\ubar_\x + \mu\left(\frac{u^2_\x}{v} - \frac{\ubar_\x^2}{\vbar}\right) -Q_2\right)\frac{\gamma \pbar \phi^2 H_1^2}{2v}\,d\x\,d\t\\
    &\quad - \int_0^t \intRp \left((\pbar_t - \s_- \pbar_\x)\frac{\gamma \pbar \phi^3 H_1^2}{2v} + \left(\left(\frac{\pbar}{v}\right)_t -\s_- \left(\frac{\pbar}{v}\right)_\x \right)\frac{\gamma P \phi^2 H_1^2}{2} \right)\,d\x\,d\t\\
    &\quad + \int_0^t \kappa \left(\frac{\th_\x}{v} - \frac{\thbar_\x}{\vbar}\right)\left(\frac{\gamma \pbar \phi^2 H_1^2}{2v} \right)\Big|_{\x=0}\,d\t.
\end{aligned}
\]
Except for the term containing $Q_2$ and the boundary term, the remaining terms can be treated similarly to those in \cite{KVW-NSF}. Hence, we only present the estimates for these two terms and omit the details of the others. Following the arguments in estimating $B^I_1$ and $B^I_2$ in Lemma~\ref{lem:BIU}, and using \eqref{def:er}, we obtain
\[
\int_0^t \intRp Q_2 \frac{\gamma \pbar \phi^2 H_1^2}{2v}\,d\x\,d\t \leq C\e \int_0^t \intRp Q_2 |\phi|\,d\x\,d\t \leq C\int_0^t D_{v_1}\,d\t + C.
\]
Moreover, similar to the estimate of $J_2^{bd}$ and $J_3^{bd}$ in Lemma~\ref{lem:bd}, we have
\begin{equation}\label{est:CDbd2}
\begin{aligned}
   &\int_0^t \kappa \left.\left(\frac{\th_\x}{v} - \frac{\thbar_\x}{\vbar}\right)\left(\frac{\gamma \pbar \phi^2 H_1^2}{2v} \right)\right|_{\x=0}\,d\t \leq C\int_0^t \left|\frac{\th_\x}{v} - \frac{\thbar_\x}{\vbar}\right||\phi|\Big|_{\x=0}\,d\t \\
&\qquad \leq C(\d_C + \d_S^{1/3}e^{-C\d_S \b}) + C\e^2 \int_0^t D_{\th_2}\,d\t.
\end{aligned}
\end{equation}
Since $Z_{5,4}$ and $Z_6$ can be estimated similarly to $Z_{5,3}$, we omit the details. To estimate $Z_7$, integration by parts yields
\[
Z_7 = -\kappa \int_0^t \left.\left(\frac{\th_\x}{v} - \frac{\thbar_\x}{\vbar}\right)PH_1^2 \right|_{\x=0} \,d\t - \kappa \int_0^t \intRp \left(\frac{\th_\x}{v} - \frac{\thbar_\x}{\vbar}\right)(PH_1^2)_\x\,d\x\,d\t =: Z_{7,1} + Z_{7,2}. 
\]
The term $Z_{7,1}$ can be estimated similarly to \eqref{est:CDbd2}, so we omit the details. Next, following the arguments in \cite{KVW-NSF}, for any sufficiently small $k_1>0$, we obtain
\[
\begin{aligned}
    |Z_{7,2}| &\leq (k_1 + C\d_R\e_r + C\d_0)\int_0^t \intRp |\mathcal{H}|^2|(\phi,\chi)|^2\,d\x\,d\t + C_{k_1}\int_0^t \intRp D_{\th_1}\,d\t\\
    &\quad +C\int_0^t \intRp (G^{BL} + G^R + G^S + D_{v_1})\,d\x\,d\t.
\end{aligned}
\]
Finally, $Z_8$ can be estimated similarly to $Z_{5,1}$, and hence we omit the details. Therefore, we obtain the desired estimate for $\int_0^t \intRp P^2\mathcal{H}^2\,d\x\,d\t$. It remains to estimate the terms involving $\widetilde{P}$ and $\psi$. For this purpose, we use the following identity (see \cite{KVW-NSF} for the details):
\begin{align*}
\begin{aligned}
&\int_{0}^{t}\int_{\mathbb{R}_+} \frac{\mathcal{H}^2}{2}\left(\frac{|\widetilde{P}|^2}{v}+\gamma\bar{p}\psi^2\right)d\xi d\tau = \int_{\mathbb{R}_+}\psi H_2\widetilde{P}\Bigg|_{\tau=0}^{\tau=t}d\xi -\int_{0}^{t}\int_{\mathbb{R}_+}\frac{H_2|\widetilde{P}|^2}{2v^2}v_\xi d\xi d\tau\\
&-\int_{0}^{t}\int_{\mathbb{R}_+}\psi\widetilde{P}\left[(H_2)_t-\sigma_-(H_2)_\xi\right]d\xi d\tau-\frac{\gamma}{2}\int_{0}^{t}\int_{\mathbb{R}_+}\bar{p}_\xi \psi^2H_2d\xi d\tau\\
&+\int_{0}^{t}\int_{\mathbb{R}_+}\dot{X}(t)H_2\left(\bar{p}\psi v^S_\xi-R(\gamma-1)\psi\theta^S_\xi-\widetilde{P}u^S_\xi\right)d\xi d\tau+\int_{0}^{t}\int_{\mathbb{R}_+}\psi H_2\phi(\bar{p}_t-\sigma_-\bar{p}_\xi)d\xi d\tau\\
&+(\gamma-1)\int_{0}^{t}\int_{\mathbb{R}_+}\psi H_2\left[(p-\pbar)\psi_\x + (p-\pbar)\ubar_\x-\kappa\left(\frac{\theta_\xi}{v}-\frac{\bar{\theta}_\xi}{\bar{v}}\right)_\xi-\mu\left(\frac{u_\xi^2}{v}-\frac{\bar{u}_\xi^2}{\bar{v}}\right)+Q_2\right]d\xi d\tau\\
&-\int_{0}^{t}\int_{\mathbb{R}_+}H_2\widetilde{P}\left[\mu\left(\frac{u_\xi}{v}-\frac{\bar{u}_\xi}{\bar{v}}\right)_\xi-Q_1\right]d\xi d\tau -\int_{0}^{t}\left(\frac{H_2|\widetilde{P}|^2}{2v}-\s_-\psi H_2 \widetilde{P}+\frac{\gamma\psi^2H_2\bar{p}}{2}\right)\Bigg|_{\xi=0}d\tau\\
&=:\sum\limits_{i=1}^{9}\mathcal{Z}_i.
\end{aligned}    
\end{align*}
Since the estimates of $\mathcal{Z}_1,\cdots, \mathcal{Z}_9$, except for $\mathcal{Z}_4$ are similar to those for $Z_i$, we omit the details and only provide the estimate for $\mathcal{Z}_4$. We remark that, by estimating $\mathcal{Z}_8$ in the same manner as $Z_7$, one obtains the term $D_{u_2}$ in \eqref{lem:HK}. First, note that 
\[
\begin{aligned}
   |\mathcal{Z}_4| &\leq C\int_0^t \intRp (|v^{BL}_\x| + |v^R_\x| + |v^C_\x| + |v^S_\x|)\psi^2 H_2\,d\x\,d\t\\ 
   &\leq C\int_0^t \intRp (|v^{BL}_\x|+|v^R_\x|)\psi^2 H_2\,d\x\,d\t + C\int_0^t G^S\,d\t + C\d_C\int_0^t \intRp \mathcal{H}^2\psi^2\,d\x\,d\t.
\end{aligned}
\]
Therefore, it suffices to estimate the first term on the right-hand side of the above inequality. Using Young's inequality, \eqref{properties H_2}, \eqref{est:KPoin}, and Lemma~\ref{lem:BL}, we obtain
\[
\begin{aligned}
    &\int_0^t \intRp |v_\x^{BL}|\psi^2 H_2\,d\x\,d\t \leq C\int_0^t \intRp \psi^2\big(|v_\x^{BL}|^{5/3} + |H_2|^{5/2} \big)\,d\x\,d\t\\
    &\leq C \int_{0}^{t} \intRp   \left(|\psi(0,\x)|+\sqrt{\xi}||\psi_\xi||_{L^2}\right)^2 \left(\frac{\delta_{BL}}{1+\delta_{BL}\xi}\right)^{\frac{10}{3}} d\xi d\tau+C\int_{0}^{t}\frac{1}{(1+\tau)^{\frac{5}{4}}} \int_{\mathbb{R}_+} \psi^2 d\xi d\tau\\
    &\leq C \int_{0}^{t} (|\psi(0,\x)|^2 + \|\psi_\xi\|_{L^2}^2) \int_{\mathbb{R}_+} \frac{\delta_{BL}^{\frac{7}{3}}(1+\d_{BL}\xi)}{\left(1+\delta_{BL}\xi\right)^{\frac{10}{3}}} d\xi d\tau+ C\varepsilon^2 \\
    &\leq C\d_{BL}^{4/3}\left(\int_0^t D_{u_1}\,d\t + \d_C^2 + e^{-C\d_S\b} \right) + C\e^2.
\end{aligned}
\]
On the other hand, using \eqref{properties H_2}, H\"older's inequality, \eqref{est:apriori} and Lemma~\ref{lem:rarefaction1}, we obtain
\[
\begin{aligned}
    &\int_0^t \intRp |v_\x^R|\psi^2 H_2\,d\x\,d\t \leq \int_0^t  \frac{C}{\sqrt{1+\t}}\|v_\x^R\|_{L^4}\|\psi^2\|_{L^{4/3}}\,d\t\\
    &\leq \int_0^t  \frac{C}{\sqrt{1+\t}}\frac{\d_R^{1/4}}{(1+\t)^{3/4}}\left(\intRp |\psi|^{8/3}\,d\x\right)^{3/4}\,d\t \leq \int_0^t  \frac{C\d_R^{1/4}}{(1+\t)^{5/4}}\|\psi\|_{L^\infty}^{1/2}\|\psi\|_{L^2}^{3/2}\,d\t\\
    &\leq C\d_R^{1/4}\e^2.
\end{aligned}
\]
Therefore, combining the above estimates, we obtain the desired estimate for $\mathcal{Z}_4$. This completes the proof of \eqref{lem:HK}, and hence of Lemma~\ref{lem:L2CD}.

\section{Proof of Lemma~\ref{lem:estL12}}\label{App:wint}
In this section, we provide a proof of Lemma~\ref{lem:estL12}.

$\bullet$ Proof of \eqref{est:ucont}: First, observe that 
\[
|u_\x^C||\vbar - v^C|= |u_\x^C|\big(|v^{BL} - v_*|+|v^R - v_m| + |v^S - v^*|\big).
\]
Following the calculations in Lemma~\ref{lem:win}, and using Lemma~\ref{lem:BL}--Lemma~\ref{lem:CD}, we obtain 
\begin{align*}
    &\begin{aligned}
        	\intRp |u_\x^C||v^{BL} - v_*|\,d\x &= \left(\int_0^{-\frac{\s_-t}{2}}+ \int_{-\frac{\s_-t}{2}}^{\infty}\right)  |u_\x^C||v^{BL} - v_*|\,d\x \\
            &\leq C\d_{BL}\d_C \left(e^{-Ct} + \frac{\ln (1+\d_{BL}t)}{1+t}\right),
    \end{aligned}\\
    &\begin{aligned}
        &\intRp |u_\x^C||v^{R} - v_m|\,d\x \\
        &\quad  \leq \left(\int_{0\le\xi \leq \frac{1}{2}\lambda_1(u_m,\theta_m)(1+t)-\s_-t} + \int_{\xi \ge \frac{1}{2}\lambda_1(u_m,\theta_m)(1+t)-\s_-t}\right) |u_\x^C||v^{R} - v_m|\,d\x\\
	&\quad \leq   C\d_C \d_R e^{-Ct} + C\delta_C\frac{\delta_R}{\e_r}  e^{-C\e_r t}.
    \end{aligned}
\end{align*}
From H\"older's inequality, we find that
\[
\begin{aligned}
    \intRp |u_\x^C||v^{S} - v^*|\,d\x &\leq \left(\intRp |u_\x^C|\, d\x\right)^{3/4} \left(\intRp |u_\x^C||v^S - v^*|^4 \,d\x\right)^{1/4}\\
    &\leq C\d_C^{3/4}\left(\intRp |u_\x^C||v^S - v^*|^4 \,d\x\right)^{1/4}.
\end{aligned}
\]
Using Lemma~\ref{lem:VS}, we obtain 
\[
\intRp |u_\x^C||v^S - v^*|^4 \,d\x = \left(\int_0^{\alpha t} + \int_{\alpha t}^\infty\right)|u_\x^C||v^S - v^*|^4 \,d\x \leq C\d_C\d_S^3e^{-C\d_St} + C\d_C\d_S^4e^{-Ct},
\]
where $\alpha = \frac{\s - 2\s_-}{2}$. Therefore, we have
\[
\intRp |u_\x^C||v^S - v^*|\,d\x \leq C\d_C\big(\d_S^{3/4}e^{-C\d_St} +\d_Se^{-Ct} \big).
\]
Thus, we obtain 
\begin{equation}\label{est:uCint}
\begin{aligned}
&\intRp |u_\x^C||\vbar - v^C|\,d\x \\
&\quad \leq C\d_C\left[\delta_{BL} \left(e^{-Ct} + \frac{\ln (1+\d_{BL}t)}{(1+t)}\right) + (\delta_R + \delta_S)e^{-Ct} + \d_S^{3/4}e^{-C\d_S t} + \frac{\delta_R}{\e_r}  e^{-C\e_r t}  \right].
\end{aligned}
\end{equation}
Following the same argument in deriving \eqref{est:uCint}, we have 
\begin{align*}
&\intRp |u_{\x \x} ^C||\vbar - v^C|d\x\\
&\quad \leq C\d_C\left[\delta_{BL} \left(e^{-Ct} + \frac{\ln (1+\d_{BL}t)}{(1+t)}\right) + (\delta_R + \delta_S)e^{-Ct} + \d_S^{3/4}e^{-C\d_S t} + \frac{\delta_R}{\e_r}  e^{-C\e_r t}  \right].
\end{align*}
$\bullet$ Proof of \eqref{est:wder}: Let $b := \frac{\s - \s_-}{2}>0$ and $\alpha = \frac{\s - 2\s_-}{2}$. From Lemma \ref{lem:BL}--\ref{lem:VS} and \eqref{est:BLR}, we obtain the followings:
\[
\begin{aligned}
	\intRp |v_\x^{BL}||v_\x^R|\,d\x \leq & C\frac{(\d_R\e_r)^{1/8}}{(1+t)^{7/8}} \delta_{BL}^2\intRp \frac{1}{(1+\delta_{BL}\x)^2}\,d\x \leq C\delta_{BL}\frac{(\d_R\e_r)^{1/8}}{(1+t)^{7/8}},\\
	\intRp |v_\x^{BL}||v_\x^S|\,d\x
	 \leq & C{\delta_{BL}^2}\d_S^2 \left(\int_0^{bt} + \int_{bt}^{\infty}\right) \frac{e^{-C\d_S|\x - (\s - \s_-)t - X(t) - \beta|}}{(1+{\delta_{BL}}\x)^2}d\x  \\
	  \leq & C\delta_{BL} \delta_S\left(\delta_{BL} e^{-C\d_S t} + \frac{\d_S}{1+\delta_{BL}t}\right),	
\end{aligned}
\]
\begin{align*}
	\intRp |v_\x^R||v_\x^S|\,d\x &\leq  \left(\int_{0\le\xi \leq \frac{1}{2}\lambda_1(u_m,\theta_m)(1+t)-\s_-t} + \int_{\xi \geq \frac{1}{2}\lambda_1(u_m,\theta_m)(1+t)-\s_-t}\right) |v_\x^R||v_\x^S|\,d\x  \\
	& \leq C(\e_r\delta_R) \delta_S e^{-C\delta_St} + C\delta_S^2\frac{(\e_r\delta_R)}{\e_r}e^{-C\e_rt}.
\end{align*}
Moreover, by calculations similar to those in \eqref{est:Ji}, we obtain
\[
\begin{split}
&\begin{aligned}
	&\intRp |v_\x^{BL}|\big(|v_\x^C| + |u_\x^C|\big)\,d\x \leq \left(\int_0^{-\frac{\s_-t}{2}} + \int_{-\frac{\s_-t}{2}}^{\infty}\right)  |v_\x^{BL}|\big(|v_\x^C| + |u_\x^C|\big) \,d\x\\
	&\phantom{\intRp |v_\x^{BL}|\big(|v_\x^C| + |u_\x^C|\big)\,d\x} \leq C\delta_{BL}\d_C \left(e^{-Ct} + \frac{1}{1+\delta_{BL}t}\right),
\end{aligned}\\
&\begin{aligned}
	&\intRp |v_\x^R|\big(|v_\x^C| + |u_\x^C|\big)\,d\x\\
    &\quad \leq \left(\int_{0\le\xi \leq \frac{1}{2}\lambda_1(u_m,\theta_m)(1+t)-\s_-t} + \int_{\xi \geq \frac{1}{2}\lambda_1(u_m,\theta_m)(1+t)-\s_-t}\right)  |v_\x^R|\big(|v_\x^C| + |u_\x^C|\big)\,d\x\\
	&\quad \leq C\delta_C\left((\e_r\d_R) e^{-Ct} + \frac{(\e_r\d_R)}{\e_r} e^{-C\e_rt}\right),
\end{aligned} 
\end{split}
\]
\[
	\intRp |v_\x^S|\big(|v_\x^C| + |u_\x^C|\big)\,d\x \leq \left(\int_0^{\alpha t} + \int_{\alpha t}^{\infty}\right) |v_\x^S|\big(|v_\x^C| + |u_\x^C|\big) \,d\x \leq C \d_S \d_C \left(e^{-C\delta_St} + e^{-Ct}\right).
\]
$\bullet$ Proof of \eqref{est:432nd}: Combining \eqref{est:ucont} and \eqref{est:wder}, we obtain
\begin{align*}
    &\int_0^t \| |(u_\x^C,u_{\x\x}^C)||\vbar - v^C|\|_{L^1}^{4/3} + \sum_{i\neq j}\||(v_i)_\x|(|(v_j)_\x|+|(u_j)_\x|)\|_{L^1}^{4/3}\,d\t\\
&\quad \leq C\d_C^{4/3}\left(\d_{BL}^{4/3} +\d_R^{4/3}+\d_S^{4/3} + 1 + \left(\frac{\d_R}{\e_r}\right)^{4/3}\frac{1}{\e_r} \right)\\
&\qquad + C\d_{BL}^{4/3}(\d_R\e_r)^{1/6}+C\big(\d_{BL}^{4/3}\d_C^{4/3} + \d_{BL}^{1/3}\d_C^{4/3}\big) + C\big(\d_{BL}^{8/3}\d_S^{1/3} + \d_{BL}^{1/3}\d_S^{8/3} \big) \\
&\qquad +C\left(\d_C^{4/3}(\d_R\e_r)^{4/3}+ \d_C^{4/3}\left(\frac{(\e_r\d_R)}{\e_r}\right)^{4/3}\frac{1}{\e_r} \right) \\
&\qquad  + C\left((\d_R\e_r)^{4/3}\d_S^{1/3} + \d_S^{8/3}\left(\frac{(\e_r\d_R)}{\e_r}\right)^{4/3}\frac{1}{\e_r}\right)+C\d_S^{4/3}\d_C^{4/3}.
\end{align*}
Thanks to the smallness of $\d_0$ and $\e_r$, we obtain
\begin{align*}
    &\int_0^t \| |(u_\x^C,u_{\x\x}^C)||\vbar - v^C|\|_{L^1}^{4/3} + \sum_{i\neq j}\||(v_i)_\x|(|(v_j)_\x|+|(u_j)_\x|)\|_{L^1}^{4/3}\,d\t\\
&\qquad \leq C\d_0\big(1+(\d_R \e_r)^{1/6}\big) + C\d_0 \left(\frac{\d_R}{\e_r}\right)^{4/3}\frac{1}{\e_r}.
\end{align*}
\qed

\section{Proof of Global Existence}\label{APP:conti}
   In this section, we prove the global existence of \eqref{eq:NSF}. Assume that the constants $\d_0,\e,c,C_0$, and $\b$ from Proposition~\ref{prop:ap} are given. 
   First, we choose smooth functions $\hat{v}$, $\hat{u}$, and $\hat{\theta}$ defined on $\bbr_+$ such that
	\begin{align}\label{est:vhat}
    \begin{aligned}
	&\|(\hat{v}-v^*, \hat{u}-u^*, \hat{\theta}-\th^*)\|_{L^2(0, \beta)}+\|(\hat{v}-v_+, \hat{u}-u_+, \hat{\theta}-\theta_+)\|_{L^2( \beta, \infty)}\\
	&\quad  +\|(\hat{v}_\xi, \hat{u}_\xi, \hat{\theta}_\xi)\|_{L^2(\mathbb{R}_+)}\le \hat{C}_0\delta_0,
    \end{aligned}
	\end{align}
	for some constant $\hat{C}_0>0$. Define the smooth function $\underline{U} := (\underline{v}, \underline{u}, \underline{\theta})$ as
    \[
    \underline{U}(\x) := U^R(0,\x) + U^C(0,\x) + \hat{U}(\x) - U_m - U^*.
    \]
    Then we have
    \begin{equation}\label{est:bars}
    \begin{aligned}
        &\|\underline{v}(\cdot)-\bar{v}(0,\cdot)\|_{H^1(\mathbb{R}_+)} = \| (v^{BL}(\cdot)+v^{S}(\cdot - \b)-\hat{v}(\cdot)-v_\ast)\|_{H^1(\Rp)}\\
        &\quad \leq \|v^{BL} - v_*\|_{L^2(\Rp)} + \|
        \hat{v} - v^C(0,\cdot) \|_{L^2(0,\b)} + \|v^S(\cdot - \b) - v^*\|_{L^2(0,\b)}\\
        &\qquad + \|\hat{v} - v_+\|_{L^2(\b,\infty)} + \|v^S(\cdot - \b) - v_+\|_{L^2(\b,\infty)}+ \|(\hat{v}_\x,v^{BL}_\x)\|_{L^2(\Rp)} + \|v^S_\x\|_{L^2(\mathbb{R})} \\
        &\quad \leq \hat{C}_1\sqrt{\d_0},
    \end{aligned}
    \end{equation}
    for some constant $\hat{C}_1>0$. By the same argument, after retaking $\hat{C}_1$ if necessary, we obtain
    \[
    \|(\underline{v}-\bar{v},\underline{u} - \bar{u}, \underline{\th} - \bar{\th})(0,\cdot)\|_{H^1(\mathbb{R}_+)} \leq \hat{C}_1\sqrt{\d_0}.
    \]
For a given $\e_0>0$, we choose $\d_0,\d_S,$ sufficiently small and $\b>0$ sufficiently large such that 
\[
\frac{\e}{4(C_0+1)}-\hat{C}_0\d_0 - \hat{C}_1\sqrt{\d_0} - \d_0^{1/40}-e^{-c\d_S\b}>0
\]
holds. We now introduce two constants, $\e_0$ and $\e_*$, defined by 
\[
\e_0=\frac{\e}{4(C_0+1)},\qquad
\e_*:=\frac{\e}{2(C_0+1)}-\hat{C}_1\sqrt{\delta_0}-\d_0^{1/40} - e^{-c\delta_S\beta},
\]
where 
$\hat{C}_0$ is the constant in \eqref{est:vhat}. 
Now, consider any initial data $(v_0,u_0,\theta_0)$ satisfying \eqref{inismall}, that is,
\begin{equation}\label{app:small}
\begin{aligned}
   &\|U_0-U^* - (U^R(0,\cdot)-U_m)\|_{L^2(0,\beta)} +  \|U_0 - U_+ - (U^R(0,\cdot) - U_m)\|_{L^2(\beta,\infty)}\\
   &\qquad + \|(U_0 - U^R(0,\cdot))_\x\|_{L^2(\Rp)} < \e_0. 
\end{aligned}
\end{equation}
From \eqref{est:vhat} and \eqref{app:small}, we obtain
\begin{equation}\label{app:small2}
	\begin{aligned}
		&\|U_0-\underline{U}\|_{H^1(\R_+)}\\
		&\quad \le \|(U_0 - U^* - (U^R(0,\cdot) - U_m)\|_{L^2(0,\beta)}+ \|U^C(0,\cdot) - U^*\|_{L^2(0,\beta)}+ \|\hat{U} - U^*\|_{L^2(0,\beta)}\\
        &\qquad + \|U_0 - U_+ - (U^R(0,\cdot) - U_m)\|_{L^2(\beta,\infty)} +\|U^C(0,\cdot) - U^*\|_{L^2(\beta,\infty)}+ \|\hat{U} - U_+\|_{L^2(\beta,\infty)} \\
		&\quad \quad + \|(U_0 - U^R(0,\cdot))_\x\|_{L^2(\Rp)} + \|(\hat{U}_\x,U^C_\x)(0,\cdot)\|_{L^2(\Rp)}\le  \e_0+\hat{C}_0\delta_0 \leq  \e_*.
	\end{aligned}
    \end{equation}
	In particular, Sobolev embedding implies $\|(v_0-\underline{v},\th_0 - \underline{\th})\|_{L^\infty(\R_+)}<C\e_*$. Thus, by choosing $\e_*$ sufficiently small, we obtain
	\[\frac{\min\{v_-,v_+\}}{2}<v_0(x)<2\max\{v_-,v_+\}, \,\, \frac{\min\{\th_-,\th_+\}}{2} < \th_0(x) < 2\max\{\th_-,\th_+\}, \,\, \x \in \Rp.\]
	Hence, Proposition \ref{prop:local} implies that there exists $T_0>0$ such that the system \eqref{eq:NSFL} admits a unique solution $(v,u,\theta)$ on $[0,T_0]$ satisfying
	\begin{equation}\label{vvunder}
		\|v-\underline{v}\|_{L^\infty(0,T_0;H^1(\R_+))}+\|u-\underline{u}\|_{L^\infty(0,T_0;H^1(\R_+))} + \|\theta-\underline{\theta}\|_{L^\infty(0,T_0;H^1(\R_+))}\le \frac{\e}{2},
	\end{equation}
	and for any $t \in [0,T_0]$ and $\x \in \Rp$,
	\[\frac{\min\{v_-,v_+\}}{4}<v(t,x)<4\max\{v_-,v_+\},\,\, \frac{\min\{\th_-,\th_+\}}{4}<\th(t,x)<4\max\{\th_-,\th_+\}.\]
 On the other hand, observe that
 \begin{equation}\label{tsmall}
 \begin{aligned}
     &\|(\underbar{U} - \overline{U}
     )(t,\cdot)\|_{H^1(\mathbb{R}_+)}\\
     &\quad \leq \|U^{BL} - U_*\|_{L^2(\Rp)} + \|U^R(0,\cdot) - U^R(t,\cdot)\|_{L^2(\Rp)}+ \|U^C(0,\cdot) - U^C(t,\cdot)\|_{L^2(\Rp)}\\
     &\qquad  +\|\hat{U} - U^*\|_{L^2(0,\b)} + \|U^S(\cdot - (\s - \s_-)t - X(t) - \b) - U^*\|_{L^2(0,\b)}\\
     &\qquad +\|\hat{U} - U_+\|_{L^2(\b,\infty)}  + \|U^S(\cdot - (\s - \s_-)t - X(t) - \b) - U_+\|_{L^2(\b,\infty)}\\
     &\qquad +\|(\hat{U}_\x,U^{BL}_\x,U^R_\x(0,\cdot),U^R_\x(t,\cdot),U^C_\x(0,\cdot),U^C_\x(t,\cdot))\|_{L^2(\Rp)}+\|U^S_\x\|_{L^2(\mathbb{R})}\\
     &\quad \leq C\sqrt{t}(\sqrt{\ln(1+t)}+(1+t)^{1/4})+C\sqrt{\delta_0}(1+\sqrt{|X(t)|}+\sqrt{t})\\
     &\quad \le C\sqrt{t}(\sqrt{\ln(1+t)}+(1+t)^{1/4})+C\sqrt{\delta_0}(1+\sqrt{t}).
 \end{aligned}
 \end{equation}
    Here, the first term of the last inequality is obtained as follows: using H\"older's inequality, \eqref{eq:R}$_1$, and Lemma~\ref{lem:rarefaction}, we have
\begin{align*}
&\|v^R(t,\cdot)-v^R(0,\cdot)\|^2_{L^2(\mathbb{R}_+)}=\int_{0}^{\infty} \left|\int_{0}^{t}v^R_\tau(\tau,\xi)d\tau\right|^2 d\xi \leq  \int_{0}^{\infty}t\int_{0}^{t}|v^R_{\tau}(\tau,\xi)|^2d\tau d\xi\\
&\quad \le Ct \int_{0}^{t}\int_{0}^{\infty} \left(|v^R_{\xi}(\tau,\xi)|^2+|u^R_{\xi}(\tau,\xi)|^2\right) d\xi d\tau \le Ct \int_{0}^{t} \frac{1}{\tau+1}d\tau \leq Ct\ln(1+t),
\end{align*}
and using \eqref{eq:VCD}$_1$ and Lemma~\ref{lem:CD},
\[
\begin{aligned}
    &\|v^{C}(t,\cdot)-v^{C}(0,\cdot)\|^2_{L^2(\mathbb{R}_+)}=\int_{0}^{\infty} \left|\int_{0}^{t}v^{C}_\tau(\tau,\xi)d\tau\right|^2 d\xi\\
    &\quad \leq Ct \int_{0}^{t}\int_{0}^{\infty} \left(|v^C_{\xi}(\tau,\xi)|^2+|u^C_{\xi}(\tau,\xi)|^2\right) d\xi d\tau \leq Ct\int_0^t\frac{\d_C^2}{1+\t}\intRp e^{-\frac{C_1(\x + \s_- \t)^2}{1+\t}}\,d\x\,d\t\\
    &\quad \leq C\d_C^2t\sqrt{1+t}.
\end{aligned}
\]
Similarly, we can obtain the same estimates for $u^R$, $\theta^R$,$u^C$, and $\th^C$. For the details of the remaining terms in \eqref{tsmall}, we refer to \cite{KVW23}.

	Therefore, if we choose $\delta_0$ and $T_1\in(0,T_0)$ small enough so that $C\sqrt{T_1}(\sqrt{\ln(1+T_1)}+(1+T_1)^{1/4})+C\sqrt{\delta}(1+\sqrt{T_1})<\frac{\e}{2}$, we have
	\[\|(\underline{v}-\bar{v}, \underline{u}-\bar{u}, \underline{\theta}-\bar{\theta})(t,\cdot)\|_{L^\infty(0,T_1;H^1(\R_+))}\le \frac{\e}{2}.\]
	Combining with the estimate \eqref{vvunder}, we obtain
	\[\|(v-\bar{v},u-\bar{u}, \theta-\bar{\theta})(t,\cdot)\|_{L^\infty(0,T_1;H^1(\R_+))}\le \e.\]
	Moreover, since it holds that 
	\((v-\underline{v},u-\underline{u}, \theta-\underline{\theta})\in C([0,T_1];H^1(\R_+)),\)
	 and the shift $X(t)$ is Lipschitz continuous, we obtain $(v-\bar{v},u-\bar{u}, \theta-\bar{\theta})\in C([0,T_1];H^1(\R_+))$.
    
    We now show that the solution can be globally extended by using the standard continuation argument. To this end, we define the maximal existence time
	\[T_M:=\sup\left\{T>0~\Bigg|~\sup_{t\in[0,T]}\|(v-\bar{v}, u-\bar{u}, \theta-\bar{\theta})(t,\cdot)\|_{H^1(\R_+)}\le \e\right\}.\]
    Suppose, for the contradiction, that $T_M$ is finite. Then, by the definition of $T_M$, we have
	\[\sup_{t\in[0,T_M]}\|(v-\bar{v}, u-\bar{u}, \theta-\bar{\theta})(t,\cdot)\|_{H^1(\R_+)}=\e.\]
	On the other hand, using \eqref{est:bars} and \eqref{app:small2}, we find that 
	\begin{align*}
	\|(v_0-\bar{v}(0,\cdot), u_0-\bar{u}(0,\cdot), \theta_0-\bar{\theta}(0,\cdot)) \|_{H^1(\R_+)}
	<\e_*+\hat{C}_1\sqrt{\delta_0} = \frac{\e}{2(C_0+1)}-\d_0^{1/40} - e^{-c\delta_S\beta}.
	\end{align*}
	Then, Proposition~\ref{prop:ap} implies that
	\[\sup_{t\in[0,T_M]}\left(\|v-\bar{v}\|_{H^1(\R_+)}+\|u-\bar{u}\|_{H^1(\R_+)}\right)
    \leq \frac{\e}{2},\]
	which is a contradiction. Therefore, we deduce $T_M=+\infty$, and the \textit{a priori} estimate \eqref{estsuppertur} holds for the whole time interval $[0,\infty)$.
\end{appendix}

\vspace{0.3cm}
\noindent \textbf{Statements and Declarations}\\

\noindent \textbf{Data availability:} No datasets were generated or analyzed during the current study. \vspace{0.3cm}

\noindent \textbf{Conflict of interest:} The authors declare that they have no conflicts of interest with this work.

\bibliography{reference}

\begin{thebibliography}{10}

\bibitem{AMA97}
{\sc Amadori, D.}
\newblock Initial-boundary value problems for nonlinear systems of conservation laws.
\newblock {\em NoDEA Nonlinear Differential Equations Appl. 4}, 1 (1997), 1--42.

\bibitem{AC97}
{\sc Amadori, D., and Colombo, R.~M.}
\newblock Continuous dependence for {$2\times 2$} conservation laws with boundary.
\newblock {\em J. Differential Equations 138}, 2 (1997), 229--266.

\bibitem{AMS24}
{\sc Ancona, F., Marson, A., and Spinolo, L.~V.}
\newblock Existence of vanishing physical viscosity solutions of characteristic initial-boundary value problems for systems of conservation laws.
\newblock {\em arXiv preprint arXiv:2401.14865\/} (2024).

\bibitem{BRN79}
{\sc Bardos, C., le~Roux, A.~Y., and N\'ed\'elec, J.-C.}
\newblock First order quasilinear equations with boundary conditions.
\newblock {\em Comm. Partial Differential Equations 4}, 9 (1979), 1017--1034.

\bibitem{BS09}
{\sc Bianchini, S., and Spinolo, L.~V.}
\newblock The boundary {R}iemann solver coming from the real vanishing viscosity approximation.
\newblock {\em Arch. Ration. Mech. Anal. 191}, 1 (2009), 1--96.

\bibitem{BS20}
{\sc Bianchini, S., and Spinolo, L.~V.}
\newblock Characteristic boundary layers for mixed hyperbolic-parabolic systems in one space dimension and applications to the {N}avier-{S}tokes and {MHD} equations.
\newblock {\em Comm. Pure Appl. Math. 73}, 10 (2020), 2180--2247.

\bibitem{DM07}
{\sc Donadello, C., and Marson, A.}
\newblock Stability of front tracking solutions to the initial and boundary value problem for systems of conservation laws.
\newblock {\em NoDEA Nonlinear Differential Equations Appl. 14}, 5-6 (2007), 569--592.

\bibitem{EEK25}
{\sc Eo, S., Eun, N., and Kang, M.-J.}
\newblock Stability of a {R}iemann shock in a physical class: From {B}renner-{N}avier-{S}tokes-{F}ourier to {E}uler.
\newblock {\em arXiv preprint arXiv:2411.03613\/} (2024).

\bibitem{FLWZ14}
{\sc Fan, L., Liu, H., Wang, T., and Zhao, H.}
\newblock Inflow problem for the one-dimensional compressible {N}avier-{S}tokes equations under large initial perturbation.
\newblock {\em J. Differential Equations 257}, 10 (2014), 3521--3553.

\bibitem{HKK23}
{\sc Han, S., Kang, M.-J., and Kim, J.}
\newblock Large-time behavior of composite waves of viscous shocks for the barotropic {N}avier-{S}tokes equations.
\newblock {\em SIAM J. Math. Anal. 55}, 5 (2023), 5526--5574.

\bibitem{HKKKO25}
{\sc Han, S., Kang, M.-J., Kim, J., Kim, N., and Oh, H.}
\newblock Convergence to superposition of boundary layer, rarefaction and shock for the inflow problem of the 1{D} {N}avier-{S}tokes equations.
\newblock {\em Comm. Math. Phys. 407}, 4 (2026), Paper No. 79, 59.

\bibitem{HKK25}
{\sc Han, S., Kang, M.-J., Kim, J., and Lee, H.}
\newblock Long-time behavior towards viscous-dispersive shock for {N}avier-{S}tokes equations of {K}orteweg type.
\newblock {\em J. Differential Equations 426\/} (2025), 317--387.

\bibitem{HM21}
{\sc Hashimoto, I., and Matsumura, A.}
\newblock Existence of radially symmetric stationary solutions for the compressible {N}avier-{S}tokes equation.
\newblock {\em Methods Appl. Anal. 28}, 3 (2021), 299--311.

\bibitem{HLM10}
{\sc Huang, F., Li, J., and Matsumura, A.}
\newblock Asymptotic stability of combination of viscous contact wave with rarefaction waves for one-dimensional compressible {N}avier-{S}tokes system.
\newblock {\em Arch. Ration. Mech. Anal. 197}, 1 (2010), 89--116.

\bibitem{HLS10}
{\sc Huang, F., Li, J., and Shi, X.}
\newblock Asymptotic behavior of solutions to the full compressible {N}avier-{S}tokes equations in the half space.
\newblock {\em Commun. Math. Sci. 8}, 3 (2010), 639--654.

\bibitem{HMS2003}
{\sc Huang, F., Matsumura, A., and Shi, X.}
\newblock A gas-solid free boundary problem for a compressible viscous gas.
\newblock {\em SIAM J. Math. Anal. 34}, 6 (2003), 1331--1355.

\bibitem{HMS03}
{\sc Huang, F., Matsumura, A., and Shi, X.}
\newblock Viscous shock wave and boundary layer solution to an inflow problem for compressible viscous gas.
\newblock {\em Comm. Math. Phys. 239}, 1-2 (2003), 261--285.

\bibitem{HXY08}
{\sc Huang, F., Xin, Z., and Yang, T.}
\newblock Contact discontinuity with general perturbations for gas motions.
\newblock {\em Adv. Math. 219}, 4 (2008), 1246--1297.

\bibitem{HKKL25}
{\sc Huang, X., Kang, M.-J., Kim, J., and Lee, H.}
\newblock Asymptotic behavior toward viscous shock for impermeable wall and inflow problems of barotropic {N}avier-{S}tokes equations.
\newblock {\em J. Math. Anal. Appl. 552}, 2 (2025), Paper No. 129803, 50.

\bibitem{HL25}
{\sc Huang, X., and Lee, H.}
\newblock Time-asymptotic stability of composite wave of viscous shocks and viscous contact wave for {N}avier-{S}tokes-{F}ourier equations.
\newblock {\em arXiv preprint arXiv:2504.04014\/} (2025).

\bibitem{KK06}
{\sc Kagei, Y., and Kawashima, S.}
\newblock Stability of planar stationary solutions to the compressible {N}avier-{S}tokes equation on the half space.
\newblock {\em Comm. Math. Phys. 266}, 2 (2006), 401--430.

\bibitem{Kang19}
{\sc Kang, M.-J.}
\newblock {$L^2$}-type contraction for shocks of scalar viscous conservation laws with strictly convex flux.
\newblock {\em J. Math. Pures Appl. (9) 145\/} (2021), 1--43.

\bibitem{KOW}
{\sc Kang, M.-J., Oh, H., and Wang, Y.}
\newblock Asymptotic behavior toward viscous shock for the outflow problem of barotropic {N}avier--{S}tokes equations.
\newblock {\em Nonlinearity 39}, 7 (2026), Paper No. 075006.

\bibitem{Kang-V-1}
{\sc Kang, M.-J., and Vasseur, A.~F.}
\newblock {$L^2$}-contraction for shock waves of scalar viscous conservation laws.
\newblock {\em Ann. Inst. H. Poincar\'{e} C Anal. Non Lin\'{e}aire 34}, 1 (2017), 139--156.

\bibitem{KV21}
{\sc Kang, M.-J., and Vasseur, A.~F.}
\newblock Contraction property for large perturbations of shocks of the barotropic {N}avier-{S}tokes system.
\newblock {\em J. Eur. Math. Soc. 23}, 2 (2021), 585--638.

\bibitem{KV-Inven}
{\sc Kang, M.-J., and Vasseur, A.~F.}
\newblock Uniqueness and stability of entropy shocks to the isentropic {E}uler system in a class of inviscid limits from a large family of {N}avier-{S}tokes systems.
\newblock {\em Invent. Math. 224}, 1 (2021), 55--146.

\bibitem{KV-2shock}
{\sc Kang, M.-J., and Vasseur, A.~F.}
\newblock Well-posedness of the {R}iemann problem with two shocks for the isentropic {E}uler system in a class of vanishing physical viscosity limits.
\newblock {\em J. Differential Equations 338\/} (2022), 128--226.

\bibitem{KVW23}
{\sc Kang, M.-J., Vasseur, A.~F., and Wang, Y.}
\newblock Time-asymptotic stability of composite waves of viscous shock and rarefaction for barotropic {N}avier-{S}tokes equations.
\newblock {\em Adv. Math. 419\/} (2023), Paper No. 108963, 66.

\bibitem{KVW-NSF}
{\sc Kang, M.-J., Vasseur, A.~F., and Wang, Y.}
\newblock Time-asymptotic stability of generic {R}iemann solutions for compressible {N}avier-{S}tokes-{F}ourier equations.
\newblock {\em Arch. Ration. Mech. Anal. 249}, 4 (2025), Paper No. 42, 80.

\bibitem{KawaOutBL}
{\sc Kawashima, S., Nishibata, S., and Zhu, P.}
\newblock Asymptotic stability of the stationary solution to the compressible {N}avier-{S}tokes equations in the half space.
\newblock {\em Comm. Math. Phys. 240}, 3 (2003), 483--500.

\bibitem{KZ08}
{\sc Kawashima, S., and Zhu, P.}
\newblock Asymptotic stability of nonlinear wave for the compressible {N}avier-{S}tokes equations in the half space.
\newblock {\em J. Differential Equations 244}, 12 (2008), 3151--3179.

\bibitem{Kim26a}
{\sc Kim, J.}
\newblock Stationary solutions to the spherically symmetric compressible fluid with capillarity effect.
\newblock {\em arXiv preprint arXiv:2605.07610\/} (2026).

\bibitem{Kim26}
{\sc Kim, J.}
\newblock Time-asymptotic {S}tability of the {S}tationary {S}olution to the {I}mpermeable {W}all {P}roblem for the radially {S}ymmetric {N}avier-{S}tokes-{K}orteweg {E}quations.
\newblock {\em arXiv preprint arXiv:2608.17617\/} (2026).

\bibitem{K1}
{\sc Kru{\v{z}}kov, S.~N.}
\newblock First order quasilinear equations with several independent variables.
\newblock {\em Mat. Sb. (N.S.) 81 (123)\/} (1970), 228--255.

\bibitem{MatBVP}
{\sc Matsumura, A.}
\newblock Inflow and outflow problems in the half space for a one-dimensional isentropic model system of compressible viscous gas.
\newblock {\em Methods and Applications of Analysis 8}, 4 (2001), 645--666.

\bibitem{MN01}
{\sc Matsumura, A., and Nishihara, K.}
\newblock Large-time behaviors of solutions to an inflow problem in the half space for a one-dimensional system of compressible viscous gas.
\newblock {\em Comm. Math. Phys. 222}, 3 (2001), 449--474.

\bibitem{Otto96}
{\sc Otto, F.}
\newblock Initial-boundary value problem for a scalar conservation law.
\newblock {\em C. R. Acad. Sci. Paris S\'er. I Math. 322}, 8 (1996), 729--734.

\bibitem{QW09}
{\sc Qin, X., and Wang, Y.}
\newblock Stability of wave patterns to the inflow problem of full compressible {N}avier-{S}tokes equations.
\newblock {\em SIAM J. Math. Anal. 41}, 5 (2009), 2057--2087.

\bibitem{QW11}
{\sc Qin, X., and Wang, Y.}
\newblock Large-time behavior of solutions to the inflow problem of full compressible {N}avier-{S}tokes equations.
\newblock {\em SIAM J. Math. Anal. 43}, 1 (2011), 341--366.

\bibitem{Sol76}
{\sc Solonnikov, V.~A.}
\newblock The solvability of the initial-boundary value problem for the equations of motion of a viscous compressible fluid.
\newblock {\em Zap. Nau\v cn. Sem. Leningrad. Otdel. Mat. Inst. Steklov. (LOMI) 56\/} (1976), 128--142, 197.
\newblock Investigations on linear operators and theory of functions, VI.

\bibitem{SZK21}
{\sc Suzuki, M., and Zhang, K.~Z.}
\newblock Stationary flows for compressible viscous fluid in a perturbed half-space.
\newblock {\em Comm. Math. Phys. 388}, 3 (2021), 1131--1180.

\bibitem{UNK10}
{\sc Ueda, Y., Nakamura, T., and Kawashima, S.}
\newblock Stability of degenerate stationary waves for viscous gases.
\newblock {\em Arch. Ration. Mech. Anal. 198}, 3 (2010), 735--762.

\bibitem{Vasseur01}
{\sc Vasseur, A.}
\newblock Strong traces for solutions of multidimensional scalar conservation laws.
\newblock {\em Arch. Ration. Mech. Anal. 160}, 3 (2001), 181--193.

\bibitem{WYY2025}
{\sc Wang, Y., Yang, Y., and Yu, Q.}
\newblock Existence of large boundary layer solutions to inflow problem of 1{D} full compressible {N}avier-{S}tokes equations.
\newblock {\em Acta Math. Sci. Ser. B (Engl. Ed.) 45}, 6 (2025), 2591--2606.

\end{thebibliography}
\end{document}